\documentclass[12pt]{article}
\usepackage{amsmath}

\usepackage{amsthm}

\usepackage{appendix}

\usepackage{amssymb}
\usepackage{youngtab}
\usepackage{enumerate}
\usepackage{bbm}
\newtheorem{theorem}{Theorem}[section]
\newtheorem{proposition}[theorem]{Proposition}
\newtheorem{lemma}[theorem]{Lemma}
\newtheorem{corollary}[theorem]{Corollary}

\newcommand{\E}{\mathbb{E}}
\newcommand{\C}{\mathcal{C}}

\theoremstyle{definition}

\newtheorem{remark}[theorem]{Remark}

\newtheorem{definition}[theorem]{Definition}

\newtheorem{example}[theorem]{Example}

\usepackage{hyperref}
\usepackage{url}
\usepackage{float}

\usepackage{wasysym}
\usepackage[usenames,dvipsnames]{color}
\usepackage{tikz}
\usetikzlibrary{decorations.markings}
\usetikzlibrary{arrows}
\usetikzlibrary{arrows.meta}
\usetikzlibrary{shapes.geometric}
\usetikzlibrary{patterns}
\tikzstyle{empty}=[circle,draw=black!80,thick]
\tikzstyle{emptyn}=[circle,draw=black!80,fill=white,scale=0.5] 
\tikzstyle{nero}=[circle,draw=black!80,fill=black!80,thick] 
\usetikzlibrary{positioning, fit, calc}

\usepackage{caption,subcaption}
\usepackage{bbm}
\usepackage{geometry}
\usepackage[font=footnotesize]{caption}

\def\e{\epsilon}
\def\a{\alpha}
\def\b{\beta}
\def\g{\gamma}
\def\d{\delta}

\def\Z{\mathbb{Z}}
\def\E{\mathbb{E}}

\def\C{\mathbb{C}}

\def\tr{\mathrm{tr}}
\def\op{\mathrm{op}}

\def\so3{\mathrm{SO}(3)}
\def\fso3{\mathrm{FSO}(3)}
\def\bbA{\mathbbm 1_A}
\def\bbAA{\mathbbm 1_{A^{-1}}}
\def\hf{{\hat f}}
\def\hg{{\hat g}}
\title{Partial associativity and rough approximate groups}
\date{}
\author{W. T. Gowers and J. Long}
\begin{document}
\maketitle

\abstract{Suppose that a binary operation $\circ$ on a finite set $X$ is injective in each variable separately and also associative. It is easy to prove that $(X,\circ)$ must be a group. In this paper we examine what happens if one knows only that a positive proportion of the triples $(x,y,z)\in X^3$ satisfy the equation $x\circ(y\circ z)=(x\circ y)\circ z$. Other results in additive combinatorics would lead one to expect that there must be an underlying `group-like' structure that is responsible for the large number of associative triples. We prove that this is indeed the case: there must be a proportional-sized subset of the multiplication table that approximately agrees with part of the multiplication table of a metric group. A recent result of Green shows that this metric approximation is necessary: it is not always possible to obtain a proportional-sized subset that agrees with part of the multiplication table of a group.

\section{Introduction}\label{intro}

The following statement is a known result in additive combinatorics. Let $n$ be a prime, let $A\subset\Z/n\Z$ and let $\phi:A\to\Z/n\Z$ be a map such that the number of quadruples $(a,b,c,d)\in A^4$ with $a+b=c+d$ and $\phi(a)+\phi(b)=\phi(c)+\phi(d)$ is at least $\a n^3$. Then there is a subset $A'\subset A$ of size at least $\b n$, where $\b$ depends on $\a$ only, such that $\phi(a)+\phi(b)=\phi(c)+\phi(d)$ whenever $a,b,c,d\in A'$ and $a+b=c+d$. A map with this last property is called a \emph{Freiman homomorphism}, so this result is saying that a map that obeys the condition for a Freiman homomorphism a constant fraction of the time can be restricted to a dense set that obeys the condition all the time. One can then go further and show that $\phi$ agrees on a further dense subset with the restriction of a `linear-like' function, which gives a global structural characterization of functions that satisfy the initial local conditions. (For more details, see for example~\cite{gowersszem}, in particular Corollary 7.6.)

There are several known statements of this general flavour, and the purpose of this paper is to prove another one. Here our starting point is a binary operation $\circ$ defined on a finite set $X$. We assume that it is {an injection (and therefore a bijection)} in each variable separately, and that there exists a constant $c>0$, independent of the size of $X$, such that the number of triples $(x,y,z)\in X^3$ with $x\circ(y\circ z)=(x\circ y)\circ z$ is at least $c|X|^3$. It is easy to see that if $c=1$ then these conditions are equivalent to the group axioms, so it is natural to ask whether if a binary operation has this property for some smaller $c$, then there must be some underlying group structure that `explains' the prevalence of associative triples. This question appears to have been known to various people -- we heard about it from Emmanuel Breuillard, who attributed it to Ehud Hrushovski, and an essentially equivalent question arose out of work we ourselves had been doing -- but it does not seem to have appeared in print.

The `99\% case' was dealt with by Elad Levi \cite{EL}, who proved that if $c$ is close to 1, then there must be a group $G$ of size approximately equal to $|X|$ and an injection $\phi:X\to G$ such that $\phi(x\circ y)=\phi(x)\phi(y)$ for almost all pairs $x,y\in X^2$. In other words, the multiplication table of $\circ$ agrees almost everywhere with a group operation. In this paper we look at the `1\% case' -- that is, the case where $c$ is a small fixed constant. We also weaken the hypothesis in a small way by considering binary operations that are only partially defined: this has no significant effect on our arguments, but it is convenient when discussing examples not to have to worry about whether they are defined everywhere. In the discussion that follows, we shall often use the word `operation' to mean `partial binary operation'. 

An easy way to create an operation with many associative triples is to take the operation $\circ$ on a group $G$ and turn it into a partial binary operation by restricting it to a dense subset $X\subset G^2$. This is not guaranteed to work, as there are not necessarily $c|G|^3$ triples $(x,y,z)\in G^3$ such that all of $(x,y), (y,z), (x,y\circ z)$ and $(x\circ y,z)$ belong to $X$. However, in many cases, such as when $X$ is a random subset, it does. More generally, given any operation with many associative triples, one can find restrictions that still have many associative triples. 

Another method is to take a subset $A$ of a group $G$ and restrict the group operation $\circ$ to all pairs $(a,b)\in A^2$ such that $a\circ b\in A$. Again, this is not guaranteed to work, but if $A$ is an \emph{approximate subgroup}, which roughly speaking means that it is closed `1\% of the time' (we shall discuss this condition in more detail in a moment), then this gives another source of examples.

A third method is based on structures that are approximately groups in a metric sense. For concreteness, we discuss a specific example. 

\begin{example}\label{fuzzyso3}
Let $\d>0$ and let $X$ be a maximal $\d$-separated subset of $\so3$ {equipped with a suitable translation invariant metric}. Now define a partial binary operation as follows. Let $\theta>0$ be a suitable absolute constant (as opposed to $\d$, which is comparable to $|X|^{-1/3}$) and then for $x,y,z\in X$ let $x\circ y=z$ if and only if $d(xy,z)\leq\theta\d$. If $d(xy,X)>\theta\delta$ then the product $x\circ y$ is undefined. {Importantly, translation invariance means that (assuming that $\theta<1/2$) the resulting operation is injective in each variable separately, in the sense that e.g. $x\circ y=x'\circ y$ with both sides defined implies that $x=x'$.}
\end{example} 

We show in an appendix that no matter how the set $X$ in Example~\ref{fuzzyso3} is chosen {the resulting binary operation will be defined a positive proportion of the time, and} there will necessarily be many associative triples. (In order to prove this, we first prove a Bogolyubov-type lemma for $\so3$ (Lemma \ref{bogolyubov} in the second appendix) that may be of independent interest. However, there is no obvious way of passing to a subset of $X^2$ where the operation is isomorphic to a restriction of a group operation. Indeed, in an earlier version of this paper, we conjectured that there was no such subset, and that conjecture has been proved by Ben Green \cite{greenso3}. 

This example shows that a natural conjecture -- that a partially associative binary operation agrees on a substantial set of pairs with a group operation -- is false. However, the example has a suggestive structure that suggests an appropriate weakening of the conjecture.

By a metric group, we mean a group $G$ equipped with a bi-invariant metric $d$, which we allow to take the value infinity. Our main result will be that if an operation has many associative triples (and is injective in each variable separately), then it agrees on a large set of pairs with a restriction of a small perturbation of the binary operation on a metric group. 

The next theorem is not in fact our main theorem, but a consequence of it. However, to state the main theorem requires one more definition, so we shall state this result first. Loosely speaking, it says that the multiplication table of a partial binary operation with many associative triples must be approximately isomorphic to part of the multiplication table of a metric group $G$. The precise statement is as follows.

\begin{theorem} \label{assoc_main} Let $c>0$, let $X$ be a finite set and let $\circ$ be a partially defined binary operation on $X$ that is injective in each variable separately. Let $E\subset X^2$ be the set of pairs on which $\circ$ is defined. Suppose that there are at least $\e|X|^3$ triples $(x,y,z)\in X^3$ such that $x\circ(y\circ z)=(x\circ y)\circ z$ (where this means in particular that all expressions and subexpressions are defined). Then for every positive integer $b$ there exist $\d(\e,b)\ge \e^{b^{26b}}$, a subset $A\subset E$ of density at least $\d$, a metric group $G$, and maps $\phi,\psi$ and $\omega$ from $X$ to $G$, such that the images $\phi(X), \psi(X)$ and $\omega(X)$ are 1-separated, and $d(\phi(x)\psi(y),\omega(z))\leq b^{-1}$ for every $(x,y,z)\in X^3$ such that $(x,y)\in A$ and $x\circ y=z$.
\end{theorem}

\subsection{Quasigroups, the quadrangle condition, torsors, and our main theorem.}\label{intro2}

Our main result will have the same conclusion as that of Theorem \ref{assoc_main} but a hypothesis that is both weaker and in some ways more natural. It arises out of the following simple question: suppose that an $n\times n$ grid is filled with labels. Under what conditions is this labelled grid the multiplication table of some group?

We can ask the question more formally as follows. Suppose we are given three sets $X$, $Y$ and $Z$ with $|X|=|Y|=n$, and a function $f:X\times Y\to Z$. Under what conditions does there exist a group $G$ of order $n$ and bijections $\phi:X\to G, \psi:Y\to G$ and $\omega:Z\to G$ such that for every $(x,y)\in X\times Y$ we have $\phi(x)\psi(y)=\omega(f(x,y))$?

To discuss this, we use the following vocabulary. We call the elements of $Z$ \emph{labels}, sets of the form $\{x\}\times Y$ \emph{columns} and sets of the form $X\times\{y\}$ \emph{rows}. If $f(x,y)=z$, we say that $z$ is the label in position $(x,y)$. A very obvious necessary condition is that $Z$ should also have cardinality $n$. Another is that each label occurs exactly once in each row and each column. 

A labelling of an $n\times n$ grid that satisfies these two conditions is known as the \emph{Latin square} $(X,Y,Z,f)$. If we think of the labelled grid as the multiplication table of the binary operation $f$, then it has the property that for each $x\in X$ the function $y\mapsto f(x,y)$ is a bijection from $Y$ to $Z$, and for each $y\in Y$ the function $x\mapsto f(x,y)$ is a bijection from $X$ to $Z$. If we identify the sets $X, Y$ and $Z$ (using arbitrary bijections) and write $x\circ y$ instead of $f(x,y)$, then we have a set $X$ with a binary operation $\circ$ with the property that for every $a,b\in X$ the equations $a\circ x=b$ and $x\circ a=b$ have unique solutions. Such an algebraic structure is called a \emph{quasigroup}. (Thus, quasigroups and Latin squares are essentially the same.)

The question now becomes the following: when is a quasigroup a group? Equivalently, when is a Latin square the multiplication table of a group?\footnote{{It is important to clarify exactly what this question is asking. When we are presented with the Latin square, we are \emph{not} given any correspondences between rows, columns and labels. Rather, we are given an arrangement of labels and asked to \emph{find} correspondences in such a way that the resulting binary operation is a group operation.}}

The following definition, due to Brandt~\cite{Brandt}, turns out to be a natural concept when determining whether a quasigroup manifests group structure.

\begin{definition}\label{qc1}
	Let $A=(X,Y,Z,f)$ be a Latin square. We say that $A$ satisfies the \emph{quadrangle condition} if for every configuration of the following form in $A$,
	\[\begin{matrix}c&d\\ a&b\\ &&&c&d'\\ &&&a&b\\ \end{matrix}\]
	we have $d=d'$. To put it a different way, we can define a ternary \emph{rectangle completion} operation on the set $Z$ of labels by mapping $(a,b,c)$ to $d$ whenever there exists a rectangle with labels $a,b,c,d$ such that $a$ is in the same row as $b$ and the same column as $c$.
\end{definition}

Observe that if $x_1,x_2,y_1,y_2$ are elements of a group $G$, and $x_1y_1=a, x_2y_1=b$ and $x_1y_2=c$, then $x_2y_2=ba^{-1}c$. This simple observation shows that if a Latin square is a group multiplication table, then the rectangle completion operation is well-defined, and therefore group multiplication tables satisfy the quadrangle condition.

It turns out that the converse is true as well: a Latin square that satisfies the quadrangle condition is the multiplication table of a group. This is a well-known observation of Brandt \cite{Brandt}. Since the proof is short, we give it here.

\begin{proposition}\label{prop:QC}
Every Latin square that satisfies the quadrangle condition is the multiplication table of a group.
\end{proposition}

\begin{proof}
Choose an arbitrary row $R$ and column $C$ and define a binary operation $\circ$ on the set of labels as follows. Given labels $a$ and $b$, find where $a$ appears in row $R$ and where $b$ appears in column $C$, and then let $a\circ b=c$, where $c$ is the label of the point in the same column as $a$ and the same row as $b$. The label of the point where $R$ and $C$ intersect is then an identity for $\circ$, and the Latin square condition implies that every element has both a left and a right inverse. It remains to check associativity. To do this, consider the following picture, which is of a portion of the Latin square, chosen to demonstrate that $(a\circ b)\circ c=a\circ(b\circ c)$. We write $d$ for $a\circ b$, $f$ for $d\circ c$, $g$ for $b\circ c$, $h$ for $a\circ g$, and $e$ for the identity.
\[\begin{matrix}g&&h\\ c&&&&g&&f\\ b&&d\\ e&& a&&b&&d\\ \end{matrix}\]
For associativity we need $f$ to equal $h$. But this follows from the quadrangle condition, since included in the above diagram are the points
\[\begin{matrix}g&&h\\ &&&&&g&&f\\ b&&d\\ &&& &&b&&d\\ \end{matrix}\]
Thus, the set of labels has an associative binary operation with an identity such that every element has a left and a right inverse, and we are done.
\end{proof}

A notable feature of the above argument is the arbitrary choice of the row $R$ and the column $C$, and hence the arbitrary choice of which label would serve as the identity element. It shows that if we are presented just with the labelled grid and not with any correspondences between rows, columns and labels, then there is no way of telling which label corresponds to the identity. Another way of expressing this observation is to say that if $G$ is a group and $x$ is any element of $G$, then we can form a group $G_x$ with identity element $x$ by taking the binary operation $a\circ b=ax^{-1}b$, which is derived from the rectangle-completion operation discussed above.

If one wishes to avoid the artificiality of choosing an arbitrary element to be the identity, one can do so by working with an algebraic structure known as a \emph{torsor}, which can be thought of as a group `except that we do not know which element is the identity'. The formal definition of a torsor is that it is a set $X$ with a ternary operation $\tau$ which has the following two properties.
\begin{itemize}
\item $\tau(x,x,y)=\tau(y,x,x)=y$ for every $x,y\in X$;
\item $\tau(x,y,\tau(z,u,v))=\tau(\tau(x,y,z),u,v)$ for every $x,y,z,u,v\in X$.
\end{itemize} 
{The quantity $\tau(x,y,z)$ should be thought of as $xy^{-1}z$. Indeed,} the relationship between groups and torsors is closely analogous to the relationship between vector spaces and affine spaces, and the ternary map is also closely analogous to the (partially defined) map $(a,b,c)\mapsto a-b+c$ that often appears in additive combinatorics when one has a set $A$ with additive structure that is not `centred on zero'.

In order to draw out the relationship between the quadrangle condition and our problem, in which we are given a partially defined binary operation, we must turn our attention to \emph{partial} Latin squares -- that is, to grids that are partially labelled in such a way that no label occurs more than once in any row or column. The formal definition is given below.

\begin{definition}\label{PLS}
	A \emph{partial Latin square} is a quintuple $(X,Y,Z,A,\phi)$, where $X,Y,Z$ are finite sets, $A\subset X\times Y$, and $\phi:A\to Z$ is a function such that if $\phi(a,b_1)=\phi(a,b_2)$ then $b_1=b_2$, and if $\phi(a_1,b)=\phi(a_2,b)$, then $a_1=a_2$. 
\end{definition}	
	
	We shall refer to the elements of $X, Y$ and $Z$ as \emph{columns}, \emph{rows} and  \emph{labels}, respectively, since we imagine drawing the object $(X,Y,Z,A,\phi)$ as a labelled grid with columns indexed by $X$ and rows indexed by $Y$, labelling the point $(x,y)\in  A$ with $\phi(x,y)$.

Given a partial Latin square $(X,Y,Z,A,\phi)$ with $|X|=|Y|=|Z|=n$, we shall sometimes abuse notation and say that $A$ is an $n\times n$ partial Latin square (or simply that $A$ is a partial Latin square). If $(X,Y,Z,A,\phi)$ is a partial Latin square and $B\subset A$, we may also refer to the partial Latin square $(X,Y,Z,B,\phi|_B)$ as $B$, calling it simply a subset of $A$ (if it is clear from context that both objects are partial Latin squares).

The above definition does not treat the sets $X,Y$ and $Z$ in a symmetric way. For some of our arguments, which themselves give a special role to the label set, that is appropriate. However, later we shall need to make statements that are symmetric in the three sets, so we note here that a partial Latin square is nothing other than a linear tripartite 3-uniform hypergraph. (Recall that a 3-uniform hypergraph is a set of triples, and it is said to be linear if no pair is contained in more than one triple.) We briefly prove that now.

\begin{lemma}\label{PLSequalsH} 
There is a one-to-one correspondence between partial Latin squares $(X,Y,Z,A,\phi)$ and linear tripartite 3-uniform hypergraphs with vertex sets $X, Y$ and $Z$.
\end{lemma}

\begin{proof}
Given a partial Latin square $P=(X,Y,Z,A,\phi)$, let $H$ consist of all triples $(x,y,z)\in X\times Y\times Z$ such that $(x,y)\in A$ and $\phi(x,y)=z$. Since $\phi$ is a function, for any $(x,y)$ there is at most one $z$ such that $(x,y,z)\in H$, and since $\phi$ is injective in each variable separately, for each $(x,z)$ there is at most one $y$ such that $(x,y,z)\in H$ and for each $(y,z)$ there is at most one $x$ such that $(x,y,z)\in H$. Therefore, $H$ is linear.

Conversely, given a linear tripartite 3-uniform hypergraph $H$ with vertex sets $X,Y,Z$, let $A$ be the set of all $(x,y)$ such that there exists $z$ with $(x,y,z)\in H$. Since $H$ is linear, such a $z$ is unique if it exists, so we can define a function $\phi:A\to Z$ by setting $\phi(x,y)$ to be the unique $z$ such that $(x,y,z)\in H$. If $\phi(x,y_1)=\phi(x,y_2)$, then $(x,y_1,z)$ and $(x,y_2,z)$ both belong to $H$, so by the linearity property $y_1=y_2$. Therefore, $\phi$ is injective in the second variable. Similarly, it is injective in the first variable.
\end{proof}

With the above observations and definitions in mind, it is natural to formulate a torsor version of the question about binary operations with many associative triples. For reasons that we shall explain in the next subsection, we call a pair of identically labelled rectangles in a partial Latin square an octahedron. The following is a precise definition.

\begin{definition}\label{cuboctahedron}
Given a partial Latin square $(X,Y,Z,A,\phi)$, an \emph{octahedron} in $A$ consists of a pair of rectangles $$R_1=((x_1,y_1),(x_2,y_1),(x_1,y_2),(x_2,y_2))\in A^4$$
and
$$R_2=((x_3,y_3),(x_4,y_3),(x_3,y_4),(x_4,y_4))\in A^4$$
such that we have the four identities
$$\phi(x_1,y_1)=\phi(x_3,y_3),\quad \quad \phi(x_2,y_1)=\phi(x_4,y_3),$$
$$\phi(x_1,y_2)=\phi(x_3,y_4),\text{ and }\phi(x_2,y_2)=\phi(x_4,y_4).$$
We allow degeneracies in this definition -- for example, there might be equalities between values of $\phi$ beyond those required by the definition itself.
\end{definition}

{As a subset of a partial Latin square $(X,Y,Z,A,\phi)$ which has been drawn as a labelled grid}, an octahedron is a configuration that looks like this (where we have chosen the example to emphasize that there is no ordering on $X$ or $Y$, so all we care about are the relations `is in the same column as', `is in the same row as', and `has the same label as').
\[\begin{matrix} c&&&d\\&a&&&&&b\\a&&&b\\&&\\&c&&&&&d\\ \end{matrix}\quad\quad\quad{(a,b,c,d\in Z)}\]

In a full Latin square, the relationship between the quadrangle condition (Definition~\ref{qc1}) and octahedra is simple: an $n\times n$ Latin square $A$ satisfies the quadrangle condition if and only if the number of octahedra contained in $A$ is equal to $n^5$ (which is the maximal possible number in a Latin square). To see this, note that the number of rectangles in $A$ is $n^4$, and if one wishes to find another rectangle with the same labelling, then there are at most $n$ choices for the first corner (since its label is determined) and at most one choice for each remaining corner (since their labels are determined, as well as at least one of their row and column). 

%This observation suggests that the obvious hypothesis to consider is that the number of cuboctahedra is at least $cn^5$, where $c>0$ is a constant independent of $n$. 

We now show that the multiplication table of a binary operation with many associative triples also contains many octahedra.

\begin{lemma}\label{assoctocuboct}
Let $X$ be a set of size $n$ and let $\circ$ be a partially defined binary operation on $X$ that is injective in each variable separately and for which there are at least $\e n^3$ triples $(x,y,z)\in X^3$ with $x\circ(y\circ z)=(x\circ y)\circ z$. Then the multiplication table of $\circ$ contains at least $\e^4n^5$ octahedra.
\end{lemma}

\begin{proof}
For each $b\in X$, let $W_b$ be the set of $(a,c)$ such that $a\circ(b\circ c)=(a\circ b)\circ c$. Then the average size of $|W_b|$ is at least $\e n^2$. Writing $\e_b$ for the density of $W_b$ in $X^2$, an easy Cauchy-Schwarz argument tells us that $W_b$ contains at least $\e_b^4n^4$ quadruples $(a_0,a_1,c_0,c_1)$ such that all four points $(a_i,c_j)$ belong to $W_b$. Therefore, by Jensen's inequality, the average number of such quadruples in $W_b$ is at least $\e^4n^4$. Each such quadruple yields a diagram of the following form.
\[\begin{matrix}g_1&&&&f_{01}&&f_{11}&&&&&&\\ g_0&&&&f_{00}&&f_{10}&&\\ c_1&&&&&&&&g_1&&f_{01}&&f_{11}\\ c_0 &&&&&&&&g_0&&f_{00}&&f_{10}\\ b&&&&d_0&&d_1\\ \\ \circ&&&&a_0&&a_1&&b&&d_0&&d_1\\ \end{matrix}\]
where the left column and bottom row say which elements are being multiplied together. The associativity of the triples $(a_i,b,c_j)$ is used to prove that $a_i\circ(b\circ c_j)=(a_i\circ b)\circ c_j=f_{ij}$, and the result is that each quadruple of triples gives us an octahedron. Note that from the octahedron we can reconstruct the pairs $(a_0,d_0)$ and $(a_1,d_1)$ from looking at which columns are used, and since the equation $a_0x=d_0$ has a unique solution, we can reconstruct $b$. Therefore, distinct $b$ give rise to distinct octahedra, and putting all this together implies that there are at least $\e^4n^5$ octahedra, as claimed.
\end{proof}

Observe that in an $n\times n$ grid labelled completely at random from a set of size $n$, the expected number of octahedra is $n^4$, since there are $n^8$ pairs of rectangles and the probability that the labels agree is $n^{-4}$ for each pair. It follows that the number of octahedra in a group multiplication table is far larger than the number of octahedra in a randomly labelled grid. The same is true of the number of octahedra in the multiplication table of a partially associative operation. These observations suggest that the octahedron count in a partial Latin square may serve as an indicator for underlying group structure. 

Thus, the hypothesis that we wish to consider is a weakening of the hypothesis of Theorem \ref{assoc_main}. Our main result is that we can obtain the same conclusion.

\begin{theorem}\label{main} Let $X,Y,Z$ be sets of size $n$, let $E\subset X\times Y$, and let $\lambda:E\to Z$ be such that $(X,Y,Z,E,\lambda)$ is a partial Latin square with at least $\e n^5$ octahedra. Then for every positive integer $b$ there exist a subset $A\subset E$ of density at least $\e^{b^{25b}}$, a metric group $G$, and maps $\phi:X\to G$, $\psi:Y\to G$ and $\omega:Z\to G$, such that the images $\phi(X), \psi(Y)$ and $\omega(Z)$ are 1-separated, and $d(\phi(x)\psi(y),\omega(z))\leq b^{-1}$ for every $(x,y,z)\in X\times Y\times Z$ such that $(x,y)\in A$ and $\lambda(x,y)=z$.
\end{theorem}  

{Theorem~\ref{assoc_main} follows immediately by applying Lemma~\ref{assoctocuboct} followed by Theorem~\ref{main}.}

While the quadrangle condition as defined for Latin squares in Definition~\ref{qc1} makes sense also for partial Latin squares, it is not symmetric in $X,Y$ and $Z$, and that turns out to be inconvenient. Instead, we make the following definitions.

\begin{definition}\label{qc2} Let $P=(X,Y,Z,A,\phi)$ be a partial Latin square. Then $P$ satisfies the \emph{column, row,} or \emph{label quadrangle condition} if it contains no configuration of the form
\[\begin{matrix}c&d\\ a&b\\ &&&c&&d\\ &&&a&b\\ \end{matrix}\ \ ,\qquad  \begin{matrix}c&d\\ a&b\\ &&&&d\\ &&&c\\ &&&a&b\\ \end{matrix}\ \ , \qquad\text{or}\qquad \begin{matrix} c&d\\ a&b\\ &&&c&d'\\ &&&a&b\\ \end{matrix}\ \ ,\]
respectively. It satisfies the \emph{quadrangle condition} if it satisfies all three of the column, row, and label quadrangle conditions.
\end{definition}
The column/row/label quadrangle condition is saying that there is no configuration obtained from an octahedron by taking one of its elements and changing just its column/row/label. For a full Latin square, the three quadrangle conditions are equivalent, but for a partial Latin square they are not, and it is convenient to distinguish between them. 

The following combinatorial statement, which is of independent interest, is in fact a special case of Theorem~\ref{main}. We will treat it separately in Section~\ref{pa:sec:flappycub} as an introduction to the more general argument.

\begin{theorem} \label{noflappycub} Let $X,Y$ and $Z$ be sets of size $n$, let $A\subset X\times Y$, and let $\phi:A\to Z$ be a partial Latin square with at least $\e n^5$ octahedra. Then there is a subset $B\subset A$ of size at least $\a n^2$, where $\a=\a(\e)>0$, that satisfies the quadrangle condition.
\end{theorem}
In order to prove it, we first obtain a dense subset that satisfies the label quadrangle condition, and then use symmetry to pass to further dense subsets that satisfy the column and row quadrangle conditions as well.

One might at first think that Theorem \ref{noflappycub} {(with a suitable bound) would imply} not just Theorem \ref{main}, but even a stronger result where $H$ is a $k$-approximate subgroup rather than an $(\e,k)$-approximate subgroup. However, while a Latin square that satisfies the quadrangle condition must be the multiplication table of a group, a partial Latin square {that satisfies the quadrangle condition} is not necessarily part of the multiplication table of a group: indeed, Example~\ref{fuzzyso3} of approximate multiplication on a $\d$-net of $\so3$ gives a counterexample. (This is significantly easier to prove than Green's result~\cite{greenso3}, which says that one cannot even restrict it to a dense set that is isomorphic to part of a group multiplication table.) More elementary counterexamples can be obtained by observing that if a group multiplication table ever contains the following configuration, 
\[\begin{matrix}e&d\\ &f&c\\ a&&b\\ \end{matrix}\]
then $ab^{-1}cd^{-1}ef^{-1}$ is equal to the identity, so any five of the labels determine the sixth. Thus, in a group multiplication table we have not only the (label) quadrangle condition but also a natural `pair of 6-cycles' generalization, which states that in a configuration such as the following, $f$ must equal $f'$.
\[\begin{matrix}e&d\\ &f&c\\ a&&b\\ &&&&e&d\\ &&&&&f'&c\\ &&&&a&&b\\ \end{matrix}\]
Note that that configuration itself satisfies the quadrangle condition (for trivial reasons) even if $f\ne f'$. 

What we therefore need to do in order to prove Theorem \ref{main} is find a subset of the partial Latin square that satisfies a generalized quadrangle condition that applies to all configurations up to a certain size. What those configurations are will be explained in the next subsection.

\subsection{Where the metric group comes from}\label{sketch}

It turns out that the metric group we obtain in Theorem~\ref{main} (and therefore also in Theorem~\ref{assoc_main}) is given by a simple universal construction. That does not mean that the proof of our main result is simple, however, because it is not obvious how to pass to a dense subset of the partial Latin square for which the universal construction has the desired properties. 

Suppose we are given a partial Latin square $(X,Y,Z,A,\lambda)$ and would like it to satisfy the conclusion of Theorem~\ref{main}. Then we want a metric group $G$ and maps $\phi:X\to G,\psi:Y\to G$ and $\omega:Z\to G$ such that the images of $\phi,\psi$ and $\omega$ are 1-separated and $d(\phi(x)\psi(y),\omega(z))\leq b^{-1}$ for every $x\in X, y\in Y$ and $z\in Z$ such that $\lambda(x,y)=z$. An obvious approach is to let $G$ be the free group on $X\cup Y\cup Z$ (we assume that $X, Y$ and $Z$ are disjoint -- if not, we make disjoint copies), to take the largest metric such that the second condition holds, and hope that it is large enough for the first condition to hold as well.

This metric can be described explicitly in a standard way using van Kampen diagrams. We take the group presentation with generators $X\cup Y\cup Z$ and all the relations given by the partial Latin square when it is thought of as a multiplication table. That is, the relations are all words of the form $xyz^{-1}$ such that $\lambda(x,y)=z$. Given two words $w_1,w_2$ in the free group on $X\cup Y\cup Z$, their distance is $b^{-1}$ times the area of the smallest van Kampen diagram with these relations and with boundary word $w_1w_2^{-1}$. Of course, if the relations do not imply that $w_1=w_2$, then there is no such van Kampen diagram, in which case we set the distance to be infinite. (It is not hard to see that a necessary and sufficient condition for the partial Latin square to arise as part of a group multiplication table is that \emph{all} distances between distinct generators are infinite.)

In order to prove Theorem~\ref{main}, we therefore need to pass to a dense subset of the partial Latin square with the property that given any two distinct elements $x_1,x_2$ of $X$, there is no van Kampen diagram of area less than $b$ with boundary word $x_1x_2^{-1}$, and similarly for $Y$ and $Z$.  

We shall discuss van Kampen diagrams more fully later in the paper, but for the benefit of the reader unfamiliar with them, we illustrate here how a partial Latin square that fails the label quadrangle condition gives rise to a van Kampen diagram of area 8 with boundary word of the form $z_1z_2^{-1}$. Suppose, then, that the partial Latin square contains the configuration
\[\begin{matrix}c&d\\ a&b\\ &&&c&d'\\ &&&a&b\\ \end{matrix}\]
and that $d\ne d'$. If we give appropriate names to the rows and columns, then the relations we obtain from this configuration give us the van Kampen diagram in Figure \ref{pa:fig:fovK} below. For example, the triangle towards the left with edges labelled $x_1, y_1$ and $a$ is telling us that the first $a$ in the configuration belongs to column $x_1$ and row $y_1$. (We direct edges so that if $xy=z$ then there is a directed path with edges $x$ then $y$, and a directed edge $z$ with the same start and end points.) Since there is also a triangle with edges labelled $x_3,y_3$ and $a$, we can read off that $x_1y_1=x_3y_3$. More generally, any cycle corresponds to a word that can be proved to be the identity if we use the given relations: in this case, the word is $x_1y_1y_3^{-1}x_3^{-1}$. 

	\begin{figure} 
		\centering
		\begin{tikzpicture}[scale=3.0, every node/.style={scale=1.0}]
		\draw (0,0) ellipse (2cm and 1cm);
		
		\draw (-2,0) node {$\bullet$};
		\draw (2,0) node {$\bullet$};
		\draw (-2/3,0) node {$\bullet$};
		\draw (2/3,0) node {$\bullet$};
		\draw (0,2/3) node {$\bullet$};
		\draw (0,-2/3) node {$\bullet$};
		\draw (0,1.1) node {$d$};
		\draw (0,-1.1) node {$d'$};
		
		\draw[-{Latex[length=4mm, width=2mm]}] (0,-2/3) -- (-2/3,0) node[midway, right] {$y_3$};
		\draw[{Latex[length=4mm, width=2mm]}-] (0,2/3) -- (2/3,0) node[midway, left] {$x_2$};
		\draw[{Latex[length=4mm, width=2mm]}-] (0,-2/3) -- (2/3,0) node[midway, left] {$x_4$};
		
		\draw[{Latex[length=5mm, width=1mm]}-] (2,0) -- (0,2/3) node[midway, above] {$y_2$};
		\draw[{Latex[length=4mm, width=2mm]}-] (0,2/3) -- (-2,0) node[midway, above] {$x_1$};	
		\draw[{Latex[length=5mm, width=1mm]}-] (2,0) -- (0,-2/3) node[midway, below] {$y_4$};
		\draw[{Latex[length=4mm, width=2mm]}-] (0,-2/3) -- (-2,0) node[midway, below] {$x_3$};
		
		\draw[-{Latex[length=4mm, width=2mm]}] (0,2/3) -- (-2/3,0) node[midway, right] {$y_1$};
		\draw[-{Latex[length=4mm, width=2mm]}] (-2,0) -- (-2/3,0) node[midway, above] {$a$};
		\draw[-{Latex[length=4mm, width=2mm]}] (2/3,0) -- (-2/3,0) node[midway, above] {$b$};
		\draw[-{Latex[length=5mm, width=1mm]}] (2/3,0) -- (2,0) node[midway, above] {$c$};
		\draw[-{Latex[length=4mm, width=2mm]}] (1.9988,0.005) -- (2,0);
		\draw[-{Latex[length=4mm, width=2mm]}] (1.9988,-0.005) -- (2,0);
		%\draw[-Latex] (2,-0.01) -- (2,0);
		
		%[pattern=north west lines, pattern color=red]
		%\begin{scope}
		%\clip (-2,0) rectangle (2,-2); % clipped area
		%\draw[pattern=north west lines, pattern color=blue] (0,0) ellipse (2cm and 1cm);
		%\end{scope}
		\end{tikzpicture}
		\caption{A van Kampen diagram of area 8 that corresponds to a proof that $d$ and $d'$ are equal in any group that satisfies the given relations.}\label{pa:fig:fovK}
	\end{figure}
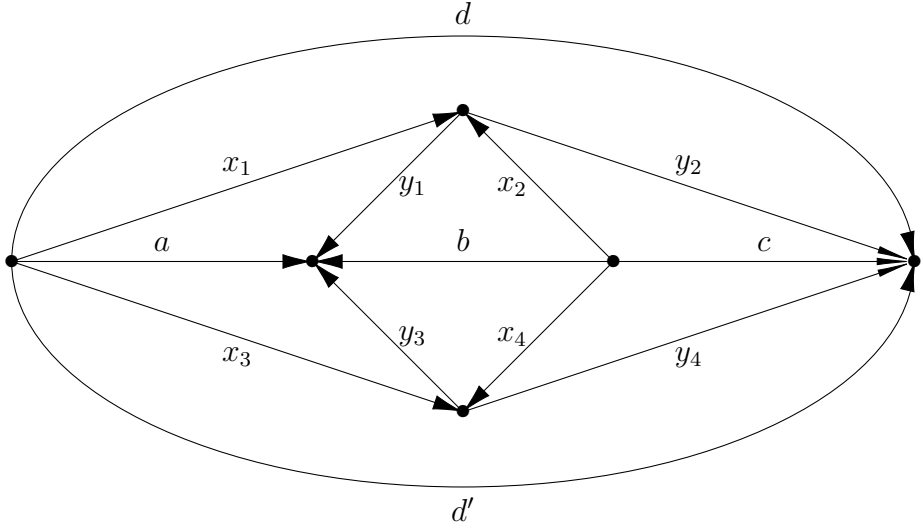

Since this van Kampen diagram has area 8 (which simply means that it has eight faces), it shows that the distance between $d$ and $d'$ is at most $8b^{-1}$. Therefore, for $b>8$ we need to avoid it. The generalized quadrangle condition we mentioned earlier is then that a partial Latin square should contain no configuration that corresponds to a van Kampen diagram with boundary word of the form $uv^{-1}$ for distinct generators $u,v$, and area less than $b$. 

Note that if $d=d'$, then the edges labelled $d$ and $d'$ become identified, and the resulting diagram is a triangulation of the 2-sphere. It is the triangulation given by the faces of an octahedron, which is why we use the word `octahedron' in Definition~\ref{cuboctahedron} to describe a pair of identically labelled rectangles in the grid.

\section{Obtaining $C$-well-defined completion operations}\label{pa:sec2}

In this section, we take an important first step by proving a weaker form of Theorem \ref{noflappycub}, or more precisely a generalization of a weaker form. Recall, as discussed in Definition~\ref{qc1}, that we can think about the quadrangle conditions in terms of certain functions being well-defined. For instance, a partial Latin square satisfies the label quadrangle condition if and only if there is a well-defined partial ternary operation on the set of labels that takes a triple $(a,b,c)$ to $d$ if and only if there exists a rectangle with its corners labelled $\begin{matrix}c&d\\ a&b\\ \end{matrix}$. We called this the \emph{rectangle completion operation}. We now make the following definitions. It will be convenient to think of a partial Latin square as a hypergraph.

\begin{definition}\label{2rcycledef} Let $H$ be a linear tripartite 3-uniform hypergraph. A \emph{label} $2r$-\emph{cycle} in $H$ is a sequence of triples 
\[(x_1,y_1,z_1),(x_2,y_2,z_2),\dots,(x_{2r},y_{2r},z_{2r})\]
such that $x_i=x_{i+1}$ if $i$ is even, and $y_i=y_{i+1}$ if $i$ is odd (and $x_{2r}=x_1$). The sequence $(z_1,\dots,z_{2r})$ is the \emph{label sequence} of the $2r$-cycle. A \emph{row $2r$-cycle} and a \emph{column $2r$-cycle} are defined in the same way after appropriate cyclic permutations of the coordinates. (For instance, for a column $2r$-cycle we require $y_i=y_{i+1}$ if $i$ is odd and $z_i=z_{i+1}$ if $i$ is even, and $(x_1,\dots,x_{2r})$ is the column sequence.)  

The \emph{label $2r$-cycle completion operation} is the not necessarily well-defined partial function that takes a $(2r-1)$-tuple $(z_1,\dots,z_{2r-1})$ to $z_{2r}$ if $(z_1,\dots,z_{2r})$ is the label sequence of some $2r$-cycle. We say that this operation is \emph{well-defined} if there is always at most one possibility for $z_{2r}$ given $z_1,\dots,z_{2r-1}$, and we say that it is $C$-\emph{well-defined} if the number of possibilities is always at most $C$.
\end{definition}

Definition~\ref{2rcycledef} generalises the notion of a rectangle in a partial Latin square (which is simply a label 4-cycle). Here are pictures of a label 6-cycle, a column 6-cycle, and a row 6-cycle.  

\[\begin{matrix} f&&e\\ &c&d\\ a&b\\ \end{matrix}\qquad\qquad \begin{matrix} a&b\\ &&&b&c\\ &&&&&c&a\\ \end{matrix}\qquad\qquad \begin{matrix}a\\ b\\ &b\\ &c\\ &&c\\ &&a\\ \end{matrix}\]

In this section, we shall be using the language of partial Latin squares, so we shall omit the word `label' when referring to label $2r$-cycles and the label $2r$-cycle completion operation. (Of course, a label $2r$-cycle in a partial Latin square is just a label $2r$-cycle in the corresponding hypergraph.) 

{When referring to a $2r$-cycle $C$, we will sometimes abbreviate the list of triples to a string, e.g. $C=x_1y_1\dots x_ry_r$ where the $x_i$ and $y_j$ are triples and the row of $x_1$ is shared with the row of $y_1$, etc.}

The main result of this section is the following.

\begin{theorem}\label{pa:CWD}
	Let $0<\epsilon<10^{-3}$ and let $k\ge2$ be a positive integer. Let $A$ be a partial Latin square containing at least $\epsilon n^5$ octahedra (see Definition~\ref{cuboctahedron}). Then we can find a subset $B\subset A$ of density $\beta\ge\epsilon^{10}$ with the property that for each $2\le r\le k$ the $2r$-cycle completion operation in $B$ is $\epsilon^{-33k^3}$-well-defined.
\end{theorem}

We begin with a well-known bound for the number of $2r$-cycles in a bipartite graph, which will underlie many of the calculations throughout this section.

\begin{lemma}\label{pa:lem1} 
	Let $A$ be a subset of the $n\times n$ grid of density $\alpha$. Then $A$ contains at least $\alpha^{2r} n^{2r}$ and at most $\alpha^rn^{2r}$ distinct labelled $2r$-cycles.
\end{lemma} 
\begin{proof}
	We may view $A$ as a bipartite graph with vertex sets $X$ and $Y$ of size $n$ and $\alpha n^2$ edges. Let $\lambda_1,\dots,\lambda_n$ be the singular values of the adjacency matrix of this graph. Then the number of $2r$-cycles is equal to $\sum_i\lambda_i^{2r}$. But the largest singular value is at least $\alpha n$, so this sum is at least $\alpha^{2r} n^{2r}$.	
		
	For the upper bound we observe that the number of $2r$-cycles can be counted by summing, over all (ordered) $r$-tuples $(x_1,\dots,x_r) \in A^r$, the indicator that there is a $2r$-cycle $x_1y_1\dots x_ry_r$. This sum is clearly at most $|A|^r=\alpha^rn^{2r}$, since that is the number of ways of choosing $(x_1,\dots,x_r)$ without the additional condition.
\end{proof}

The lower bound on the octahedron count in $A$ requires that the labelling of a random rectangle is repeated, on average, many times. This motivates the following definition.

\begin{definition}\label{pa:def1}
	Given a partial Latin square $A$, a $2r$-cycle is $\theta$-\emph{popular} in $A$ if the labelling of the cycle occurs at least $\theta n$ times in $A$. 
\end{definition}

\noindent Note that the trivial maximum for the number of occurrences of a $2r$-cycle with a given labelling is $n$, since once one has chosen which of at most $n$ points to choose with the first label, the condition that no label is repeated in any row or column implies that rest of the $2r$-cycle, if it exists, is determined by the labelling.

The first step towards obtaining the decompositions we need is a dependent random selection that ensures that almost all $2r$-cycles can be decomposed into popular rectangles in many ways. The next definition explains what we mean by `decomposed' here.

\begin{definition}\label{point}
	Given a $2r$-cycle $C=x_1y_1\dots x_ry_r$ in a partial Latin square $A$, a \emph{point decomposition} of $C$ in $A$ is a collection of $2r$ rectangles, all belonging to $A$ and all sharing a point $u$, with the corners opposite to $u$ being the $x_i$ and $y_i$. We call the point decomposition \emph{$\epsilon$-popular} if each of the $2r$ rectangles is $\epsilon$-popular in $A$ (see Definition~\ref{pa:def1}).
\end{definition}

Point decompositions for a rectangle and a 6-cycle are shown in Figure~\ref{pa:fig:pt4}.

\begin{figure}
	\centering
	\begin{subfigure}{0.5\linewidth}
		\centering
		\begin{tikzpicture}[scale=0.7, every node/.style={scale=1.0}]
		\draw (0,0) rectangle (6,5);
		\draw (2,0) -- (2,5);
		\draw (0,2) -- (6,2);
		
		\draw (0,5) node[label=above left:$a$] {};
		\draw (6,5) node[label=above right:$b$] {};
		\draw (6,0) node[label=below right:$c$] {};
		\draw (0,0) node[label=below left:$d$] {};
		
		\draw (2,2) node[label=below left:$u$] {};

		\draw (0,0) node[] {$\bullet$};
		\draw (2,0) node[] {$\bullet$};
		\draw (6,0) node[] {$\bullet$};
		\draw (0,2) node[] {$\bullet$};
		\draw (6,2) node[] {$\bullet$};
		\draw (0,5) node[] {$\bullet$};
		\draw (2,5) node[] {$\bullet$};
		\draw (6,5) node[] {$\bullet$};
		\draw (2,2) node[] {$\bullet$};
		\end{tikzpicture}
	\end{subfigure}%
	\begin{subfigure}{0.5\linewidth}
		\centering
		\begin{tikzpicture}[scale=0.7, every node/.style={scale=1.0}]
		\draw (-1,0) rectangle (1,5);
		\draw (1,3) rectangle (4,0);
		\draw (1,5) rectangle (2,2);
		\draw (-1,0) rectangle (4,2);
		
		\draw (-1,5) node[label=above left:$a$] {};
		\draw (2,5) node[label=above right:$b$] {};
		\draw (2,3) node[label=above right:$c$] {};
		\draw (4,3) node[label=above right:$d$] {};
		\draw (4,0) node[label=below right:$e$] {};
		\draw (-1,0) node[label=below left:$f$] {};
		\draw (1,2) node[label=below left:$u$] {};
		
		\draw (-1,0) node[] {$\bullet$};
		\draw (1,0) node[] {$\bullet$};
		\draw (4,0) node[] {$\bullet$};
		\draw (-1,2) node[] {$\bullet$};
		\draw (1,2) node[] {$\bullet$};
		\draw (2,2) node[] {$\bullet$};
		\draw (4,2) node[] {$\bullet$};
		\draw (1,3) node[] {$\bullet$};
		\draw (2,3) node[] {$\bullet$};
		\draw (4,3) node[] {$\bullet$};
		\draw (-1,5) node[] {$\bullet$};
		\draw (1,5) node[] {$\bullet$};
		\draw (2,5) node[] {$\bullet$};
		\end{tikzpicture}
	\end{subfigure}
	\caption{A point decomposition of a rectangle $(a,b,c,d)$ and a 6-cycle $(a,b,c,d,e,f)$. If all rectangles with $u$ and a vertex from the cycle as opposite corners are $\epsilon$-popular, then the decomposition is $\epsilon$-popular.}\label{pa:fig:pt4}
\end{figure}
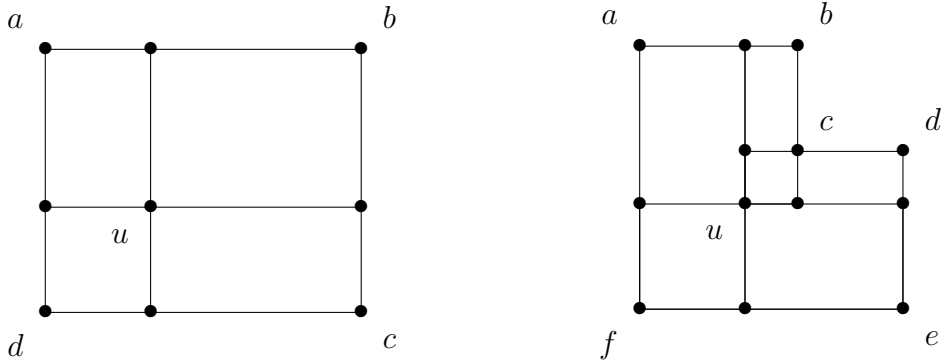

\begin{lemma}\label{pa:lem4}
	Let $0<\epsilon,\delta<\tfrac{1}{100}$ and let $k>1$ be a fixed integer. Given a partial $n\times n$ Latin square $A$ containing at least $\epsilon n^5$ octahedra, we can find
	a subset $B_1\subset A$ of density $\beta_1\ge\epsilon/2$ such that for each $2\le r\le k$, a proportion at least $1-\delta$ of $2r$-cycles in $B_1$ have at least $\delta \epsilon^{4k} n^2$ different $\epsilon/2$-popular point decompositions in $A$.
\end{lemma}

\begin{proof} We define a graph $G$ with vertex set given by $[n]^2$ corresponding to the cells of the $n\times n$ grid, and edges given by joining $x$ to $y$
	if the rectangle with opposite corners $x$ and $y$ has all its vertices in $A$ and is $\e/2$-popular. 
	
	Let $X$ be the number of edges in $G$ and $Y$ be the number of non-edges. An edge in $G$ can be associated to a set of at least $\e n/2$ (and at most $n$) octahedra, by combining the rectangle corresponding to the edge with one of the other rectangles with the same labelling. Similarly, a non-edge in $G$ can be associated to a set of less than $\e n/2$ octahedra. In such a way, all octahedra of $A$ are accounted for. Therefore
	$$ Xn+Y\e n/2\ge \e n^5$$
	$$\Rightarrow Xn+\e n^5/2\ge \e n^5$$
	so $G$ has average degree at least $\e n^2$.
	
	Let $\eta=\eta(\delta,k)=\delta\epsilon^{4k}$. A $2r$-cycle has at least $\eta n^2$ different $\epsilon/2$-popular point decompositions in $A$ if the common neighbourhood (in $G$) of the $2r$ corner vertices has size at least $\eta n^2$.
	
	We choose a vertex $v$ in $G$ uniformly at random, and let $N(v)$ be the neighbourhood of $v$ in $G$. This is our dependent random selection. It remains to prove that it works with positive probability.
	
	Let $C=x_1y_1\dots x_ry_r$ be a given $2r$-cycle in $A$. Let $N(C)$ be the set
	of vertices in $G$ that are joined to all of $x_1,\dots,y_r$. We shall say that $C$ is {\em bad} if
	$|N(C)|<\eta n^2$. If $C$ is bad, we have that
	\[\mathbb{P}(C\subset N(v))=\frac{|N(C)|}{n^2}<\eta.\]
	
	Let $Z_r$ count the number of bad $2r$-cycles in $N(v)$. We have $\mathbb{E}Z_r\le \eta n^{2r}$. Let $Z=\sum_{r=2}^k n^{-2r}Z_r$. Then
	\[\mathbb{E}Z\le \sum_{r=2}^k \eta\le k\eta.\]
	
	Our lower bound on the average degree of $G$ also gives us that
	\[\mathbb{E}(|N(v)|)\ge \epsilon n^2.\]
	
	In particular, we have
	\[\mathbb{E}\Big(|N(v)|n^{-2}-\epsilon/2-\epsilon Z(2k\eta)^{-1}\Big)\ge 0\]
	so there is a choice of vertex $v$ such that $|N(v)|n^{-2}\ge \epsilon/2$ and $|N(v)|n^{-2}\ge \epsilon\eta^{-1}Z/2k$. The first inequality gives us that the total count, $X_r$, of $2r$-cycles in $N(v)$ is at least $(\epsilon/2)^{2r}n^{2r}$, while the second inequality implies that $Z\le 2k\eta \epsilon^{-1}$. So for each $2\le r\le k$, 
	\[Z_rn^{-2r}\le 2k\eta \epsilon^{-1}\le 2k\eta \epsilon^{-1}(\epsilon/2)^{-2r}n^{-2r}X_r, \]
	which implies that
	\[Z_r\le k\eta(\epsilon/2)^{-(2r+1)}X_r.\]
	
	Therefore, letting $B_1=N(v)$ for this choice of $v$, we have $\beta_1n^2=|N(v)|\ge\epsilon n^2/2$ and the proportion of $2r$-cycles in $N(v)$ which are bad is at most $k\eta(\epsilon/2)^{-(2r+1)}\le\delta$.
\end{proof} 

Using Lemma~\ref{pa:lem4} we may pass to a dense subset $B_1$ of $A$ such that almost all $2r$-cycles have many (within a constant factor of the trivial maximum) popular point decompositions in $A$. However, for our purposes the `almost all' is not sufficient, and we need to use a more complicated decomposition to boost Lemma~\ref{pa:lem4} into an `all' statement.

The following definitions introduce these more complex decompositions.

\begin{definition}\label{ring}
	Let $X$ be a partial Latin square. Given a $2r$-cycle $C=x_1y_1\dots x_ry_r$, let $C'=y_1'x_2'\dots x_r'y_r'x_1'$ be a second $2r$-cycle in $X$ such that all the rectangles $R_1,\dots,R_{2r}$ with opposite corner pairs either $(x_i,x_i')$ or $(y_i,y_i')$ belong to $X$. We call the collection of $C'$ and the rectangles $R_1,\dots,R_{2r}$ a \emph{ring decomposition} of $C$ in $X$. An example (for $r=2$) is shown in Figure~\ref{pa:fig:rd}. As in Definition~\ref{point}, we say that a ring decomposition $(C',R_1,\dots, R_{2r})$ of $C$ is \emph{$\e$-popular} if $C'$ and all the rectangles $R_1,\dots, R_{2r}$ are $\e$-popular in $X$ (see Definition~\ref{pa:def1}).
\end{definition}

\begin{remark}\label{ordering}
{The reason we listed the points of the cycle $C'$ in the order $y_1'x_2'\dots x_r'y_r'x_1'$ is that according to Definition~\ref{2rcycledef} of a $2r$-cycle, the first two points of a cycle must share a row rather than a column, so had we listed the vertices in the more obvious order, then it would not technically have been a cycle. This is slightly more than hair splitting: it is important that if two points in $C$ share a row, then the corresponding points in $C'$ share a column, and vice versa.}
\end{remark}

\begin{definition}\label{full}
	Let $X$ be a partial Latin square. Let $C$ be a $2r$-cycle in $X$, and let $(C',R_1,\dots, R_{2r})$ be a ring decomposition of $C$. A \emph{full decomposition} of $C$ in $X$ consists of a point decomposition of $C'$, together with point decompositions of each rectangle $R_i$. Since a point decomposition of a $2r$-cycle itself consists of $2r$ rectangles, a full decomposition of a $2r$-cycle involves $10r$ rectangles as shown (for $r=2$) in Figure~\ref{pa:fig:fd}. As usual, we say that a full decomposition is \emph{$\e$-popular} if all of these $10r$-rectangles are $\e$-popular in $X$ (see Definition~\ref{pa:def1}).
\end{definition}

\begin{figure}
	\centering
	\begin{tikzpicture}[scale=1.0, every node/.style={scale=1.0}]
	\draw (0,0) rectangle (1,2);
	\draw (1,0) rectangle (3,1);
	\draw (2,1) rectangle (3,3);
	\draw (0,2) rectangle (2,3);
	
	\draw (0,0) node[label=below left:$d$] {};
	\draw (3,0) node[label=below right:$c$] {};
	\draw (3,3) node[label=above right:$b$] {};
	\draw (0,3) node[label=above left:$a$] {};
	
	\draw (0,0) node[] {$\bullet$};
	\draw (1,0) node[] {$\bullet$};
	\draw (3,0) node[] {$\bullet$};
	\draw (3,1) node[] {$\bullet$};
	\draw (3,3) node[] {$\bullet$};
	\draw (2,3) node[] {$\bullet$};
	\draw (0,3) node[] {$\bullet$};
	\draw (0,2) node[] {$\bullet$};
	\draw (1,2) node[] {$\bullet$};
	\draw (2,2) node[] {$\bullet$};
	\draw (1,1) node[] {$\bullet$};
	\draw (2,1) node[] {$\bullet$};
	\end{tikzpicture}
	\caption{A ring decomposition of a 4-cycle $(a,b,c,d)$.}\label{pa:fig:rd}
\end{figure}
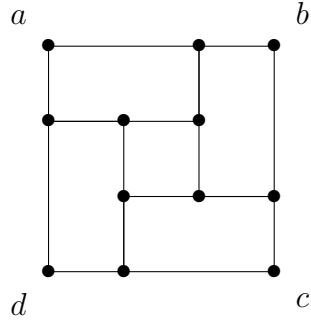

\begin{figure} 
	\centering
	\begin{tikzpicture}[scale=2.5, every node/.style={scale=1.0}]
	\draw (0,0) rectangle (1,2);
	\draw (1,0) rectangle (3,1);
	\draw (2,1) rectangle (3,3);
	\draw (0,2) rectangle (2,3);
	
	\draw (0,0) node[label=below left:$d$] {};
	\draw (3,0) node[label=below right:$c$] {};
	\draw (3,3) node[label=above right:$b$] {};
	\draw (0,3) node[label=above left:$a$] {};
	
	\draw (0,0) node[] {$\bullet$};
	\draw (1,0) node[] {$\bullet$};
	\draw (3,0) node[] {$\bullet$};
	\draw (3,1) node[] {$\bullet$};
	\draw (3,3) node[] {$\bullet$};
	\draw (2,3) node[] {$\bullet$};
	\draw (0,3) node[] {$\bullet$};
	\draw (0,2) node[] {$\bullet$};
	\draw (1,2) node[] {$\bullet$};
	\draw (2,2) node[] {$\bullet$};
	\draw (1,1) node[] {$\bullet$};
	\draw (2,1) node[] {$\bullet$};
	
	%crosshairs
	\draw (0,2.55) -- (2,2.55);
	\draw (1.55,2) -- (1.55,3);
	\draw (2.4,1) -- (2.4,3);
	\draw (2,1.7) -- (3,1.7);
	\draw (1.8,0) -- (1.8,1);
	\draw (1,0.3) -- (3,0.3);
	\draw (0,1.6) -- (1,1.6);
	\draw (0.5,0) -- (0.5,2);
	\draw (1,1.3) -- (2,1.3);
	\draw (1.35,1) -- (1.35,2);
	
	\draw (0,2.55) node[] {$\bullet$};
	\draw (2,2.55) node[] {$\bullet$};
	\draw (1.55,2) node[] {$\bullet$};
	\draw (1.55,3) node[] {$\bullet$};
	\draw (1.55,2.55) node[] {$\bullet$};
	
	\draw (2.4,1) node[] {$\bullet$};
	\draw (2.4,3) node[] {$\bullet$};
	\draw (2,1.7) node[] {$\bullet$};
	\draw (3,1.7) node[] {$\bullet$};
	\draw (2.4,1.7) node[] {$\bullet$};
	
	\draw (1.8,0) node[] {$\bullet$};
	\draw (1.8,1) node[] {$\bullet$};
	\draw (1,0.3) node[] {$\bullet$};
	\draw (3,0.3) node[] {$\bullet$};
	\draw (1.8,0.3) node[] {$\bullet$};
	
	\draw (0,1.6) node[] {$\bullet$};
	\draw (1,1.6) node[] {$\bullet$};
	\draw (0.5,0) node[] {$\bullet$};
	\draw (0.5,2) node[] {$\bullet$};
	\draw (0.5,1.6) node[] {$\bullet$};
	
	\draw (1,1.3) node[] {$\bullet$};
	\draw (2,1.3) node[] {$\bullet$};
	\draw (1.35,1) node[] {$\bullet$};
	\draw (1.35,2) node[] {$\bullet$};
	\draw (1.35,1.3) node[] {$\bullet$};
	\end{tikzpicture}
	\caption{A full decomposition of the rectangle with labels $(a,b,c,d)$. If each of the 20 small rectangles in the figure is $\epsilon$-popular, then we say that the decomposition is $\epsilon$-popular.}\label{pa:fig:fd}
\end{figure}
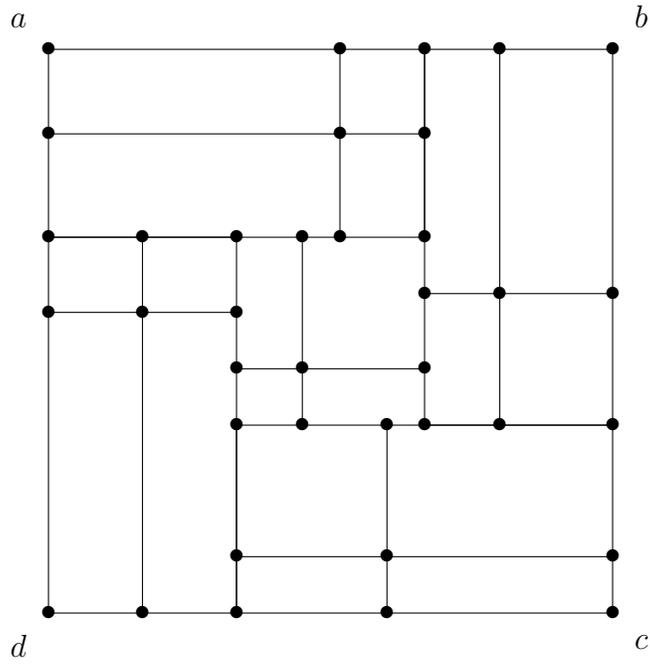

\begin{remark}\label{pa:rem1}
	It will be important to keep track of the order (in $n$) of the trivial maxima for the number of ring decompositions and full decompositions of a $2r$-cycle in a dense subset of an $n\times n$ Latin square. The number of ring decompositions is at most $n^{2r}$, since a ring decomposition of a $2r$-cycle $C$ is uniquely defined by a $2r$-cycle $C'$. In a full decomposition, $C'$ and all the rectangles in the ring decomposition are given point decompositions, each of which can be chosen in at most $n^2$ ways. So the number of full decompositions is at most $n^{2r+2(2r+1)}=n^{6r+2}$.
\end{remark}

Our next step is to pass to a subset $B_2$ of $B_1$ such that all $2r$-cycles in $B_2$ have within a constant factor of the trivial maximum number of ring decompositions. Since almost all $2r$-cycles in $B_1$ have many popular point decompositions, we will then be able to pass to a further subset $B_3$ of $B_2$ so that all $2r$-cycles in $B_3$ have within a constant factor of the trivial maximum number of $\epsilon$-popular full decompositions.

In order to achieve the first step of this process, we will again apply a dependent random selection argument. The following lemma is in fact far more general than we need, but the full statement is more natural to prove than the special case that we will use. 

\begin{lemma}\label{pa:lemtechnical}
	Let $k$ be a positive integer. Let $G$ be a bipartite graph with vertex classes $X$, $Y$ of size $n$ and edge density $\delta$, with $0<\delta<\tfrac{1}{100}$. Then we can pass to subsets $X'\subset X$ and $Y'\subset Y$, each of size at least $\delta^2n/16$, such that the edge density in $G'=G|_{X'\times Y'}$ is at least $\delta/4$ and for any $2\le r\le k$ and any choice of $r$ vertices $x_1,\dots,x_r\in X'$ and $y_1,\dots,y_r\in Y'$ we have at least $\delta^{5k^2+4k} n^{2r}$ choices of vertices $u_1,\dots, u_r, v_1,\dots, v_r$ in $G$ with $u_iv_j\in E(G)$, $x_iu_i\in E(G)$ and $y_iv_i\in E(G)$ for each $i,j\in\{1,\dots,r\}$.
\end{lemma}

{Figure~\ref{DRSlemfig} shows the vertices $x_i,y_j, u_k$ and $v_l$ with the corresponding edges when $k=r=3$.}

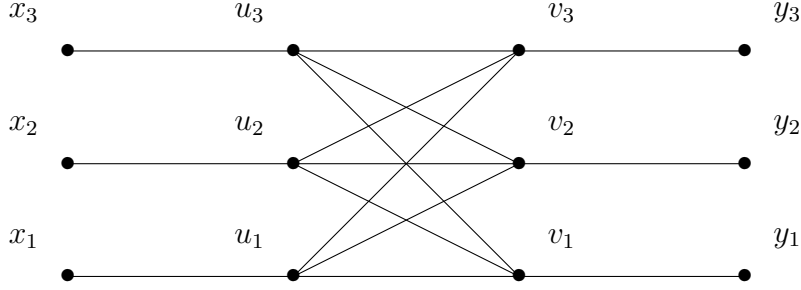
\begin{figure} 
	\centering
	\begin{tikzpicture}[scale=1, every node/.style={scale=1.0}]
	\draw (0,0) node[label=above left:$x_1$] {$\bullet$};
	\draw (0,1.5) node[label=above left:$x_2$] {$\bullet$};
	\draw (0,3) node[label=above left:$x_3$] {$\bullet$};
	
		\draw (3,0) node[label=above left:$u_1$] {$\bullet$};
	\draw (3,1.5) node[label=above left:$u_2$] {$\bullet$};
	\draw (3,3) node[label=above left:$u_3$] {$\bullet$};
	
		\draw (6,0) node[label=above right:$v_1$] {$\bullet$};
	\draw (6,1.5) node[label=above right:$v_2$] {$\bullet$};
	\draw (6,3) node[label=above right:$v_3$] {$\bullet$};
	
		\draw (9,0) node[label=above right:$y_1$] {$\bullet$};
	\draw (9,1.5) node[label=above right:$y_2$] {$\bullet$};
	\draw (9,3) node[label=above right:$y_3$] {$\bullet$};
	
	\draw (0,0) -- (3,0);
	\draw (0,1.5) -- (3,1.5);
	\draw (0,3) -- (3,3);
	
	\draw (6,0) -- (9,0);
	\draw (6,1.5) -- (9,1.5);
	\draw (6,3) -- (9,3);
	
	\draw (3,0) -- (6,0);
	\draw (3,0) -- (6,1.5);
	\draw (3,0) -- (6,3);
	\draw (3,1.5) -- (6,0);
	\draw (3,1.5) -- (6,1.5);
	\draw (3,1.5) -- (6,3);
	\draw (3,3) -- (6,0);
	\draw (3,3) -- (6,1.5);
	\draw (3,3) -- (6,3);
	\end{tikzpicture}
	\caption{The subgraph found in Lemma~\ref{pa:lemtechnical} when $k=r=3$.}\label{DRSlemfig}
\end{figure}
\begin{proof}
	Let us begin by discarding all vertices from $X$ of degree smaller than $\delta n/2$. This leaves a set $X_1\subset X$ of size at least $\delta n/2$.
	
	Let $c_1=\delta^{2k}$ and $c_2=\delta^{5k}$. We will use a dependent random selection argument that allows us to pass to a subset $X_2\subset X_1$ of size at least $(\delta^2/8) n$ with the property that for a $(1-c_1)$ proportion of choices $(x_1,\dots,x_{k+1})$ from $X_2$ we have at least $c_2 n$ vertices in the shared neighbourhood $\Gamma(x_1,\dots,x_{k+1})\subset Y$. 
	
	We do this by picking a vertex $y\in Y$ at random and considering $\Gamma(y)$. Observe that
	\[\mathbb{E}(|\Gamma(y)|)\ge \delta|X_1|/2\ge \delta^2n/4 .\]
	Let us call a $(k+1)$-tuple $(x_1,\dots,x_{k+1})$ \emph{bad} if $|\Gamma(x_1,\dots,x_{k+1})|<c_2 n$. Let $B$ be the number of bad tuples in $\Gamma(y)$. The probability that a given bad ${(k+1)}$-tuple belongs to $\Gamma(y)$ is less than $c_2$, since for this to happen $y$ must be picked from $\Gamma(x_1,\dots,x_{k+1})$. Therefore
	\[\mathbb{E}(B)<c_2 n^{k+1}\]
	and so
	\[\mathbb{E}(c_1|\Gamma(y)|^{k+1}-c_1(\delta^2/8)^{k+1} n^{k+1}-B)> (c_1(\delta^2/4)^{k+1}-c_1(\delta^2/8)^{k+1}-c_2)n^{k+1}.\]
	Since $c_2=c_1\delta^{3k}$, this expectation is positive and so there is some choice of $y$ for which both $c_1|\Gamma(y)|^{k+1}\ge c_1(\delta^2/8)^{k+1} n^{k+1}$ and $c_1|\Gamma(y)|^{k+1}\ge B$. These inequalities imply that $|\Gamma(y)|\ge (\delta^2/8) n$ and that at most a proportion $c_1$ of the ${(k+1)}$-tuples from $\Gamma(y)$ are bad. So we may take $X_2=\Gamma(y)$.
	
	Now we let $X_3$ be the subset of $X_2$ consisting of all vertices $x_1\in X_2$ with the property that for a proportion $(1-2c_1)$ of the choices of $x_2,\dots,x_{{k+1}}\in X_2$, the shared neighbourhood $\Gamma(x_1,\dots,x_{{k+1}})\subset Y$ contains at least $c_2 n$ vertices. Since $|\Gamma(x_1,\dots,x_{{k+1}})|\ge c_2 n$ for at least a proportion $(1-c_1)$ of all ${(k+1)}$-tuples, $|X_3|\ge |X_2|/2\ge \delta^2n/16$.
	
	Since each vertex in $X_3$ has at least $\delta n/2$ neighbours in $Y$, the number of edges from $Y$ to $X_3$ is at least $\delta n|X_3|/2$. We now pass to the subset $Y_1\subset Y$ that consists of all vertices with at least $\delta |X_3|/4$ edges into $X_3$. We note that $|Y_1|\ge \delta n/4$.
	
	Now let $x_1,\dots, x_k$ be chosen from $X_3$ and $y_1,\dots,y_k$ from $Y_1$. Let $A_1,\dots,A_k$ be the neighbourhoods of the $y_i$ in $X_3$ -- note that $|A_i|\ge \delta |X_3|/4$. Let $T=A_1\times\dots\times A_k$ and note that it has cardinality at least $(\delta |X_3|/4)^{k}\ge (\delta |X_2|/8)^{k}$. 
	
	By the choice of $X_3$, we know that the number of choices of $u_1,\dots,u_{k}\in X_2$ such that $|\Gamma(x_i,u_1,\dots,u_{k})|<c_2 n$ is at most $2c_1 |X_2|^{k}$ for each $i=1,\dots,k$. Letting $c_1=\delta^{2k}$ so that $2c_1 k< (\delta/8)^{k}/2$ 
	and noting that $|T|=(\delta|X_2|/8)^k$, we see that there must be at least $(\delta|X_2|/8)^{k}/2$ choices of $(a_1,\dots,a_k)\in T$ such that $|\Gamma(x_i,a_1,\dots,a_k)|\ge c_2 n$ for each $i=1,\dots,k$. Observe that for any such choice of $(a_1,\dots,a_k)$ and for any choice of $b_i\in \Gamma(x_i,a_1,\dots,a_k)$ we get a complete bipartite graph between the $a_i$ and the $b_i$ as well as the edges $x_ib_i$ and $y_ia_i$ for each $i$.
	
	The number of choices of the $a_i$ and $b_i$ from the above paragraph is at least 
	\[\Big((\delta|X_2|/8)^{k}/2\Big)\times \Big((c_2 n)^k\Big)\ge (\delta^3/64)^{k}(\delta^{5k})^kn^{2k}\]
	\[\ge\delta^{5k^2+4k}n^{2k}.\]
	
	Observe that the subgraph induced by the $x_i,y_j,a_k$ and $b_l$ contains a $2r$-cycle $a_1b_1\dots a_rb_r$ as well as the edges $x_ia_i$ and $y_ib_i$ for each $i$. Moreover, the edge density in $X_3\times Y_1$ is at least $\delta/4$, so taking $X'=X_3$ and $Y'=Y_1$, the result follows.
\end{proof}

\begin{remark}\label{pa:rem2}
	It is well known that given a dense bipartite graph $G$, we may pass to a dense subgraph $H$ such that any two vertices of $H$ are joined by many $P_3$s in $G$~\cite{DRC}. Lemma~\ref{pa:lemtechnical} shows that a considerable generalization of this statement is available for relatively little extra effort: given any fixed bipartite graph $H'$ with $t$ special vertices $v_1,\dots,v_t$ such that the shortest path from any $v_i$ to any $v_j$ has length at least 3, we may pass to a dense subgraph $H$ of $G$ such that for any $u_1,\dots, u_t$ the number of isomorphic copies $\phi(H')$ of $H'$ in $H$ with $\phi(v_i)=u_i$ for all $i$ is within a constant of the trivial maximum. The $P_3$ statement is the special case where $H'$ is a path of length 3 and $v_1$ and $v_2$ are its endpoints. (A similar observation was made in a blog post of Tao \cite{taocohomology}, but he was content to discuss just the special case he needed, and he left the proof as an exercise for the reader.) 
\end{remark}

As an immediate corollary we obtain the following result, which will soon be applied in order to help guarantee the presence of many ring decompositions.

\begin{lemma}\label{pa:lem5}
	Let $k>1$ be a positive integer. Let $G$ be a bipartite graph with vertex classes $X$, $Y$ of size $n$ and edge density $\delta$, with $0<\delta<\tfrac{1}{100}$. Then we can pass to subsets $X'\subset X$ and $Y'\subset Y$, each of size at least $\delta^2n/16$, such that the edge density in $G'=G|_{X'\times Y'}$ is at least $\delta/4$ and for any $2\le r\le k$ and any choice of $r$ vertices $x_1,\dots,x_r\in X'$ and $y_1,\dots,y_r\in Y'$ we have at least $\delta^{7k^2} n^{2r}$ choices of $2r$-cycle $u_1v_1\dots u_rv_r$ in $G$ with $x_iu_i\in E(G)$ and $y_iv_i\in E(G)$ for each $i=1,\dots,r$.
\end{lemma}
\begin{proof}
	The result follows by applying Lemma~\ref{pa:lemtechnical}, and noting that the complete bipartite graph on $r+r$ vertices contains a $2r$-cycle. 
\end{proof}

When viewed as a statement about subsets of the grid, Lemma~\ref{pa:lem5} states that we may pass to a dense subset $B_2\subset B_1$ such that all $2r$-cycles in $B_2$ have many ring decompositions in $B_1$ (see Definition~\ref{ring}). We must now pass to a further subset in which all $2r$-cycles have many popular full decompositions.

In the statement of the following lemma, we introduce a set $\mathcal{C}$ of cycles. Each $C\in\mathcal{C}$ is a $2r$-cycle for some $2\le r \le k$. At this point in the argument, it may help to think about $\mathcal{C}$ as the set of $\theta$-popular cycles (for some suitable $\theta$), as this will be how $\mathcal{C}$ is defined in our first application of the lemma. However, we will apply the lemma again later in the paper with a different collection $\mathcal{C}$ of cycles, and hence we state the result in this more general way.

\begin{lemma}\label{pa:lem6} 
	Let $0<\beta,\delta,\gamma<\tfrac{1}{100}$ and $k>1$. Let $B$ be an $n\times n$ partial Latin square of density at least $\beta$, and let $\mathcal{C}$ be a collection of cycles in $B$ with the property that for each $2\le r\le k$ at least a proportion
	$1-\delta$ of $2r$-cycles in $B$ belong to $\mathcal{C}$. If $\delta\leq\beta^{9k^2}$ then we
	can find a subset $B'$ of $B$ with density $\beta'\ge\beta^8$ with the property that any $2r$-cycle in
	$B'$ has at least $\beta^{8k^2}n^{2r}$ different ring decompositions (see Definition~\ref{ring}) into cycles belonging to $\mathcal{C}$.
\end{lemma} 
\begin{proof} 
	Recall from Definition~\ref{ring} that a ring decomposition of a $2r$-cycle $C$ involves a $2r$-cycle $C'$ and $2r$ rectangles $R_1,\dots,R_{2r}$ between these cycles, as shown in Figure~\ref{pa:fig:rd}. The ring decomposition is fully determined by the choice of cycle $C'$. We shall call a ring decomposition of a cycle $C$ \emph{good} if all of $C', R_1,\dots,R_{2r}$ belong to $\mathcal{C}$.  We call a $2r$-cycle \emph{indecomposable} if it has fewer than $\beta^{8k^2}n^{2r}$ good ring decompositions in $B$, so that our goal is to pass to a dense subset of $B$ in which no $2r$-cycle is indecomposable for $2\le r\le k$.
	
	In parallel with the partial Latin square $B$, we shall also consider the corresponding bipartite graph $G$ in which the rows and columns form the vertex sets and the points of $B$ form the edges (the labels here are ignored). The $2r$-cycles in $B$ correspond precisely to the $2r$-cycles in $G$.
	
	We begin by applying Lemma~\ref{pa:lem5} to $G$. This allows us to pass to a subset $B'$ of $B$ of density at least $(\beta^2/16)^2(\beta/4)\ge\beta^7$ with the property that each $2r$-cycle in $B'$ has at least $\beta^{7k^2}n^{2r}$ ring decompositions in $B$. 
	
	Let $C$ be a $2r$-cycle in $B'$. If $(C', R_1,\dots,R_{2r})$ is a ring decomposition of $C$ which is bad (i.e. not good), then either $C'$ does not belong to $\mathcal{C}$ or some $R_i$ does not belong to $\mathcal{C}$, or both. Since only a proportion $\delta\le \beta^{9k^2}$ of all $2r$-cycles in $B$ are bad, and the maximum possible number of $2r$-cycles in $B'$ is trivially bounded by $n^{2r}$, the number of ring decompositions of $C$ for which $C'$ is bad is at most $\beta^{9k^2}n^{2r}$. Therefore $C$ has at least $(\beta^{7k^2}-\beta^{9k^2})n^{2r}$ ring decompositions $(C', R_1,\dots,R_{2r})$ in $B$ in which the $2r$-cycle $C'$ is good.
	
	If for each $2\le r \le k$ there are no more than $\beta^7n^2/4k^2$ disjoint indecomposable $2r$-cycles in $B'$, then discarding all points from a maximal disjoint set of indecomposable cycles we discard at most $\beta^7n^2/2$ points, leaving a set of density at least $\beta^7/2\ge \beta^8$ with no indecomposable cycles and so we are done.
	
	Thus, for some $r$ it must be possible to find at least $\beta^7n^2/4k^2$ disjoint indecomposable $2r$-cycles $C_1,\dots,C_t$ in $B'$. Each of these cycles $C_i$ has at least
	$$(\beta^{7k^2}-\beta^{9k^2})n^{2r}>\beta^{7k^2}n^{2r}/2$$
	 ring decompositions in $B$ in which the $2r$-cycle $C'$ is good but also has fewer than $\beta^{8k^2}n^{2r}$ good ring decompositions. It follows that each $C_i$ has at least
	 $$(\beta^{7k^2}-2\beta^{8k^2})n^{2r}/2>\beta^{7k^2}n^{2r}/3$$
	 ring decompositions $(C_i', R_{i,1},\dots,R_{i,2r})$ in which some rectangle $R_{i,j}$ does not belong to $\mathcal{C}$. 
	
	Since the cycles $C_i$ are disjoint, the collections $R_{i_1,j}$ and $R_{i_2,j}$ of rectangles are disjoint when $i_1\neq i_2$ (since a rectangle $R_{i_1,j}$ contains points from $C_{i_1}$ which is disjoint from $C_{i_2}$). Furthermore, a rectangle $R_{i,j}$ can belong to at most $n^{2r-2}$ different ring decompositions of $C_i$ (since $R$ determines two vertices of $C_i'$). This means that $B$ contains at least 
	$$\Big(\beta^{7k^2}n^{2}/3\Big)\Big(\beta^7n^2/4k^2\Big)>\beta^{9k^2}n^4$$
	bad rectangles in $B$.
	
	But the number of bad rectangles is at most $\delta\beta^2n^4$, so if $\delta\leq\beta^{9k^2}$ then we have a contradiction.	
\end{proof}

By applying Lemmas~\ref{pa:lem4}, \ref{pa:lem5} and \ref{pa:lem6} we will be able to pass to a dense subset $B$ of $A$ in which all $2r$-cycles have many popular full decompositions (see Definition~\ref{full} and Figure~\ref{pa:fig:fd}). Before we give the details, we give one more technical lemma which draws a connection between popular full decompositions and the $2r$-cycle completion operation described in Definition~\ref{2rcycledef}. 

\begin{lemma}\label{pa:lem7}
	Let $A$ be a partial Latin square and let $B$ be a subset of $A$. Suppose that every $2r$-cycle in $B$ has at least $\gamma n^{6r+2}$ different $\epsilon$-popular full decompositions in $A$ (see Definition~\ref{full}). Then the $2r$-cycle completion operation in $B$ is $\epsilon^{-10r}\gamma^{-1}$-well-defined (see Definition~\ref{2rcycledef}).
\end{lemma}
\begin{proof}
	Suppose that we have a tuple $(a_1,\dots,a_{2r-1})$ such that the set $\{x_i\}$ of possible labelling completions has size at least $K$. For each completion we can find $\gamma n^{6r+2}$ $\epsilon$-popular full decompositions. 
	
	Let us think about a typical one of these decompositions as follows. (For the discussion that follows, it may help to look at Figure~\ref{pa:fig:fd2}.) We begin with a $2r$-cycle $C$ with points $x_1,y_1,\dots,x_r,y_r$, where $x_i$ has label $a_{2i-1}$ when $1\leq i\leq k$, $y_i$ has label $a_{2i}$ when $1\leq i\leq k-1$, and we do not know about the label attached to the point $y_r$. (It is important to be clear that the $x_i$ and $y_i$ are elements of $[n]^2$ and not of $[n]$ in this discussion.)
	
	Next, we have another $2r$-cycle $C'=y_1'x_2'\dots x_r'y_r'x_1'$ {(see Remark \ref{ordering} for an explanation of why we list its points in this order).} 
	
	Now we complete the cycles $C$ and $C'$ to a ring decomposition by adding in $2r$ points $u_1,v_1,\dots,u_r,v_r$, where $u_i$ is in the row that contains $x_i$ and $y_i$ and the column that contains $x_i'$ and $y_i'$, and $v_i$ is in the column that contains $y_i$ and $x_{i+1}$ and the row that contains $y_i'$ and $x_{i+1}'$. (The points $u_i$ and $v_i$ do not form a $2r$-cycle.) 
	
	The rectangles $R_1,\dots,R_{2r}$ of this ring decomposition are given by $S_i=(x_i,u_i,x_i',v_{i-1})$ and $T_i=(y_i,v_i,y_i',u_i)$. To form a point decomposition, we now add points $p_i$ and $q_i$, and form the four rectangles that have a vertex in $S_i$ and the opposite vertex at $p_i$, and the four rectangles that have a vertex in $T_i$ and the opposite vertex at $q_i$. As well as the point $p_i$, we have to add four more points to $S_i$ in order to complete the decomposition into four rectangles. Of these four points, let $r_i$ and $s_i$ be the ones in the same row and the same column as $x_i$; we shall not bother giving names to the other two. Similarly, let $w_i$ and $z_i$ be the points in the same column and row as $y_i$ that are part of the decomposition of $T_i$ into four rectangles. 
	
	\begin{figure} 
		\centering
		\begin{tikzpicture}[scale=3.5, every node/.style={scale=0.9}]
		
		\fill [blue,opacity=0.2] (0,0) rectangle (1,2);
		\fill [blue,opacity=0.2] (2,1) rectangle (3,3);
		\fill [red,opacity=0.2] (0,2) rectangle (2,3);
		\fill [red,opacity=0.2] (1,0) rectangle (3,1);
		\fill [green,opacity=0.2] (1,1) rectangle (2,2);
		\draw (0,0) rectangle (1,2);
		\draw (1,0) rectangle (3,1);
		\draw (2,1) rectangle (3,3);
		\draw (0,2) rectangle (2,3);
		
		\draw (0,0) node[label=below left:$y_2$] {};
		\draw (3,0) node[label=below right:$x_2$] {};
		\draw (3,3) node[label=above right:$y_1$] {};
		\draw (0,3) node[label=above left:$x_1$] {};
		
		\draw (0,0) node[] {$\bullet$};
		\draw (1,0) node[label=below:$u_2$] {$\bullet$};
		\draw (3,0) node[] {$\bullet$};
		\draw (3,1) node[label=right:$v_1$] {$\bullet$};
		\draw (3,3) node[] {$\bullet$};
		\draw (2,3) node[label=above:$u_1$] {$\bullet$};
		\draw (0,3) node[] {$\bullet$};
		\draw (0,2) node[label=left:$v_2$] {$\bullet$};
		\draw (1,2) node[label={[label distance=-0.3cm]135:$y_2'$}] {$\bullet$};
		\draw (2,2) node[label={[label distance=-0.3cm]45:$x_1'$}] {$\bullet$};
		\draw (1,1) node[label={[label distance=-0.3cm]225:$x_2'$}] {$\bullet$};
		\draw (2,1) node[label={[label distance=-0.3cm]315:$y_1'$}] {$\bullet$};
		
		%crosshairs
		\draw (0,2.55) -- (2,2.55);
		\draw (1.55,2) -- (1.55,3);
		\draw (2.4,1) -- (2.4,3);
		\draw (2,1.7) -- (3,1.7);
		\draw (1.8,0) -- (1.8,1);
		\draw (1,0.3) -- (3,0.3);
		\draw (0,1.6) -- (1,1.6);
		\draw (0.5,0) -- (0.5,2);
		\draw (1,1.3) -- (2,1.3);
		\draw (1.35,1) -- (1.35,2);
		
		\draw (0,2.55) node[label=left:$s_1$] {$\bullet$};
		\draw (2,2.55) node[] {$\bullet$};
		\draw (1.55,2) node[] {$\bullet$};
		\draw (1.55,3) node[label=above:$r_1$] {$\bullet$};
		\draw (1.55,2.55) node[label={[label distance=-0.3cm]225:$p_1$}] {$\bullet$};
		
		\draw (2.4,1) node[] {$\bullet$};
		\draw (2.4,3) node[label=above:$z_1$] {$\bullet$};
		\draw (2,1.7) node[] {$\bullet$};
		\draw (3,1.7) node[label=right:$w_1$] {$\bullet$};
		\draw (2.4,1.7) node[label={[label distance=-0.3cm]45:$q_1$}] {$\bullet$};
		
		\draw (1.8,0) node[label=below:$r_2$] {$\bullet$};
		\draw (1.8,1) node[] {$\bullet$};
		\draw (1,0.3) node[] {$\bullet$};
		\draw (3,0.3) node[label=right:$s_2$] {$\bullet$};
		\draw (1.8,0.3) node[label={[label distance=-0.3cm]45:$p_2$}] {$\bullet$};
		
		\draw (0,1.6) node[label=left:$w_2$] {$\bullet$};
		\draw (1,1.6) node[] {$\bullet$};
		\draw (0.5,0) node[label=below:$z_2$] {$\bullet$};
		\draw (0.5,2) node[] {$\bullet$};
		\draw (0.5,1.6) node[label={[label distance=-0.3cm]225:$q_2$}] {$\bullet$};
		
		\draw (1,1.3) node[] {$\bullet$};
		\draw (2,1.3) node[] {$\bullet$};
		\draw (1.35,1) node[] {$\bullet$};
		\draw (1.35,2) node[] {$\bullet$};
		\draw (1.35,1.3) node[] {$\bullet$};
		\end{tikzpicture}
		\caption{A full decomposition of the 4-cycle $(a,b,c,d)$, with labelling as described in the proof of Lemma~\ref{pa:lem7}. The rectangles $S_1$ and $S_2$ are shaded in red, $T_1$ and $T_2$ in blue, and the cycle $C'$ in green.}\label{pa:fig:fd2}
	\end{figure}
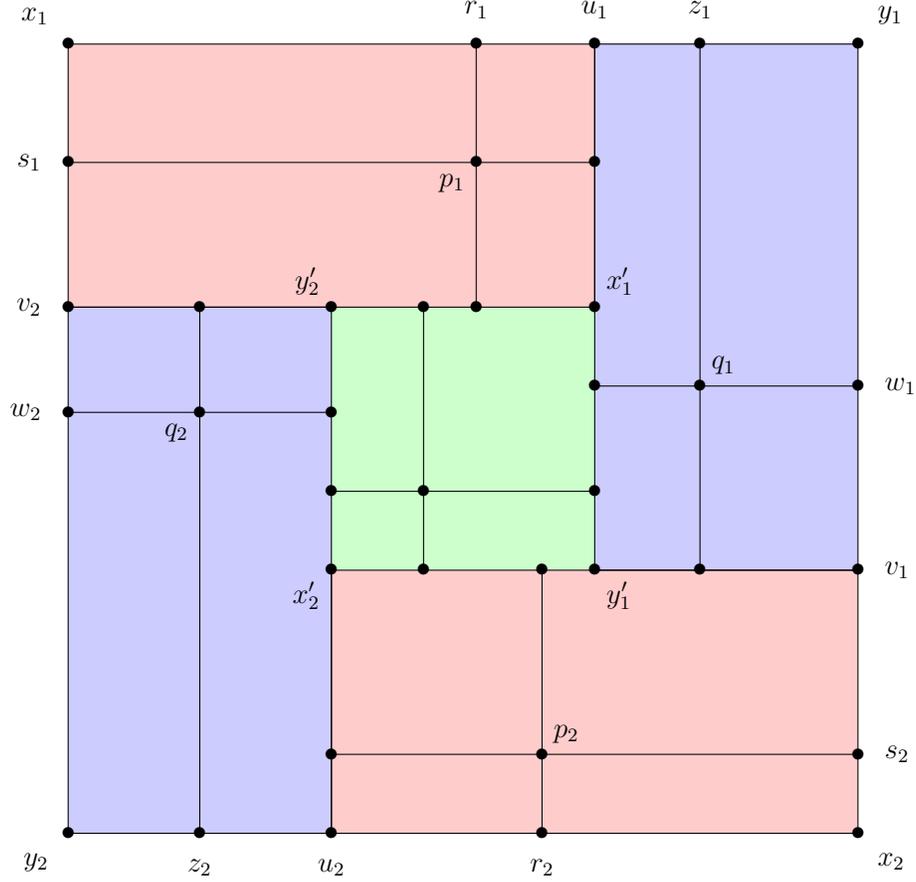
	
	Now let us consider a certain subset of the (variable set of) points of the full decomposition. We shall take the points $u_i$ and $v_i$, the points $r_i$ and $s_i$, and the points $w_i$ and $z_i$ with $1\leq i\leq r-1$. We shall also take the two points from the point decomposition of $C'$ that are in the same row and column as $x_1'$, and the two points from the decomposition of the rectangle $T_r$ that are in the same row and column as $y_r'$. This makes a total of $6r+2$ points, so by the pigeonhole principle we can find some choice of labellings of these $6r+2$ points that occurs at least $K\gamma$ times amongst the set of $\epsilon$-popular full decompositions of $2r$-cycles $C$ for which the points $x_1,y_1,\dots,x_r$ are labelled $a_1,\dots,a_{2r-1}$. 
	
	Observe that a full decomposition of a given cycle is uniquely determined by the way it is labelled, since once a point has been specified, any other point in the same row or column is then determined by its label. Observe also that since each rectangle in an $\epsilon$-popular full decomposition is $\epsilon$-popular, given three labels of any rectangle there are at most $1/\epsilon$ different choices of label for the fourth, since otherwise there would be more than $n$ rectangles that shared three labels, which is impossible. 
	
	Our aim now is use this observation to show that once the labellings of the $6r+2$ points specified earlier are given, the number of possible labellings of the remaining points is at most $\epsilon^{-10r}$. Since we know that it is also at least $K\gamma$, this will give us our desired upper bound on $K$.  
	
	To do this, we consider the 4-cycle completion operation as described in Definition~\ref{2rcycledef}. The observation above implies that if we know the labels at some subset of the points of the $\e$-popular full decomposition from which we can repeatedly apply the 4-cycle completion operation to generate the entire decomposition, and if there are $t$ other points, then the number of possible ways of completing the labelling is at most $\epsilon^{-t}$. We apply this to the set of $6r+2$ points we have chosen.
	
	Note first that the rectangles $R_i$, apart from the rectangle containing the unfixed point of $C$, each contain five points from the set in their point decompositions, and furthermore these five generate (using the 4-cycle completion operation) the other four. Therefore the closure of the set contains all the points in all the point decompositions of the rectangles $S_1,\dots,S_r$ and $T_1,\dots,T_{r-1}$. These include the points $x_1',\dots,x_r'$ and $y_1',\dots,y_{r-1}'$. Since we also have the points in the same row and column as $x_1'$, we obtain the central point of the point decomposition of $C'$, and using this we can work round $C'$ and obtain all the points in its point decomposition. And now we have five points of the rectangle $T_r$ that generate the others (since they lie along two edges), which shows that the $6r+2$ points we choose generate all the points of the full decomposition. It is not hard to check that a full decomposition contains $18r+1$ points, so we find, as promised, that the number of labellings given the labels at the $6r+2$ points and $2r-1$ of the points of $C$ is at most $\epsilon^{-10r}$, as claimed, and this proves that $K\leq\epsilon^{-10r}\gamma$. 
\end{proof}

We are now ready to prove the result we stated at the beginning of the section.

\begin{proof}[Proof of Theorem \ref{pa:CWD}]
	Apply Lemma~\ref{pa:lem4} with $\delta=(\epsilon /2)^{9k^2}$. This allows us to pass to a subset $B_1\subset A$ of density $\beta_1\ge\epsilon/2$ such that for each $2\le r\le k$ a proportion at least $1-\delta$ of $2r$-cycles in $B_1$ have at least 
	$$(\epsilon/2)^{9k^2}\epsilon^{4k}n^2\ge \epsilon^{11k^2}n^2$$ 
	different $\epsilon/2$-popular point decompositions.
	
	From here we apply Lemma~\ref{pa:lem6}, where we take the collection $\mathcal{C}$ of cycles to be those $2r$-cycles $C$ with $2\le r\le k$ and such that $C$ has at least $\epsilon^{11k^2}n^2$ different $\e/2$-popular point decompositions. We can do this since $B_1$ has density $\beta_1\ge\epsilon/2$, so $\delta\leq\beta_1^{9k^2}$. The lemma gives us a subset $B_2$ of $B_1$ of density (in the original $n\times n$ grid) $\beta_2\ge\beta_1^8\ge\epsilon^{10}$ in which every $2r$-cycle in $B_2$ has at least 
	\begin{align*}\Big((\epsilon/2)^{8k^2}n^{2r}\Big)\Big(\epsilon ^{11k^2}n^2\Big)^{2r+1}
	&\ge\epsilon^{20k^2+20k^3}n^{6r+2}\\
	&\ge \epsilon^{30k^3}n^{6r+2}\\
	\end{align*}
	different $\epsilon/2$-popular full decompositions. (The first bracket on the left is a lower bound for the number of ring decompositions in which every cycle belongs to $\mathcal{C}$, and the second is a lower bound for the number of ways of converting each one into an $\epsilon/2$-popular full decomposition.) 
	
	This allows us to apply Lemma~\ref{pa:lem7} with $\gamma=\epsilon^{30k^3}$, which implies the result with $B=B_2$ (since $(\epsilon/2)^{-10r}\gamma^{-1}\le (\epsilon/2)^{-10k-30k^3}\le \epsilon^{-33k^3}$).
\end{proof}

We draw attention here to an analogy with {the additive combinatorics result mentioned at the beginning of the paper,} {which states that if $\phi:\Z_N\to\Z_N$ is a map} such that $\phi(x)+\phi(y)=\phi(z)+\phi(w)$ for a positive proportion of the quadruples $x+y=z+w$, {then} we can pass to a dense subset $A\subset\Z_N$ such that the restriction of $\phi$ to $A$ is a Freiman homomorphism. One way of proving this result begins by showing that it is possible to pass to a set $A'$ such that for each $w$, the number of values that $\phi(x)+\phi(y)-\phi(z)$ can take when $x+y-z=w$ is bounded by some constant $C$ that is independent of $N$. This first step mirrors what we have achieved thus far. It is then necessary to find a separate argument to pass to a further subset where $C$ is reduced to 1. 

We have to do the same here, though at this point the analogy breaks down somewhat, since in the additive problem, Pl\"unnecke's inequality is used, but our setting does not involve an ambient group so we do not appear to have an analogous tool. Thus, while Theorem~\ref{pa:CWD} constitutes significant progress towards our positive result, it turns out that we are still quite a long way from reducing $C$ to $1$. 

\section{Simplifying the decompositions}\label{pa:sec3}

Perhaps surprisingly, the first step towards reducing $C$ to $1$ will involve abandoning full decompositions. The reason is that when we start to consider van Kampen diagrams, we want the configurations that we consider in a partial Latin square to correspond to triangulated surfaces. Each labelled point in a partial Latin square corresponds to a triangle, and the triangles corresponding to two points share an edge if and only if the points share a column, a row, or a label. The problem with full decompositions is that they give rise to edges that are contained in more than two faces, and hence not to surfaces.   

In this section, we shall use Theorem~\ref{pa:CWD} as a tool in a `second pass' through the arguments in Section~\ref{pa:sec2}. Our first lemma for this section shows that the property of $C$-well-definedness (see Definition~\ref{2rcycledef}) is sufficient to ensure that almost all of the cycles in $B$ are popular, for a lower threshold of popularity (see Definition~\ref{pa:def1}). This is significant because it enables us to repeat the process of the previous section, but eliminates the need for Lemma~\ref{pa:lem4} and point decompositions, which are the source of the edges that are contained in too many faces. We will simply be able to reapply Lemma~\ref{pa:lem6} to the subset $B$ with a different set $\mathcal{C}$ of cycles: $\mathcal{C}$ will now simply be the set of $\theta$-popular cycles, for some appropriate $\theta$, rather than the set of those cycles with many popular point decompositions. 

\begin{lemma}\label{pa:lem8}
	Let $B$ be an $n\times n$ partial Latin square of density $\beta$. Suppose that the $2r$-cycle completion operation in $B$ is $C$-well-defined (see Definition~\ref{2rcycledef}). Let $\delta$, $\theta$ be such that $\beta^{2r}\delta\theta^{-1}> C$. Then the proportion of $2r$-cycles in $B$ that are not $\theta$-popular is at most $\delta$.
\end{lemma}
\begin{proof}
	By Lemma~\ref{pa:lem1}, the number of $2r$-cycles in $B$ is at least $\beta^{2r}n^{2r}$. Therefore, given a tuple $(a_1,\dots,a_{2r-1})$ of labels, the number of $2r$-cycles with first $2r-1$ labels $(a_1,\dots,a_{2r-1})$ is on average at least $\beta^{2r}n$. However, since the $2r$-cycle completion operation is $C$-well-defined we have further that the number of different $a_{2r}$ completing a $2r$-cycle labelling $(a_1,\dots,a_{2r})$ in $B$ is at most $C$.
	
	If a proportion greater than $\delta$ of $2r$-cycles are not $\theta$-popular, then by averaging there must be some $(a_1,\dots,a_{2r-1})$ such that a proportion greater than $\delta$ of $2r$-cycles starting with these labels are not $\theta$-popular. But that means that there must be more than $\beta^{2r}\delta\theta^{-1}> C$ completions which is a contradiction to the assumption that the $2r$-cycle completion operation in $B$ is $C$-well-defined.
\end{proof}

We are now ready to put together our technical lemmas to prove the following proposition, which will be the main tool for passing to a dense subset of a partial Latin square that avoids the configurations we wish to avoid.

\begin{proposition}\label{pa:prop2}
	Let $0<\epsilon<10^{-3}$ and let $A$ be a partial Latin square containing at least $\epsilon n^5$ octohedra (see Definition~\ref{cuboctahedron}). Let $k\ge 100$ be an integer. Then we can find $B\subset A$ of density $\beta\ge  \epsilon^{80}$ such that for each $r=2,\dots,k$ we have that every label $2r$-cycle in $B$ has at least $\epsilon^{80k^2}n^{2r}$ different $\theta$-popular ring decompositions (see Definition~\ref{ring}) in $A$, where $\theta\ge  \epsilon^{35k^3}$. Moreover, the number of octahedra in $B$ is at least $\epsilon^{70k^3}n^5$.
\end{proposition}

\begin{proof}
	We begin by applying Theorem~\ref{pa:CWD}. This allows us to pass to a subset $B_1\subset A$ of density $\beta_1\ge\epsilon^{10}$ with the property that for each $2\le r\le k$ the $2r$-cycle completion operation in $B$ is $C$-well-defined, where $C=\epsilon^{-33k^3}$.
	
	By Lemma~\ref{pa:lem8} we see that a proportion greater than $1-\delta$ of $2r$-cycles (for each $2\le r\le k$) in $B$ are $\theta$-popular for any choice of $\theta<\beta_1^{2k}\delta/C$.
	
	We now apply Lemma~\ref{pa:lem6} again, but taking the collection $\mathcal{C}$ to consist of those $2r$-cycles for $2\le r\le k$ which are $\theta$-popular. To do this, we take $\delta=(\beta_1)^{9k^2}$. With this value of $\delta$, we may take some $\theta\ge\epsilon^{20k+90k^2+33k^3}\ge \epsilon^{35k^3}$. 
	 
	The lemma then gives us a subset $B_2$ of density $\beta_2\ge\beta_1^{8}\ge \epsilon^{80}$ in which every $2r$-cycle in $B_2$ has at least $\beta_1^{8k^2}n^{2r}\ge \epsilon^{80k^2}n^{2r}$ many $\theta$-popular ring decompositions in $A$.
	
	Since $B_2$ is a subset of $B_1$, the rectangle completion operation in $B_2$ is still $C$-well-defined. By Lemma~\ref{pa:lem1} the number of rectangles in $B_2$ is at least $\beta_2^4n^4$, and since octahedra are counted by pairs of rectangles with the same labelling, the octahedron count is minimized when the the number of rectangles with each labelling is as balanced as possible (by convexity). For each triple of labels $(a,b,c)$ the number of possible completions $d$ is at most $C$, so the number of octahedra is at least 
	$$(\beta_2^4 n/C)^2n^3=(\beta_2^8/C^2)n^5\ge \epsilon^{70k^3}n^5$$
	as required.
\end{proof}

We now observe that an octahedron, which consists of two identically labelled rectangles, still corresponds to an octahedron if we permute the coordinates of the points. Indeed, using the hypergraph point of view we can define it more symmetrically as a sequence of eight triples 
\[(x_{000},y_{000},z_{000}),(x_{001},y_{001},z_{001}),\dots,(x_{111},y_{111},z_{111})\]
such that $x_\epsilon=x_\eta$ if $\e_2=\eta_2$ and $\e_3=\eta_3$, $y_\e=y_\eta$ if $\e_1=\eta_1$ and $\e_3=\eta_3$, and $z_\e=z_\eta$ if $\e_1=\eta_1$ and $\e_2=\eta_2$. (Note that this is \emph{not} an octahedron in the usual hypergraph sense. Indeed, it cannot be, since adjacent faces of an octahedron intersect in an edge, so that hypergraph is not linear. Rather, it is the hypergraph formed by the triangular faces of a cuboctahedron. {However, as mentioned in Section~\ref{sketch}, this object does correspond to an octahedron when viewed in the framework of van Kampen diagrams, {so to keep terminology to a minimum we use the word `octahedron' even when referring} to the corresponding hypergraph just described.})

This observation implies that Proposition~\ref{pa:prop2} remains true if the word `label' is replaced by either `column' or `row' and we interpret `ring decompositions' in the appropriate way -- that is, after permuting the coordinates. (To put it slightly differently, to perform a column ring decomposition, one can interchange the column and label coordinates, perform a label ring decomposition, and interchange the column and label coordinates again.)
Thus, we will in fact be able to find decompositions of all three kinds of $2r$-cycles. This will be crucial for our argument in Section~\ref{popreparg1}.

\begin{theorem}\label{pa:thm2}
	Fix $\e\le 10^{-3}$ and $k\ge 100$. Let $A$ be a 3-uniform, linear hypergraph that contains at least $\epsilon n^5$ octahedra. Then there exists a sequence $A=A_0\supset A_1\supset \dots$ such that each $A_i$ has density at least $\alpha_i(\epsilon,k)$ and $A_i$ contains at least $\epsilon_i(\epsilon,k)n^5$ octahedra, and for each $r=2,\dots,k$, every $2r$-cycle in $A_i$ has at least $\gamma_i(\epsilon,k) n^{2r}$ different $\theta_i(\epsilon,k)$-popular ring decompositions in $A_{i-1}$. Each of the parameters $\alpha_i,\epsilon_i,\theta_i,\gamma_i$ may be chosen to be at least $\e^{k^{15i}}$.
\end{theorem}
\begin{proof}
	Let $A_0=A$. We now repeatedly use Proposition~\ref{pa:prop2}.
	Once we have chosen $A_i$, we apply Proposition~\ref{pa:prop2} to pass to a dense subset $B_i^{(1)}$ in which all label $2r$-cycles have at least $\gamma_in^{2r}$ different $\theta_i$-popular ring decompositions in $A_i^{(1)}$ for $2\le r\le k$. We then apply the proposition to pass to a dense subset $B_i^{(2)}$ of $B_i^{(1)}$ in which all column $2r$-cycles have at least $\gamma_i'n^{2r}$ different $\theta_i'$-popular ring decompositions in $B_i^{(1)}$ for $2\leq r\leq k$. Finally, we apply the proposition to pass to a dense subset $B_i^{(3)}$ of $B_i^{(2)}$ such that all row $2r$-cycles have at least $\gamma_i''n^{2r}$ different $\theta_i''$-popular ring decompositions in $B_i^{(2)}$ for $2\leq r\leq k$. The dependence between these parameters, which we shall discuss in more detail in a moment, is given by Proposition~\ref{pa:prop2}. We now set $A_{i+1}$ to be $B_i^{(3)}$. 
	
	If the density of $A_i$ is $\alpha_i$ and the number of octahedra in $A_i$ is $\epsilon_in^5$, then the density of $B_i^{(1)}$ is at least $\epsilon_i^{80}$. Moreover, the octahedron count of $B_i^{(1)}$ is at least $\epsilon_i^{70k^3}n^5$. Therefore, the density of $B_i^{(2)}$ is at least 
	$$(\epsilon_i^{70k^3})^{80}\ge\epsilon_i^{2^{13}k^3}$$
	and the octahedron count of $B_i^{(2)}$ is at least 
	$$(\e_i^{70k^3})^{70k^3}\ge \e_i^{2^{13}k^6}.$$
	This implies that the density of $B_i^{(3)}$ is at least
	$$(\e_i^{2^{13}k^6})^{80}\ge \e_i^{2^{20}k^6} \ge \e_i^{k^{15}}$$
	and the octahedron count is at least 
	$$ (\e_i^{2^{13}k^6})^{70k^3}n^5\ge \e_i^{2^{20}k^9}n^5 \ge \e_i^{k^{15}}n^5. $$
	Lastly, we also have 
	$$\theta_i''\ge (\e_i^{2^{13}k^6})^{35k^3}\ge \e_i^{2^{19}k^9} \ge \e_i^{k^{15}} $$
	and
	$$\gamma_i''\ge (\e_i^{2^{13}k^6})^{80k^2}\ge \e_i^{2^{20}k^8} \ge \e_i^{k^{15}}.$$
	Note also that $\gamma_i,\gamma_i'\ge \gamma_i''$ and $\theta_i,\theta_i'\ge\theta_i''$ since the octahedron counts of $B_i^{(1)}$ and $B_i^{(2)}$ are larger than that of $B_i^{(3)}$.
	
	Therefore, $A_{i+1}$ is still dense, and has the property that any $2r$-cycle (for $2\le r\le k$) in $A_{i+1}$ is popularly decomposable in $A_i$. 
	
	After each step of the inductive construction, the density $\alpha_{i+1}$ is at least $\e_i^{k^{15}}$ and the octahedron count $\epsilon_{i+1} n^5$ is at least $\e_i^{k^{15}}n^5.$ The threshold for popularity $\theta_{i+1}$ is at least $\e_i^{k^{15}}$, and $\gamma_{i+1}\ge \e_i^{k^{15}}$ also.
	
	Therefore, starting at $A_0=A$ with $\epsilon n^5$ octahedra, we find that for $i\ge 1$ we have
	$$\epsilon_i\ge\e^{k^{15i}}.$$
	This gives us
	$$\alpha_i\ge (\e^{k^{15(i-1)}})^{k^{15}}= \e^{k^{15i}}$$
	and similarly $\theta_i\ge\e^{k^{15i}}$ and $\gamma_i\ge \e^{k^{15i}}.$

	Therefore every $2r$-cycle in $A_i$ is $\e^{k^{15i}}$-popularly decomposable in $A_{i-1}$ in at least $\e^{k^{15i}} n^{2r}$ different ways.
\end{proof}

To close this section, we observe that saying that a $2r$-cycle is popularly decomposable in many different ways is equivalent to saying that the $2r$-cycle has many decompositions of another kind, which we now define.

\begin{definition}\label{shatteredring}
	Let $C$ be a $2r$-cycle $x_1y_1\dots x_ry_r$ with $x_1$ and $y_1$ sharing a row. A \emph{dispersed ring decomposition} of $C$ consists of a $2r$-cycle $x_1'y_1'\dots x_r'y_r'$ with $x_1'$ and $y_1'$ sharing a column, together with rectangles $R_i=x_i''u_ix_i'''v_i$ and $S_i=y_i''w_iy_i'''z_i$ (where $u_i$ shares a row with $x_i''$ and $w_i$ shares a column with $y_i''$) such that for each $i$, $x_i$ and $x_i''$ have the same label, $x_i'$ and $x_i'''$ have the same label, $y_i$ and $y_i''$ have the same label, $y_i'$ and $y_i'''$ have the same label, $u_i$ and $z_i$ have the same label, and $w_i$ and $v_{i+1}$ have the same label.
\end{definition}

The reason for this terminology is as follows. Consider a ring decomposition of a label $2r$-cycle. It consists of another label $2r$-cycle and a collection of label 4-cycles. To obtain a dispersed ring decomposition, we take the various cycles that form the decomposition and `disperse' them by replacing them by other cycles of the same length that have the same label sequences. The conditions above are precisely the ones that are guaranteed to hold for the various cycles after we have done this: a point in one cycle has to have the same label as a point in another cycle if before the `dispersing' they were the same point. Note that to say that a ring decomposition is popular is precisely to say that one can obtain many dispersed ring decompositions from it in this way.

The equivalence mentioned above is given more precisely by the following simple lemma.

\begin{lemma}\label{pa:shatteredcount}
	Let $(X,Y,Z,A,\phi)$ be a partial Latin square with $|X|=|Y|=|Z|=n$. Let $F$ be a $2r$-cycle which is $\theta$-popularly decomposable in $A$ in at least $\gamma n^{2r}$ different ways. Then $F$ has at least $\gamma\theta^{2r+1}n^{4r+1}$ different dispersed ring decompositions. 
\end{lemma}

\begin{proof}
	Without loss of generality $F$ is a label $2r$-cycle. There are at least $\gamma n^{2r}$ different ring decompositions of $F$ into cycles that are $\theta$-popular. Each of these popular cycles can be replaced with one of $\theta n$ different cycles with the same label sequence as the original, giving a total of $(\theta n)^{2r+1}$ further choices, from which the result follows.
\end{proof}

In the next section we shall see how to describe a dispersed ring decomposition of a $2r$-cycle as a triangulated surface of a certain shape that has that $2r$-cycle as its boundary, which should make the concept significantly clearer. 

\section{The van Kampen picture}\label{sec4}

In this section we discuss in more detail the role that van Kampen diagrams play in our proof. Though the material is well known, we include it here partly because our treatment of it is not standard in all respects, and partly for the convenience of readers who may not be familiar with it. 

Let $(X,Y,Z,A,\phi)$ be a partial Latin square. As we have seen, we can think of it as a linear tripartite 3-uniform hypergraph $H$. However, it will be very helpful to represent $H$ in an unusual way as follows. 

\begin{definition}\label{vKcomplex} Let $H$ be a linear tripartite 3-uniform hypergraph with vertex sets $X,Y$ and $Z$. The \emph{van Kampen complex} $K(H)$ corresponding to $H$ is a {directed} simplicial complex with three vertices $u,v,w$, a {directed} 1-cell joining $u$ to $v$ for each $x\in X$, a {directed} 1-cell joining $v$ to $w$ for each $y\in Y$, a {directed} 1-cell joining $u$ to $w$ for each $z\in Z$, and a 2-cell bounded by the 1-cells corresponding to $x\in X$, $y\in Y$ and $z\in Z$ if and only if $xyz$ is a face of $H$.
\end{definition}

{The directions are chosen to conform with the definition of a van Kampen diagram. A face $xyz$ of $H$ corresponds to a point $(x,y,z)$ in the partial Latin square, which, when thought of as a multiplication table, corresponds to the relation $xy=z$. We therefore want the 2-cell corresponding to $xyz$ to have boundary word $xyz^{-1}$.} 

To put the above definition more loosely, $K(H)$ is simply the result of replacing the vertex sets of $H$ by sets of edges joining $u$ to $v$, $v$ to $w$, and $u$ to $w$, so for each face of $H$, its vertices become edges.

Several of the concepts that we have defined up to now become more natural and geometrical when reinterpreted in terms of {triangulated} surfaces that may or may not live in $K(H)$. The vertices $u,v$ and $w$ in $K(H)$ do not play an important role, and in order to represent these reinterpretations pictorially it is convenient to `de-identify' them. For example, to obtain the surface corresponding to what we have called an octahedron, we take two edges between $u$ and $v$, two between $v$ and $w$, and two between $u$ and $w$, and we fill all of the eight resulting triangles with faces, which results not in an octahedron, but in an octahedron with opposite vertices identified. However, there is no harm in drawing it without the identifications and remembering that strictly speaking the identifications are needed if we want to talk about copies of the octahedron inside $K(H)$. The next definitions allow us to do this precisely.

\begin{definition}\label{surfacedefs}
An $n_1\times n_2\times n_3$ \emph{tripartite simplicial complex} is a {directed} 2-dimensional simplicial complex with three vertex classes $U$, $V$ and $W$, and {$n_1$ edges directed from $U$ to $V$, $n_2$ edges directed from $V$ to $W$, and $n_3$ edges directed from $U$ to $W$.} (Multiple edges are allowed.) A \emph{homomorphism} from a tripartite simplicial complex $K_1$ to a tripartite simplicial complex $K_2$ is a map $\phi$ that takes the vertices, edges and faces of $K_1$ to the vertices, edges and faces of $K_2$ and respects incidences, in the sense that if $u,v$ are the start and end points of an edge $e$ in $K_1$, then $\phi(u)$ and $\phi(v)$ are the start and end points of $\phi(e)$, and if the edges $e_1,e_2,e_3$ bound a face $f$ in $K_1$, then $\phi(e_1),\phi(e_2)$ and $\phi(e_3)$ bound the face $\phi(f)$ in $K_2$. If in addition $\phi$ is a bijection on the edges and faces (but not necessarily the vertices), then it is an \emph{isomorphism}. 

A \emph{tripartite surface} is a tripartite simplicial complex such that each edge is contained in at most two faces. The \emph{boundary} of a tripartite surface is the set of edges that are contained in exactly one face. 
A \emph{copy} of a surface $S$ in a tripartite simplicial complex $K$ is the image of a homomorphism $\phi:S\to K$. We call the surface a \emph{disc} if it is homeomorphic to a disc, and a \emph{sphere} if it is homeomorphic to the 2-sphere. The \emph{area} of a surface is the number of faces it contains, and the \emph{length} of the boundary of a surface is the number of edges it contains.
\end{definition}

It may seem strange not to insist that a copy of a surface is \emph{iso}morphic to the surface itself, but for the purposes of counting it is convenient to allow a small proportion of the copies to be degenerate. 

Note that the van Kampen complex of a linear tripartite 3-uniform hypergraph is a tripartite simplicial complex with the additional properties that it has just one vertex in each vertex class, and that no two edges are contained in more than one face. We call such a complex \emph{linear} as well. {As the next (almost trivial) proposition shows, every linear tripartite simplicial complex arises in this way, so there is a simple equivalence between linear tripartite 3-uniform hypergraphs and linear tripartite simplicial complexes.}

\begin{proposition}
{Every linear tripartite simplicial complex is the van Kampen complex of a linear tripartite 3-uniform hypergraph.}
\end{proposition}

\begin{proof}
{Let $K$ be a linear tripartite simplicial complex with vertices $u,v$, and $w$, and let $X$ be the set of edges from $u$ to $v$, $Y$ the set of edges from $v$ to $w$, and $Z$ the set of edges from $u$ to $w$. Let $H$ be the hypergraph with vertex sets $X,Y$ and $Z$, where $(x,y,z)$ is a face of $H$ if and only if the edges $x,y$ and $z$ bound a face of $K$. Then $H$ is clearly tripartite. It is also linear, since the condition that $K$ is linear is precisely the condition that no two vertices of $H$ are contained in more than one face.}
\end{proof}

The main statement we wish to prove, which for reasons sketched in the introduction (and explained in more detail in Section \ref{overview}) will imply our main theorem, is the following.

\begin{theorem}\label{nosmallslitspheres}
Let $b$ be a positive integer, let $K$ be linear tripartite simplicial complex $K$ with with $n$ edges joining each pair of vertices, and suppose that $K$ contains at least $\e n^5$ octahedra. Then $K$ contains a subcomplex $L$ with at least $c(\e)n^2$ faces such that $L$ does not contain a copy of any disc with area less than $b$ and boundary of length 2.
\end{theorem}

{
One particular kind of surface plays an important role in our arguments.

\begin{definition}\label{2rdisks}
	A \emph{$2r$-gon} is the disc with boundary of length $2r$ formed by $2r$ triangles that share a central vertex.
\end{definition}

A $2r$-gon corresponds to a $2r$-cycle in a partial Latin square, as shown in Figure~\ref{newfig2}.

}

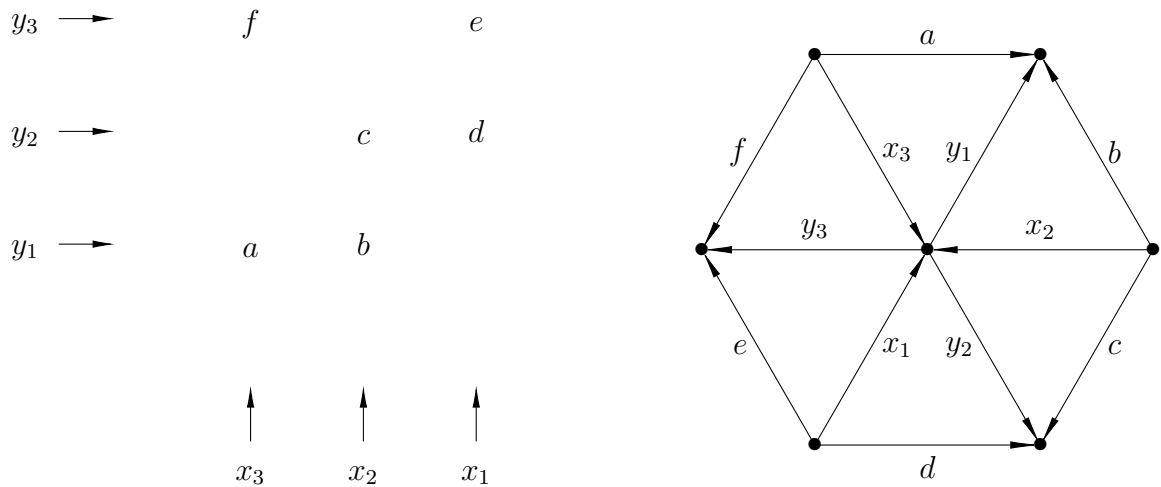
\begin{figure}
	\centering
	\begin{tikzpicture}[scale=1.5, every node/.style={scale=1.0}]
	
	\draw (-6,0) node {$a$};
	\draw (-5,0.05) node {$b$};
	\draw (-5,1) node {$c$};
	\draw (-4,1.05) node {$d$};
	\draw (-4,2) node {$e$};
	\draw (-6,2) node {$f$};
	
	\draw (-4,-2) node {$x_1$};
	\draw (-5,-2) node {$x_2$};
	\draw (-6,-2) node {$x_3$};
	
	\draw[-{Latex[length=3mm, width=1mm]}] (-4,-2+0.3) -- (-4,-2+0.8) ;
	\draw[-{Latex[length=3mm, width=1mm]}] (-5,-2+0.3) -- (-5,-2+0.8) ;
	\draw[-{Latex[length=3mm, width=1mm]}] (-6,-2+0.3) -- (-6,-2+0.8) ;
	
	\draw (-8,0.0) node {$y_1$};
	\draw (-8,1.0) node {$y_2$};
	\draw (-8,2.0) node {$y_3$};	
	
	\draw[-{Latex[length=3mm, width=1mm]}] (-8+0.3,0.05) -- (-8+0.8,0.05) ;
	\draw[-{Latex[length=3mm, width=1mm]}] (-8+0.3,1.05) -- (-8+0.8,1.05) ;
	\draw[-{Latex[length=3mm, width=1mm]}] (-8+0.3,2.05) -- (-8+0.8,2.05) ;
	
	\draw (-2,0) node {$\bullet$};
	\draw (2,0) node {$\bullet$};
	\draw (-1,1.732) node {$\bullet$};
	\draw (1,1.732) node {$\bullet$};
	\draw (-1,-1.732) node {$\bullet$};
	\draw (1,-1.732) node {$\bullet$};
	
	\draw (0,0) node {$\bullet$};
	
	\draw[-{Latex[length=4mm, width=1mm]}] (0,0) -- (-2,0) node[midway, above] {$y_3$};
	\draw[-{Latex[length=4mm, width=1mm]}] (2,0) -- (0,0) node[midway, above] {$x_2$};
	
	\draw[-{Latex[length=4mm, width=1mm]}] (-1,-1.732) -- (0,0) node[midway, right] {$x_1$};
	\draw[-{Latex[length=4mm, width=1mm]}] (0,0) -- (1,1.732) node[midway, left] {$y_1$};
	
	\draw[-{Latex[length=4mm, width=1mm]}] (-1,1.732) -- (0,0) node[midway, right] {$x_3$};
	\draw[-{Latex[length=4mm, width=1mm]}] (0,0) -- (1,-1.732) node[midway, left] {$y_2$};
	
	\draw[-{Latex[length=4mm, width=1mm]}] (-1,1.732) -- (1,1.732) node[midway, above] {$a$} ;
	\draw[{Latex[length=4mm, width=1mm]}-] (1,1.732) -- (2,0) node[midway, right] {$b$} ;
	\draw[-{Latex[length=4mm, width=1mm]}] (2,0) -- (1,-1.732) node[midway, right] {$c$} ;
	\draw[{Latex[length=4mm, width=1mm]}-] (1,-1.732) -- (-1,-1.732) node[midway, below] {$d$} ;
	\draw[-{Latex[length=4mm, width=1mm]}] (-1,-1.732) -- (-2,0) node[midway, left] {$e$};
	\draw[{Latex[length=4mm, width=1mm]}-] (-2,0) -- (-1,1.732) node[midway, left] {$f$};
	
%	\draw[-{Latex[length=4mm, width=1mm]}] (0,0) -- (0,4) ;

	\end{tikzpicture}
	\caption{A label 6-cycle in a partial Latin square is shown on the left, with the corresponding rows and columns shown. The corresponding 6-gon (see Definition~\ref{2rdisks}) is shown on the right. A copy of this disc would appear in the van Kampen complex of the partial Latin square.}\label{newfig2}
\end{figure}

We can now give the geometrical interpretation of a dispersed ring decomposition from Definition~\ref{shatteredring} that we promised earlier. 
A dispersed ring decomposition of a $2r$-cycle is obtained by starting with a prism, the two ends of which are $2r$-gons. Each face of this prism (that is, two $2r$-gons and $2r$ rectangles) is triangulated by joining a central point to every vertex. That is, the dispersed ring decomposition contains two $2r$-gons, one at each end of the prism, with $2r$ 4-gons joining them. One of the $2r$-gons corresponds to to the $2r$-cycle that has been decomposed, and the rest of the prism is a more complicated triangulation of the disc that has the same boundary, which is obtained by taking another $2r$-gon and surrounding it with $2r$ 4-gons. Figure \ref{pa:fig:shatteredPFvK} below shows this more complicated triangulation in the case $r=2$: that is, it illustrates the disc that comes from a dispersed ring decomposition of a 4-cycle. Note that in this case the prism is a cube. 

\begin{figure}
	\centering
	\begin{tikzpicture}[scale=3.0, every node/.style={scale=1.0}]
	
	\draw (0,0) node {$\bullet$};
	\draw (0,4) node {$\bullet$};
	\draw (4,0) node {$\bullet$};
	\draw (4,4) node {$\bullet$};
	
	\draw (1,1) node {$\bullet$};
	\draw (1,3) node {$\bullet$};
	\draw (3,3) node {$\bullet$};
	\draw (3,1) node {$\bullet$};
	
	\draw (2,3.5) node {$\bullet$};
	\draw (0.5,2) node {$\bullet$};
	\draw (2,0.5) node {$\bullet$};
	\draw (3.5,2) node {$\bullet$};
	\draw (2,2) node {$\bullet$};
	
	\draw[-{Latex[length=4mm, width=1mm]}] (0,0) -- (0,4) ;
	\draw[-{Latex[length=4mm, width=1mm]}] (0,0) -- (4,0) ;
	\draw[{Latex[length=4mm, width=1mm]}-] (4,0) -- (4,4) ;
	\draw[{Latex[length=4mm, width=1mm]}-] (0,4) -- (4,4) ;
	
	\draw[-{Latex[length=4mm, width=1mm]}] (0,0) -- (1,1) ;
	\draw[-{Latex[length=4mm, width=1mm]}] (1,3) -- (0,4) ;
	\draw[{Latex[length=4mm, width=1mm]}-] (4,0) -- (3,1) ;
	\draw[{Latex[length=4mm, width=1mm]}-] (3,3) -- (4,4) ;
	
	\draw[-{Latex[length=4mm, width=1mm]}] (0,0) -- (2,0.5) ;
	\draw[{Latex[length=4mm, width=1mm]}-] (1,1) -- (2,0.5) ;
	\draw[-{Latex[length=4mm, width=1mm]}] (3,1) -- (2,0.5) ;
	\draw[{Latex[length=4mm, width=1mm]}-] (4,0) -- (2,0.5) ;
	
	\draw[{Latex[length=4mm, width=1mm]}-] (1,1) -- (2,2) ;
	\draw[-{Latex[length=4mm, width=1mm]}] (1,3) -- (2,2) ;
	\draw[{Latex[length=4mm, width=1mm]}-] (3,3) -- (2,2) ;
	\draw[-{Latex[length=4mm, width=1mm]}] (3,1) -- (2,2) ;
	
	\draw[{Latex[length=4mm, width=1mm]}-] (0,4) -- (0.5,2) ;
	\draw[-{Latex[length=4mm, width=1mm]}] (1,3) -- (0.5,2) ;
	\draw[{Latex[length=4mm, width=1mm]}-] (1,1) -- (0.5,2) ;
	\draw[-{Latex[length=4mm, width=1mm]}] (0,0) -- (0.5,2) ;
	
	\draw[{Latex[length=4mm, width=1mm]}-] (0,4) -- (2,3.5) ;
	\draw[-{Latex[length=4mm, width=1mm]}] (4,4) -- (2,3.5) ;
	\draw[{Latex[length=4mm, width=1mm]}-] (3,3) -- (2,3.5) ;
	\draw[-{Latex[length=4mm, width=1mm]}] (1,3) -- (2,3.5) ;
	
	\draw[{Latex[length=4mm, width=1mm]}-] (4,0) -- (3.5,2) ;
	\draw[-{Latex[length=4mm, width=1mm]}] (3,1) -- (3.5,2) ;
	\draw[{Latex[length=4mm, width=1mm]}-] (3,3) -- (3.5,2) ;
	\draw[-{Latex[length=4mm, width=1.5mm]}] (4,4) -- (3.5,2) ;
	
	\draw[{Latex[length=4mm, width=1mm]}-] (1,1) -- (1,3) ;
	\draw[{Latex[length=4mm, width=1mm]}-] (1,1) -- (3,1) ;
	\draw[-{Latex[length=4mm, width=1mm]}] (3,1) -- (3,3) ;
	\draw[-{Latex[length=4mm, width=1mm]}] (1,3) -- (3,3) ;

	\end{tikzpicture}
	\caption{The disc corresponding to the dispersed ring decomposition of a 4-cycle. We have omitted the labels on the edges.}\label{pa:fig:shatteredPFvK}
\end{figure}
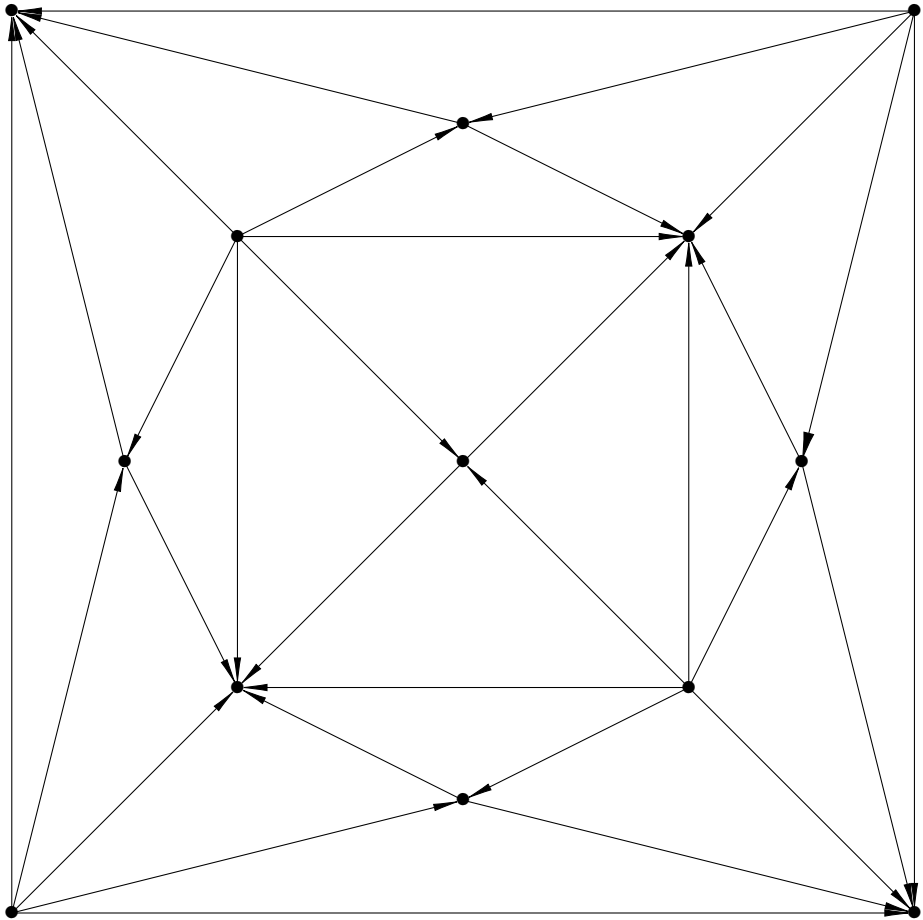

	For convenience, we encapsulate the above remarks in the following definition, which translates the notion of a dispersed ring decomposition (see Definition~\ref{shatteredring}) into the language of tripartite surfaces.
	
\begin{definition}\label{2rprisms}
	Given a tripartite simplicial complex $K$ and a copy $D$ of a $2r$-gon in $K$, a \emph{dispersed ring decomposition of $D$ in $K$} is a copy $D'$ in $K$ of the disc obtained by surrounding a $2r$-gon with $2r$ 4-gons as described above to give a triangulation of the disc, in which the boundary of $D'$ is the same as the boundary of $D'$.
\end{definition}

\iftrue
\else

Although we have formulated this definition in grid terms, referring to cycles and labels, it has a natural description in hypergraph terms.

\begin{definition}
	Let $F$ be a $2r$-PF. A \emph{shattered ring decomposition} of $F$ consists of a second $2r$-PF $F'$ with petals in the same vertex class, together with $2r$ 4-PFs, each of has a petal equal to a petal of $F$ and its opposite petal equal to the corresponding petal of $F'$, and each of which shares a petal with its predecessor and a petal with its successor, in such a way that the assignment of vertex classes to the inner vertices of each 4-PF is the reflection of the assigment of classes to its predecessor.
\end{definition}

The hypergraph forms of shattered ring decompositions of a 4-PF, 6-PF and 8-PF are shown in Figure~\ref{pa:fig0} (with the 4-PF, 6-PF and 8-PF not drawn -- their petals will coincide with the degree-1 vertices in the diagrams).

\begin{figure}
	\centering
	\begin{subfigure}{0.33\linewidth}
		\centering
		\begin{tikzpicture}[scale=0.2, every node/.style={scale=0.45}]
		% inner square pyramid net
		\draw (1,1) rectangle (-1,-1);
		\draw (1,1) -- (3,0) -- (1,-1);
		\draw (1,-1) -- (0,-3) -- (-1,-1);
		\draw (-1,-1) -- (-3,0) -- (-1,1);
		\draw (-1,1) -- (0,3) -- (1,1);
		
		% outer ring of pyramid nets
		\draw (-1,5) rectangle (1,7);
		\draw (1,5) -- (0,3) -- (-1,5);
		\draw (1,5) -- (6,6) -- (1,7);
		\draw (1,7) -- (0,9) -- (-1,7);
		\draw (-1,7) -- (-6,6) -- (-1,5);
		
		\draw (5,1) rectangle (7,-1);
		\draw (5,-1) -- (3,0) -- (5,1);
		\draw (5,1) -- (6,6) -- (7,1);
		\draw (7,1) -- (9,0) -- (7,-1);
		\draw (7,-1) -- (6,-6) -- (5,-1);
		
		\draw (-1,-5) rectangle (1,-7);
		\draw (-1,-5) -- (0,-3) -- (1,-5);
		\draw (1,-5) -- (6,-6) -- (1,-7);
		\draw (1,-7) -- (0,-9) -- (-1,-7);
		\draw (-1,-7) -- (-6,-6) -- (-1,-5);
		
		\draw (-7,-1) rectangle (-5,1);
		\draw (-5,1) -- (-3,0) -- (-5,-1);
		\draw (-5,-1) -- (-6,-6) -- (-7,-1);
		\draw (-7,-1) -- (-9,0) -- (-7,1);
		\draw (-7,1) -- (-6,6) -- (-5,1);
		
		% shading
		\draw[fill,gray,opacity=0.2] (1,1) -- (-1,1) -- (0,3) -- cycle;
		\draw[fill,gray,opacity=0.2] (1,1) -- (1,-1) -- (3,0) -- cycle;
		\draw[fill,gray,opacity=0.2] (-1,-1) -- (1,-1) -- (0,-3) -- cycle;
		\draw[fill,gray,opacity=0.2] (-1,-1) -- (-1,1) -- (-3,0) -- cycle;
		\draw[fill,gray,opacity=0.2] (0,3) -- (1,5) -- (-1,5) -- cycle;
		\draw[fill,gray,opacity=0.2] (3,0) -- (5,1) -- (5,-1) -- cycle;
		\draw[fill,gray,opacity=0.2] (0,-3) -- (1,-5) -- (-1,-5) -- cycle;
		\draw[fill,gray,opacity=0.2] (-3,0) -- (-5,1) -- (-5,-1) -- cycle;
		\draw[fill,gray,opacity=0.2] (1,5) -- (6,6) -- (1,7) -- cycle;
		\draw[fill,gray,opacity=0.2] (-1,5) -- (-1,7) -- (-6,6) -- cycle;
		\draw[fill,gray,opacity=0.2] (-1,7) -- (1,7) -- (0,9) -- cycle;
		\draw[fill,gray,opacity=0.2] (5,1) -- (7,1) -- (6,6) -- cycle;
		\draw[fill,gray,opacity=0.2] (5,-1) -- (7,-1) -- (6,-6) -- cycle;
		\draw[fill,gray,opacity=0.2] (7,1) -- (7,-1) -- (9,0) -- cycle;
		\draw[fill,gray,opacity=0.2] (1,-5) -- (1,-7) -- (6,-6) -- cycle;
		\draw[fill,gray,opacity=0.2] (-1,-5) -- (-1,-7) -- (-6,-6) -- cycle;
		\draw[fill,gray,opacity=0.2] (-1,-7) -- (1,-7) -- (0,-9) -- cycle;
		\draw[fill,gray,opacity=0.2] (-5,1) -- (-7,1) -- (-6,6) -- cycle;
		\draw[fill,gray,opacity=0.2] (-5,-1) -- (-7,-1) -- (-6,-6) -- cycle;
		\draw[fill,gray,opacity=0.2] (-7,1) -- (-7,-1) -- (-9,0) -- cycle;
		\end{tikzpicture}
		\caption{} %A shattered ring decomposition of a 4-PF
	\end{subfigure}
	\begin{subfigure}{0.33\linewidth}
		\centering
		%\raggedright % aligns first picture to left. \centering for centre.
		\begin{tikzpicture}[scale=0.24, every node/.style={scale=0.9}]
		\providecommand* \angle {30} % no idea what this line does
		\coordinate[](O) at (0,0);
		\coordinate[](A) at (-1,1.732);
		\coordinate[](B) at (1,1.732);
		\coordinate[](C) at (0,3.464);
		\coordinate[](D) at (-1,5.196);
		\coordinate[](E) at (1,5.196);
		\coordinate[](F) at (1,7.196);
		\coordinate[](G) at (-1,7.196);
		\coordinate[](H) at (0,8.928);
		\coordinate[](I) at (3.5,6.06);
		\coordinate[](J) at (-3.5,6.06);
		\draw (A) -- (B) -- (C) -- cycle;
		\draw (C) -- (D) -- (E) -- cycle;
		\draw (D) -- (E) -- (F) -- (G) -- cycle;
		\draw (F) -- (G) -- (H) -- cycle;
		\draw (E) -- (I);
		\draw (F) -- (I);
		\draw (D) -- (J);
		\draw (G) -- (J);
		
		% shading
		\draw[fill,gray,opacity=0.2] (A) -- (B) -- (C) -- cycle;
		\draw[fill,gray,opacity=0.2] (C) -- (D) -- (E) -- cycle;
		\draw[fill,gray,opacity=0.2] (D) -- (G) -- (J) -- cycle;
		\draw[fill,gray,opacity=0.2] (E) -- (F) -- (I);
		\draw[fill,gray,opacity=0.2] (F) -- (G) -- (H);
		
		\begin{scope}[rotate around={60:(O)}]
		\coordinate[](O) at (0,0);
		\coordinate[](A) at (-1,1.732);
		\coordinate[](B) at (1,1.732);
		\coordinate[](C) at (0,3.464);
		\coordinate[](D) at (-1,5.196);
		\coordinate[](E) at (1,5.196);
		\coordinate[](F) at (1,7.196);
		\coordinate[](G) at (-1,7.196);
		\coordinate[](H) at (0,8.928);
		\coordinate[](I) at (3.5,6.06);
		\coordinate[](J) at (-3.5,6.06);
		\draw (A) -- (B) -- (C) -- cycle;
		\draw (C) -- (D) -- (E) -- cycle;
		\draw (D) -- (E) -- (F) -- (G) -- cycle;
		\draw (F) -- (G) -- (H) -- cycle;
		\draw (E) -- (I);
		\draw (F) -- (I);
		\draw (D) -- (J);
		\draw (G) -- (J);
		% shading
		\draw[fill,gray,opacity=0.2] (A) -- (B) -- (C) -- cycle;
		\draw[fill,gray,opacity=0.2] (C) -- (D) -- (E) -- cycle;
		\draw[fill,gray,opacity=0.2] (D) -- (G) -- (J) -- cycle;
		\draw[fill,gray,opacity=0.2] (E) -- (F) -- (I);
		\draw[fill,gray,opacity=0.2] (F) -- (G) -- (H);
		\end{scope}
		
		\begin{scope}[rotate around={120:(O)}]
		\coordinate[](O) at (0,0);
		\coordinate[](A) at (-1,1.732);
		\coordinate[](B) at (1,1.732);
		\coordinate[](C) at (0,3.464);
		\coordinate[](D) at (-1,5.196);
		\coordinate[](E) at (1,5.196);
		\coordinate[](F) at (1,7.196);
		\coordinate[](G) at (-1,7.196);
		\coordinate[](H) at (0,8.928);
		%\draw (0,8.928) node[] {$\bullet$};
		\coordinate[](I) at (3.5,6.06);
		\coordinate[](J) at (-3.5,6.06);
		\draw (A) -- (B) -- (C) -- cycle;
		\draw (C) -- (D) -- (E) -- cycle;
		\draw (D) -- (E) -- (F) -- (G) -- cycle;
		\draw (F) -- (G) -- (H) -- cycle;
		\draw (E) -- (I);
		\draw (F) -- (I);
		\draw (D) -- (J);
		\draw (G) -- (J);
		% shading
		\draw[fill,gray,opacity=0.2] (A) -- (B) -- (C) -- cycle;
		\draw[fill,gray,opacity=0.2] (C) -- (D) -- (E) -- cycle;
		\draw[fill,gray,opacity=0.2] (D) -- (G) -- (J) -- cycle;
		\draw[fill,gray,opacity=0.2] (E) -- (F) -- (I);
		\draw[fill,gray,opacity=0.2] (F) -- (G) -- (H);
		\end{scope}
		
		\begin{scope}[rotate around={180:(O)}]
		\coordinate[](O) at (0,0);
		\coordinate[](A) at (-1,1.732);
		\coordinate[](B) at (1,1.732);
		\coordinate[](C) at (0,3.464);
		\coordinate[](D) at (-1,5.196);
		\coordinate[](E) at (1,5.196);
		\coordinate[](F) at (1,7.196);
		\coordinate[](G) at (-1,7.196);
		\coordinate[](H) at (0,8.928);
		%\draw (0,8.928) node[] {$\bullet$};
		\coordinate[](I) at (3.5,6.06);
		\coordinate[](J) at (-3.5,6.06);
		\draw (A) -- (B) -- (C) -- cycle;
		\draw (C) -- (D) -- (E) -- cycle;
		\draw (D) -- (E) -- (F) -- (G) -- cycle;
		\draw (F) -- (G) -- (H) -- cycle;
		\draw (E) -- (I);
		\draw (F) -- (I);
		\draw (D) -- (J);
		\draw (G) -- (J);
		% shading
		\draw[fill,gray,opacity=0.2] (A) -- (B) -- (C) -- cycle;
		\draw[fill,gray,opacity=0.2] (C) -- (D) -- (E) -- cycle;
		\draw[fill,gray,opacity=0.2] (D) -- (G) -- (J) -- cycle;
		\draw[fill,gray,opacity=0.2] (E) -- (F) -- (I);
		\draw[fill,gray,opacity=0.2] (F) -- (G) -- (H);
		\end{scope}
		
		\begin{scope}[rotate around={240:(O)}]
		\coordinate[](O) at (0,0);
		\coordinate[](A) at (-1,1.732);
		\coordinate[](B) at (1,1.732);
		\coordinate[](C) at (0,3.464);
		\coordinate[](D) at (-1,5.196);
		\coordinate[](E) at (1,5.196);
		\coordinate[](F) at (1,7.196);
		\coordinate[](G) at (-1,7.196);
		\coordinate[](H) at (0,8.928);
		%\draw (0,8.928) node[] {$\bullet$};
		\coordinate[](I) at (3.5,6.06);
		\coordinate[](J) at (-3.5,6.06);
		\draw (A) -- (B) -- (C) -- cycle;
		\draw (C) -- (D) -- (E) -- cycle;
		\draw (D) -- (E) -- (F) -- (G) -- cycle;
		\draw (F) -- (G) -- (H) -- cycle;
		\draw (E) -- (I);
		\draw (F) -- (I);
		\draw (D) -- (J);
		\draw (G) -- (J);
		% shading
		\draw[fill,gray,opacity=0.2] (A) -- (B) -- (C) -- cycle;
		\draw[fill,gray,opacity=0.2] (C) -- (D) -- (E) -- cycle;
		\draw[fill,gray,opacity=0.2] (D) -- (G) -- (J) -- cycle;
		\draw[fill,gray,opacity=0.2] (E) -- (F) -- (I);
		\draw[fill,gray,opacity=0.2] (F) -- (G) -- (H);
		\end{scope}
		
		\begin{scope}[rotate around={300:(O)}]
		\coordinate[](O) at (0,0);
		\coordinate[](A) at (-1,1.732);
		\coordinate[](B) at (1,1.732);
		\coordinate[](C) at (0,3.464);
		\coordinate[](D) at (-1,5.196);
		\coordinate[](E) at (1,5.196);
		\coordinate[](F) at (1,7.196);
		\coordinate[](G) at (-1,7.196);
		\coordinate[](H) at (0,8.928);
		%\draw (0,8.928) node[] {$\bullet$};
		\coordinate[](I) at (3.5,6.06);
		\coordinate[](J) at (-3.5,6.06);
		\draw (A) -- (B) -- (C) -- cycle;
		\draw (C) -- (D) -- (E) -- cycle;
		\draw (D) -- (E) -- (F) -- (G) -- cycle;
		\draw (F) -- (G) -- (H) -- cycle;
		\draw (E) -- (I);
		\draw (F) -- (I);
		\draw (D) -- (J);
		\draw (G) -- (J);
		% shading
		\draw[fill,gray,opacity=0.2] (A) -- (B) -- (C) -- cycle;
		\draw[fill,gray,opacity=0.2] (C) -- (D) -- (E) -- cycle;
		\draw[fill,gray,opacity=0.2] (D) -- (G) -- (J) -- cycle;
		\draw[fill,gray,opacity=0.2] (E) -- (F) -- (I);
		\draw[fill,gray,opacity=0.2] (F) -- (G) -- (H);
		\end{scope}
		
		\end{tikzpicture}
		\caption{} %A shattered ring decomposition of a 6-PF
	\end{subfigure}% comment out space to have subfigures side by side
	\begin{subfigure}{0.33\linewidth}
		\centering
		%\raggedleft %aligns second picture to the right. \centering for centre.
		
		%---------------------------------------------------
		% SHATTERED 6PF
		%---------------------------------------------------
		\begin{tikzpicture}[scale=0.18, every node/.style={scale=0.9}]
		\providecommand* \angle {30}
		\coordinate[](O) at (0,0);
		\coordinate[](A) at (-2.828,2.828);
		\coordinate[](B) at (2.828,2.828);
		\coordinate[](C) at (0,5.657);
		\coordinate[](D) at (-1,7.389);
		\coordinate[](E) at (1,7.389);
		\coordinate[](F) at (1,9.389);
		\coordinate[](G) at (-1,9.389);
		\coordinate[](H) at (3.45,8.389);
		\coordinate[](I) at (0,11.121);
		%\draw (0,11.121) node[] {$\bullet$};
		\coordinate[](J) at (-3.45,8.389);
		\draw (A) -- (C);
		\draw (B) -- (C);
		\draw (C) -- (D) -- (E) -- cycle;
		\draw (D) -- (E) -- (F) -- (G) -- cycle;
		\draw (E) -- (F) -- (H) -- cycle;
		\draw (F) -- (G) -- (I) -- cycle;
		\draw (D) -- (G) -- (J) -- cycle;
		% shading
		\draw[fill,gray,opacity=0.2] (-1.7,4.0) -- (1.7,4.0) -- (C) -- cycle;
		\draw[fill,gray,opacity=0.2] (C) -- (D) -- (E) -- cycle;
		\draw[fill,gray,opacity=0.2] (D) -- (G) -- (J) -- cycle;
		\draw[fill,gray,opacity=0.2] (E) -- (F) -- (H);
		\draw[fill,gray,opacity=0.2] (F) -- (G) -- (I);
		
		\begin{scope}[rotate around={45:(O)}]
		\coordinate[](O) at (0,0);
		\coordinate[](A) at (-2.828,2.828);
		\coordinate[](B) at (2.828,2.828);
		\coordinate[](C) at (0,5.657);
		\coordinate[](D) at (-1,7.389);
		\coordinate[](E) at (1,7.389);
		\coordinate[](F) at (1,9.389);
		\coordinate[](G) at (-1,9.389);
		\coordinate[](H) at (3.45,8.389);
		\coordinate[](I) at (0,11.121);
		%\draw (0,11.121) node[] {$\bullet$};
		\coordinate[](J) at (-3.45,8.389);
		\draw (A) -- (C);
		\draw (B) -- (C);
		\draw (C) -- (D) -- (E) -- cycle;
		\draw (D) -- (E) -- (F) -- (G) -- cycle;
		\draw (E) -- (F) -- (H) -- cycle;
		\draw (F) -- (G) -- (I) -- cycle;
		\draw (D) -- (G) -- (J) -- cycle;
		% shading
		\draw[fill,gray,opacity=0.2] (-1.7,4.0) -- (1.7,4.0) -- (C) -- cycle;
		\draw[fill,gray,opacity=0.2] (C) -- (D) -- (E) -- cycle;
		\draw[fill,gray,opacity=0.2] (D) -- (G) -- (J) -- cycle;
		\draw[fill,gray,opacity=0.2] (E) -- (F) -- (H);
		\draw[fill,gray,opacity=0.2] (F) -- (G) -- (I);
		\end{scope}
		
		\begin{scope}[rotate around={90:(O)}]
		\coordinate[](O) at (0,0);
		\coordinate[](A) at (-2.828,2.828);
		\coordinate[](B) at (2.828,2.828);
		\coordinate[](C) at (0,5.657);
		\coordinate[](D) at (-1,7.389);
		\coordinate[](E) at (1,7.389);
		\coordinate[](F) at (1,9.389);
		\coordinate[](G) at (-1,9.389);
		\coordinate[](H) at (3.45,8.389);
		\coordinate[](I) at (0,11.121);
		%\draw (0,11.121) node[] {$\bullet$};
		\coordinate[](J) at (-3.45,8.389);
		\draw (A) -- (C);
		\draw (B) -- (C);
		\draw (C) -- (D) -- (E) -- cycle;
		\draw (D) -- (E) -- (F) -- (G) -- cycle;
		\draw (E) -- (F) -- (H) -- cycle;
		\draw (F) -- (G) -- (I) -- cycle;
		\draw (D) -- (G) -- (J) -- cycle;
		% shading
		\draw[fill,gray,opacity=0.2] (-1.7,4.0) -- (1.7,4.0) -- (C) -- cycle;
		\draw[fill,gray,opacity=0.2] (C) -- (D) -- (E) -- cycle;
		\draw[fill,gray,opacity=0.2] (D) -- (G) -- (J) -- cycle;
		\draw[fill,gray,opacity=0.2] (E) -- (F) -- (H);
		\draw[fill,gray,opacity=0.2] (F) -- (G) -- (I);
		\end{scope}
		
		\begin{scope}[rotate around={135:(O)}]
		\coordinate[](O) at (0,0);
		\coordinate[](A) at (-2.828,2.828);
		\coordinate[](B) at (2.828,2.828);
		\coordinate[](C) at (0,5.657);
		\coordinate[](D) at (-1,7.389);
		\coordinate[](E) at (1,7.389);
		\coordinate[](F) at (1,9.389);
		\coordinate[](G) at (-1,9.389);
		\coordinate[](H) at (3.45,8.389);
		\coordinate[](I) at (0,11.121);
		%\draw (0,11.121) node[] {$\bullet$};
		\coordinate[](J) at (-3.45,8.389);
		\draw (A) -- (C);
		\draw (B) -- (C);
		\draw (C) -- (D) -- (E) -- cycle;
		\draw (D) -- (E) -- (F) -- (G) -- cycle;
		\draw (E) -- (F) -- (H) -- cycle;
		\draw (F) -- (G) -- (I) -- cycle;
		\draw (D) -- (G) -- (J) -- cycle;
		% shading
		\draw[fill,gray,opacity=0.2] (-1.7,4.0) -- (1.7,4.0) -- (C) -- cycle;
		\draw[fill,gray,opacity=0.2] (C) -- (D) -- (E) -- cycle;
		\draw[fill,gray,opacity=0.2] (D) -- (G) -- (J) -- cycle;
		\draw[fill,gray,opacity=0.2] (E) -- (F) -- (H);
		\draw[fill,gray,opacity=0.2] (F) -- (G) -- (I);
		\end{scope}
		
		\begin{scope}[rotate around={180:(O)}]
		\coordinate[](O) at (0,0);
		\coordinate[](A) at (-2.828,2.828);
		\coordinate[](B) at (2.828,2.828);
		\coordinate[](C) at (0,5.657);
		\coordinate[](D) at (-1,7.389);
		\coordinate[](E) at (1,7.389);
		\coordinate[](F) at (1,9.389);
		\coordinate[](G) at (-1,9.389);
		\coordinate[](H) at (3.45,8.389);
		\coordinate[](I) at (0,11.121);
		%\draw (0,11.121) node[] {$\bullet$};
		\coordinate[](J) at (-3.45,8.389);
		\draw (A) -- (C);
		\draw (B) -- (C);
		\draw (C) -- (D) -- (E) -- cycle;
		\draw (D) -- (E) -- (F) -- (G) -- cycle;
		\draw (E) -- (F) -- (H) -- cycle;
		\draw (F) -- (G) -- (I) -- cycle;
		\draw (D) -- (G) -- (J) -- cycle;
		% shading
		\draw[fill,gray,opacity=0.2] (-1.7,4.0) -- (1.7,4.0) -- (C) -- cycle;
		\draw[fill,gray,opacity=0.2] (C) -- (D) -- (E) -- cycle;
		\draw[fill,gray,opacity=0.2] (D) -- (G) -- (J) -- cycle;
		\draw[fill,gray,opacity=0.2] (E) -- (F) -- (H);
		\draw[fill,gray,opacity=0.2] (F) -- (G) -- (I);
		\end{scope}
		
		\begin{scope}[rotate around={225:(O)}]
		\coordinate[](O) at (0,0);
		\coordinate[](A) at (-2.828,2.828);
		\coordinate[](B) at (2.828,2.828);
		\coordinate[](C) at (0,5.657);
		\coordinate[](D) at (-1,7.389);
		\coordinate[](E) at (1,7.389);
		\coordinate[](F) at (1,9.389);
		\coordinate[](G) at (-1,9.389);
		\coordinate[](H) at (3.45,8.389);
		\coordinate[](I) at (0,11.121);
		%\draw (0,11.121) node[] {$\bullet$};
		\coordinate[](J) at (-3.45,8.389);
		\draw (A) -- (C);
		\draw (B) -- (C);
		\draw (C) -- (D) -- (E) -- cycle;
		\draw (D) -- (E) -- (F) -- (G) -- cycle;
		\draw (E) -- (F) -- (H) -- cycle;
		\draw (F) -- (G) -- (I) -- cycle;
		\draw (D) -- (G) -- (J) -- cycle;
		% shading
		\draw[fill,gray,opacity=0.2] (-1.7,4.0) -- (1.7,4.0) -- (C) -- cycle;
		\draw[fill,gray,opacity=0.2] (C) -- (D) -- (E) -- cycle;
		\draw[fill,gray,opacity=0.2] (D) -- (G) -- (J) -- cycle;
		\draw[fill,gray,opacity=0.2] (E) -- (F) -- (H);
		\draw[fill,gray,opacity=0.2] (F) -- (G) -- (I);
		\end{scope}
		
		\begin{scope}[rotate around={270:(O)}]
		\coordinate[](O) at (0,0);
		\coordinate[](A) at (-2.828,2.828);
		\coordinate[](B) at (2.828,2.828);
		\coordinate[](C) at (0,5.657);
		\coordinate[](D) at (-1,7.389);
		\coordinate[](E) at (1,7.389);
		\coordinate[](F) at (1,9.389);
		\coordinate[](G) at (-1,9.389);
		\coordinate[](H) at (3.45,8.389);
		\coordinate[](I) at (0,11.121);
		%\draw (0,11.121) node[] {$\bullet$};
		\coordinate[](J) at (-3.45,8.389);
		\draw (A) -- (C);
		\draw (B) -- (C);
		\draw (C) -- (D) -- (E) -- cycle;
		\draw (D) -- (E) -- (F) -- (G) -- cycle;
		\draw (E) -- (F) -- (H) -- cycle;
		\draw (F) -- (G) -- (I) -- cycle;
		\draw (D) -- (G) -- (J) -- cycle;
		% shading
		\draw[fill,gray,opacity=0.2] (-1.7,4.0) -- (1.7,4.0) -- (C) -- cycle;
		\draw[fill,gray,opacity=0.2] (C) -- (D) -- (E) -- cycle;
		\draw[fill,gray,opacity=0.2] (D) -- (G) -- (J) -- cycle;
		\draw[fill,gray,opacity=0.2] (E) -- (F) -- (H);
		\draw[fill,gray,opacity=0.2] (F) -- (G) -- (I);
		\end{scope}
		
		\begin{scope}[rotate around={315:(O)}]
		\coordinate[](O) at (0,0);
		\coordinate[](A) at (-2.828,2.828);
		\coordinate[](B) at (2.828,2.828);
		\coordinate[](C) at (0,5.657);
		\coordinate[](D) at (-1,7.389);
		\coordinate[](E) at (1,7.389);
		\coordinate[](F) at (1,9.389);
		\coordinate[](G) at (-1,9.389);
		\coordinate[](H) at (3.45,8.389);
		\coordinate[](I) at (0,11.121);
		%\draw (0,11.121) node[] {$\bullet$};
		\coordinate[](J) at (-3.45,8.389);
		\draw (A) -- (C);
		\draw (B) -- (C);
		\draw (C) -- (D) -- (E) -- cycle;
		\draw (D) -- (E) -- (F) -- (G) -- cycle;
		\draw (E) -- (F) -- (H) -- cycle;
		\draw (F) -- (G) -- (I) -- cycle;
		\draw (D) -- (G) -- (J) -- cycle;
		% shading
		\draw[fill,gray,opacity=0.2] (-1.7,4.0) -- (1.7,4.0) -- (C) -- cycle;
		\draw[fill,gray,opacity=0.2] (C) -- (D) -- (E) -- cycle;
		\draw[fill,gray,opacity=0.2] (D) -- (G) -- (J) -- cycle;
		\draw[fill,gray,opacity=0.2] (E) -- (F) -- (H);
		\draw[fill,gray,opacity=0.2] (F) -- (G) -- (I);
		\end{scope}
		\end{tikzpicture}
		
		\caption{} %A shattered ring decomposition of an 8-PF
	\end{subfigure}
	\caption{A shattered ring decomposition of a 4-PF, 6-PF and 8-PF in the hypergraph representation are depicted in (a), (b) and (c) respectively. In each figure the triangles correspond to faces of the 3-uniform hypergraph.}\label{pa:fig0}
\end{figure}

If a $2r$-PF $F$ is $\theta_i(\epsilon,k)$-popularly decomposable in at least $\gamma_i(\epsilon,k) n^{2r}$ different ways, this means that there are at least $\gamma_i(\epsilon,k) n^{2r}$ different ring decompositions of $F$ into PFs that are $\theta_i(\epsilon,k)$-popular. If a $2s$-PF $F'$ is $\theta_i(\epsilon,k)$-popular, this means that there are at least $\theta_i(\epsilon,k)n$ different $2s$-PFs that share all their petals with $F'$. This gives us the following lemma.

\begin{lemma}\label{pa:shatteredcount}
	Let $A$ be a linear tripartite hypergraph with $n$ vertices in each class. Let $F$ be a $2r$-PF which is $\theta_i(\epsilon,k)$-popularly decomposable in $A$ in at least $\gamma_i(\epsilon,k) n^{2r}$ different ways. Then $F$ has at least $\gamma_i\theta_i^{2r+1}n^{4r+1}$ different shattered ring decompositions. 
\end{lemma}

\begin{proof}
	As discussed above, there are at least $\gamma_in^{2r}$ different ring decompositions of $F$ into PFs which are $\theta_i$-popular. Each of these popular PFs can be replaced with one of $\theta_i n$ different PFs sharing petals with the original, giving a total of $(\theta_i n)^{2r+1}$ further choices, from which the result follows.
\end{proof}

Broadly speaking, the arguments in the next section will involve starting with a particular hypergraph $H$ and repeatedly replacing $2r$-PFs in $H$ with shattered ring decompositions. Keeping track of the number of ways these replacements are possible will be achieved using Lemma~\ref{pa:shatteredcount}.

\fi

\section{Popular replacement of discs}\label{popreparg1}

We now turn to the proof of Theorem \ref{nosmallslitspheres}. Since the details will get somewhat involved, it will be instructive to begin with the case concerning what we call \emph{slit octahedra}. These are tripartite surfaces that are isomorphic to the triangulated disc illustrated in Figure~\ref{pa:fig:fovK}, which can be obtained from an octahedron by cutting along one of its edges and opening up the cut. Since these are precisely the discs that show that the quadrangle condition fails, this special case will prove Theorem~\ref{noflappycub}.

The main fact that we shall need to prove is that if we pass to a suitable dense subset, then every copy of a slit octahedron in that subset can be replaced by a certain more complicated triangulated disc with the same boundary in a near-maximal number of ways (where `near-maximal' means within a constant of the trivial maximum). To prove this, we shall start with a fixed copy of the surface, and repeatedly replace $2r$-gons by discs with the same boundary that correspond to dispersed ring decompositions (see Definition~\ref{2rprisms}). This we shall be able to do in many ways, so we regard it as a kind of `unfixing' process, where little by little we unfix vertices in order to convert the original completely fixed slit octahedron into a variable surface, at each stage ensuring that the number of possibilities for the variable surface is within a constant of the trivial maximum, given the points that are still fixed. This idea will be explained in more detail later in the section.

The structure of the section will be as follows. First, we shall translate the main results of Section~\ref{pa:sec3} into the language of {linear tripartite simplicial} complexes introduced in Section~\ref{sec4}. Next, we shall give a more detailed overview of the popular replacement argument. This is followed by a brief section containing a technical lemma in which we determine the maximum possible number of copies of a certain tripartite surface that can appear in a {linear tripartite simplicial} complex. We are then ready to describe the popular replacement argument. We begin with the special case of the slit octahedron, before generalizing the approach to prove Theorem~\ref{nosmallslitspheres}.

\subsection{The surface picture}

As a first step, we reinterpret Theorem~\ref{pa:thm2} as a statement about the van Kampen complex from Definition~\ref{vKcomplex}. Doing so will allow us to work almost entirely with simplical complexes in this section of the proof, {thereby avoiding the need to think about the same object in two different ways at once}. 

Recall that the hypothesis of Theorem~\ref{pa:thm2} involves a 3-uniform, linear hypergraph $A$ that contains at least $\e n^5$ octahedra. The van Kampen complex $K(A)$ is therefore a linear tripartite simplicial complex (see Definition~\ref{surfacedefs}) with $n$ edges joining each pair of vertices and containing at least $\e n^5$ octahedra, where in the van Kampen case an octahedron is simply a (tripartite) triangulated surface isomorphic to an octahedron, when an octahedron is also considered as a triangulated sphere.

Given this hypothesis, the conclusion of Theorem~\ref{pa:thm2} is that there exists a sequence $A=A_0\supset A_1\supset \dots$ such that each $A_i$ has density at least $\alpha_i(\epsilon,k)$, each $A_i$ contains at least $\epsilon_i(\epsilon,k)n^5$ octahedra, and for each $r=2,\dots,k$, every $2r$-cycle in $A_i$ has at least $\gamma_i(\epsilon,k) n^{2r}$ different $\theta_i(\epsilon,k)$-popular ring decompositions in $A_{i-1}$. Each $A_i$ has an associated van Kampen complex $K(A_i)$ which is a subcomplex of $K(A)$, and the theorem tells us that $K(A_i)$ has at least $\alpha_i(\epsilon,k)n^2$ faces. Moreover, $K(A_i)$ contains at least $\epsilon_i(\epsilon,k)n^5$ octahedra.

Recall that a $2r$-cycle corresponds in the surface picture to a $2r$-gon (see Definition~\ref{2rdisks}). It remains to interpret, in the language of tripartite surfaces, the conclusion that for each $r=2,\dots,k$, every $2r$-cycle in $A_i$ has at least $\gamma_i(\epsilon,k) n^{2r}$ different $\theta_i(\epsilon,k)$-popular ring decompositions in $A_{i-1}$. By applying Lemma~\ref{pa:shatteredcount}, it follows that every $2r$-cycle in $A_i$ has at least $\gamma_i(\epsilon,k)\theta_i(\epsilon,k)^{2r+1} n^{4r+1}$ different dispersed ring decompositions (see Definition~\ref{shatteredring}). At the end of Section~\ref{sec4} we discussed the dispersed ring decomposition in terms of tripartite surfaces (see Definition~\ref{2rprisms}), so we are now ready to formulate the van Kampen version of Theorem~\ref{pa:thm2}.

\begin{theorem}\label{surfacethm3.3}
	Let $\e\le 10^{-3}$ and let $k\ge 100$. Let $K$ be an $n\times n\times n$ linear tripartite simplicial complex (see Definition~\ref{surfacedefs}) containing at least $\e n^5$ octahedra. Then there exists a sequence $K=K_0\supset K_1\supset \dots$ of subcomplexes of $K$ such that each $K_i$ has at least $\alpha_i(\epsilon,k)n^2$ faces, at least $\epsilon_i(\epsilon,k)n^5$ octahedra, and also has the property that, for each $r=2,\dots,k$, a given $2r$-gon in $K_i$ has at least $\gamma_i(\epsilon,k)\theta_i(\epsilon,k)^{2r+1} n^{4r+1}$ different dispersed ring decompositions in $K_i$.  
\end{theorem}

\subsection{Overview}\label{overview}

The structure of our proof of Theorem~\ref{nosmallslitspheres} will be as follows. We start with a linear tripartite simplicial complex that contains many octahedra. We then apply Theorem~\ref{surfacethm3.3} to create a sequence $K_0\supset K_1\supset \dots$ of tripartite simplicial complexes with the properties stated. We fix some $s$ and pick some particular small disc $D_0$ with boundary of length 2 (see Definition \ref{surfacedefs}, or for an illustration see Figure \ref{pa:fig:fovK}) that we wish to get rid of, and consider the following auxiliary graph on the edges of the 1-skeleton of $K_{s}$: we join two edges $x$ and $y$ of $K_{s}$ by an edge if there is a copy of $D_0$ in $K_{s}$ such that $x$ and $y$ are the boundary edges of that copy. If the maximum degree of this auxiliary graph is bounded then we may pass to a dense independent set, and then we are done, since this independent set corresponds to a dense set of faces in $K_{s}$ such that no two faces are the boundary faces of a copy of $D_0$, and that implies that there is no copy of $D_0$. (Of course, we need to repeat this argument for all the discs we are trying to eliminate.)  

If the maximum degree is \emph{not} bounded, then we would like to find a contradiction. We are given a vertex of large degree in the auxiliary graph, which corresponds to a face of $K_{s}$ that is contained in many different copies of $D_0$, each with a different `opposite face'. Given one of these discs, we perform our unfixing process. Initially, we say that all edges are \emph{fixed}, meaning that we have specified precisely one copy of $D_0$. We then find a $2r$-gon in this copy and use dispersed ring decompositions guaranteed by Theorem~\ref{surfacethm3.3} to replace it with a new, more complicated disc, which we can do in many different ways. However we do the replacement, $D_0$ turns into a copy of a larger disc $D_1$ that still has a boundary of length 2. The copies of $D_1$ thus obtained lie in $K_{s-1}\supset K_{s}$, and we obtain $\Omega(n^{4r+1})$ of them, the trivial maximum being $n^{4r+1}$. We say that the internal edges in the chosen $2r$-gons are \emph{unfixed}, since they may differ from copy to copy. Note that the number of fixed edges has decreased. 

We may continue this process, choosing at each step a $2r$-gon with some fixed internal edges from $D_i$ and using dispersed ring decompositions to generate a larger collection of copies of a disc $D_{i+1}$ that lies in $K_{s-i-1}$, with fewer fixed edges. If $s$ is chosen sufficiently large relative to the area of $D_0$ then we may proceed until we obtain a collection $\mathcal{D}$ of copies of some disc $D_t$ in which the two boundary edges are fixed but every edge incident to an internal vertex is unfixed. One of the boundary edges corresponds to our initial vertex of high degree in the auxiliary graph. By repeating this process for each choice of neighbour of our chosen vertex from that auxiliary graph, we obtain many different collections of copies of $D_t$, {which all share} one of the two boundary edges. By taking the union of all of these collections, we end up violating the trivial upper bound on the maximum possible number of copies of $D_t$ in {an $n\times n \times n$ linear tripartite simplicial} complex. 

The next sections will expand on the details required for this argument. As promised earlier, we shall begin with a detailed account of the argument when $D_0$ is the `slit octahedron' illustrated in Figure \ref{pa:fig:fovK}, and then we shall tackle the necessary generalizations. Before we embark on this it will be necessary to work out the trivial maximum for the number of copies of a given disc with a given set of fixed edges in a {linear tripartite simplicial} complex. The main task of this section will then be to verify that during the unfixing process, the number of copies we obtain is always within a constant of the appropriate trivial maximum, so that in particular this is the case when we reach the tripartite disc $D_t$ with all non-boundary edges unfixed. This is essential for obtaining our desired contradiction. 

\subsection{The maximum number of copies of a partially fixed disc}

We begin with a definition that formalizes the notion of a partially fixed disc, which will be needed to describe the surfaces $D_i$ that appear part way through the overview described above.

\begin{definition}\label{partialsurface}
Let $K$ be {an $n\times n \times n$ linear tripartite simplicial} complex. We define a \emph{partially fixed disc} in $K$ to be a triple $(D,E,\gamma)$, where $D$ is a disc (see Definition~\ref{surfacedefs}), $E$ is a subset of the edges of $D$, and $\gamma$ is a homomorphism from $E$ to the 1-skeleton of $K$ that respects the tripartition of the vertices of $D$. We call the edges in $E$ \emph{fixed} and the other edges \emph{unfixed}. We call a face \emph{unfixed} if it contains at least one unfixed edge. 

A \emph{copy of} $(D,E,\gamma)$ in $K$ is a copy of $D$ in $K$ that extends $\g$ in the obvious sense. Less formally, it is a copy of $D$ in $K$ for which the images of the fixed edges have to be given by $\gamma$. By the \emph{trivial maximum number of copies} of a partially fixed disc $(D,E,\gamma)$ we mean the maximum possible number of copies of a partially fixed disc $(D,E,\gamma')$ in {an $n\times n \times n$ linear tripartite simplicial} complex $K$. Since the trivial maximum does not depend on the complex $K$ or the map $\g$, we also define an \emph{abstract partially fixed disc} to be just a pair $(D,E)$, where $D$ and $E$ are as above. If no confusion is likely to arise, we shall omit the word `abstract'. As above, the edges in $E$ will be called fixed. 
\end{definition}

{Next, we prove an upper bound on the trivial maximum possible number of copies of a partially fixed disc.}

\begin{lemma}\label{pa:numofvKs}
	Let $D$ be an abstract partially fixed disc obtained by triangulating the disc and fixing the boundary edges. Then the trivial maximum number of copies of $D$ is at most $n^{V_I}$ where $V_I$ is the number of internal vertices -- that is, vertices that do not lie on the boundary.
\end{lemma} 

\begin{proof}
	The proof is by induction on the number of faces of $D$. The result is trivial when $D$ is a single face with all three edges fixed. Now suppose that $D$ has at least two faces. Suppose first that there is a face $f$ that has two boundary edges. Then the third edge must be internal. The label of this edge is determined by the labels on the two boundary edges. If we remove the face $f$ and fix its internal edge, then we obtain a disc that still has $V_I$ internal vertices, and hence at most $n^{V_I}$ copies, so we are done.
	
	If $D$ does not have such a face, then we split into two further cases. Suppose first that $D$ has an internal vertex: that is, a vertex that does not lie on the boundary. Then there must be an internal vertex that is joined by an edge to a boundary vertex $w$. The neighbours of $w$ form a path from its predecessor along the boundary to its successor. Let $v$ be the first internal vertex along this path. Then $v$ is joined to $w$ and to its predecessor, which gives us a face that has one boundary edge and two internal edges. We can choose the label for one of the internal edges in at most $n$ ways, and that determines the label for the other. Having done so, if we remove the face and fix the two internal edges, we obtain a simply connected disc $D'$ with one less internal vertex. For each of the at most $n$ choices of labelling for the newly fixed edges we get at most $n^{V_I-1}$ copies of $D'$, by the inductive hypothesis, so the number of copies of $D$ is at most $n^{V_I}$ as required.
	
	The final case is where $D$ does not have any internal vertices or any faces with two boundary edges. This case cannot in fact occur. Indeed, if it did, then note that the number of vertices would equal the number of boundary edges, and the number of faces would be at most the number of internal edges (since each face would contain at least two internal edges and each internal edge would be contained in two faces). It would follow that $V-E+F\leq 0$, contradicting Euler's formula (which would give $V-E+F=1$, since we are not counting the external face as a face).	
\end{proof}

A simple example that is important for us is that of a $2r$-gon: if the boundary is fixed, then we are left with at most $n$ possibilities. In the grid picture, this corresponds to the fact that if we know the labels of a label $2r$-cycle (say), then the first point of the cycle (which can be chosen in at most $n$ ways) determines the rest of the cycle if it exists.

An even more important example is where $D$ is taken to be the disc corresponding to the dispersed ring decomposition of a $2r$-gon as described in Definition~\ref{2rprisms}, again with the boundary cycle fixed. This bounds the maximum possible number dispersed ring decompositions of a given $2r$-gon. The number of internal vertices is $4r+1$, since the opposite $2r$-gon contributes $2r+1$ vertices, and each of the $2r$ 4-gons has a further internal vertex in the middle. Thus, Lemma~\ref{pa:numofvKs} gives an upper bound of $n^{4r+1}$ for the number of dispersed ring decompositions of a given $2r$-gon. But Lemma~\ref{pa:shatteredcount} gives us $\Omega(n^{4r+1})$ such decompositions, so we see again that our machinery from the previous section gives us within a constant factor of the maximum number of such objects.

\subsection{The slit octahedron case}\label{pa:sec:flappycub}

The aim of this section is to prove that if we are given a linear tripartite simplicial complex $K$ with $\e n^5$ octahedra, then we can pass to a dense subset $L$ of $K$ in which there are no slit octahedra. 

We begin by defining the auxiliary graph described in Section~\ref{overview}.

\begin{definition}\label{auxiliary}
	Given a linear tripartite simplicial complex, we let the auxiliary graph $G(K)$ have vertex set given by the edges in the 1-skeleton of $K$, with edges $x$ and $y$ joined if $x$ and $y$ are the boundary edges of some copy of the slit octahedron in $K$.
\end{definition}

{As discussed in Section~\ref{overview}, if we can prove that this auxiliary graph is of bounded degree, then we will be able to pass to a dense independent subset of the vertices and thereby eliminate all slit octahedra.} We shall achieve this by applying Theorem~\ref{surfacethm3.3} to $K$ and obtaining a subcomplex $K_s$ for some appropriately chosen $s$. The rough idea is that if we fix a vertex $z$ of the auxiliary graph $G(K_s)$ (recall that this is an edge of the complex $K_s$), then each edge $zz_i$ in $G(K_s)$ gives rise to a large number of copies of a certain tripartite surface (see Definition~\ref{surfacedefs}) homeomorphic to a disc with boundary of length 2, which we build from the initial slit-octahedron by unfixing all the interior edges using dispersed ring decompositions (see Definition~\ref{2rprisms}). If there are too many edges $zz_i$ in $G(K_s)$, this ends up contradicting Lemma \ref{pa:numofvKs}.

We now give the details.

\begin{lemma}\label{pa:lem10}
	Let $K$ be an $n\times n\times n$ linear tripartite simplicial complex containing at least $\epsilon n^5$ octahedra. Then there is a subcomplex $L$ of $K$ with at least $\epsilon^{2^{400}}n^2$ faces such that the maximum degree in the graph $G(L)$ is at most~$\epsilon^{-2^{450}}$.
\end{lemma}

\begin{proof}
	
	First we apply Theorem~\ref{surfacethm3.3} with $k=100$ (we only need $k\ge 4$) to obtain a sequence $K=K_0\supset K_1\supset \dots\supset K_4$ with the property that $K_i$ has at least $\alpha_i(\epsilon)n^2$ faces, and for each $r=2,\dots,4$ we have that every $2r$-gon in $K_i$ has at least $\gamma_i(\epsilon)\theta_i(\epsilon)^{2r+1} n^{4r+1}$ different dispersed ring decompositions in $K_{i-1}$. The parameters $\alpha_i(\epsilon), \theta_i(\epsilon)$ and $\gamma_i(\epsilon)$ are all at least $\epsilon^{100^{15i}}\ge\epsilon^{2^{100i}}$.
	
	Now suppose that the auxiliary graph $G(K_4)$ of $K_4$ has a vertex of degree at least $M$. Without loss of generality, let us assume that the vertex class that contains this vertex is $Z$. If the vertex is $z$, then we can find a set $\{z_1,\dots,z_M\}$ of distinct vertices in $Z$ such that for each $j$ there exists a copy of the slit octahedron in $K_4$ for which the boundary has edges $z$ and $z_j$ (recall from Definition~\ref{auxiliary} that the vertices of $G(K_4)$ are edges of $K_4$). 
	
	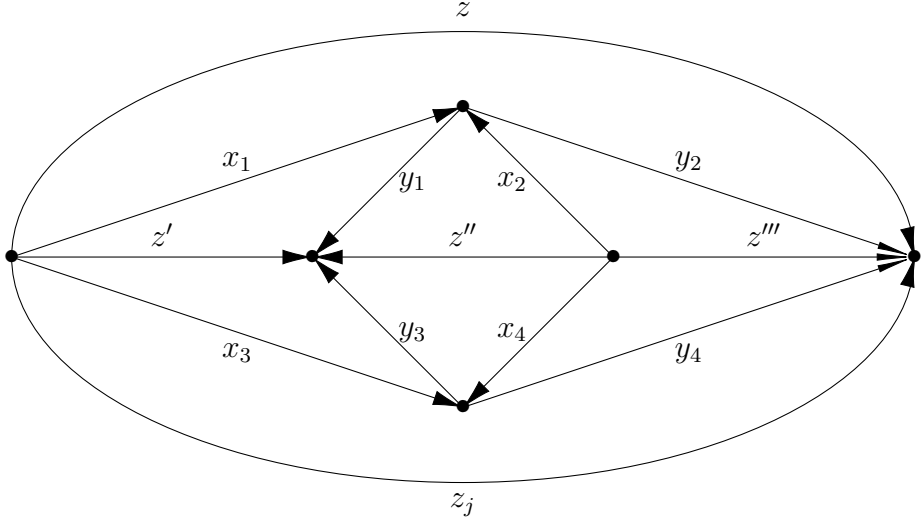
\begin{figure}
		\centering
		\begin{tikzpicture}[scale=3.0, every node/.style={scale=1.0}]
		\draw (0,0) ellipse (2cm and 1cm);
		
		\draw (-2,0) node {$\bullet$};
		\draw (2,0) node {$\bullet$};
		\draw (-2/3,0) node {$\bullet$};
		\draw (2/3,0) node {$\bullet$};
		\draw (0,2/3) node {$\bullet$};
		\draw (0,-2/3) node {$\bullet$};
		\draw (0,1.1) node {$z$};
		\draw (0,-1.1) node {$z_j$};
		
		\draw[-{Latex[length=4mm, width=2mm]}] (0,-2/3) -- (-2/3,0) node[midway, right] {$y_3$};
		\draw[{Latex[length=4mm, width=2mm]}-] (0,2/3) -- (2/3,0) node[midway, left] {$x_2$};
		\draw[{Latex[length=4mm, width=2mm]}-] (0,-2/3) -- (2/3,0) node[midway, left] {$x_4$};
		
		\draw[{Latex[length=5mm, width=1mm]}-] (2,0) -- (0,2/3) node[midway, above] {$y_2$};
		\draw[{Latex[length=4mm, width=2mm]}-] (0,2/3) -- (-2,0) node[midway, above] {$x_1$};	
		\draw[{Latex[length=5mm, width=1mm]}-] (2,0) -- (0,-2/3) node[midway, below] {$y_4$};
		\draw[{Latex[length=4mm, width=2mm]}-] (0,-2/3) -- (-2,0) node[midway, below] {$x_3$};
		
		\draw[-{Latex[length=4mm, width=2mm]}] (0,2/3) -- (-2/3,0) node[midway, right] {$y_1$};
		\draw[-{Latex[length=4mm, width=2mm]}] (-2,0) -- (-2/3,0) node[midway, above] {$z'$};
		\draw[-{Latex[length=4mm, width=2mm]}] (2/3,0) -- (-2/3,0) node[midway, above] {$z''$};
		\draw[-{Latex[length=5mm, width=1mm]}] (2/3,0) -- (2,0) node[midway, above] {$z'''$};
		\draw[-{Latex[length=4mm, width=2mm]}] (1.9988,0.005) -- (2,0);
		\draw[-{Latex[length=4mm, width=2mm]}] (1.9988,-0.005) -- (2,0);
		%\draw[-Latex] (2,-0.01) -- (2,0);
		
		%[pattern=north west lines, pattern color=red]
		%\begin{scope}
		%\clip (-2,0) rectangle (2,-2); % clipped area
		%\draw[pattern=north west lines, pattern color=blue] (0,0) ellipse (2cm and 1cm);
		%\end{scope}
		\end{tikzpicture}
		\caption{A slit octahedron yielding an edge between $z$ and $z_j$ in the auxiliary graph.}\label{pa:fig:fcvK}
	\end{figure}

		Let us now fix $j$ and let $D_0$ be the corresponding copy of the slit octahedron, which we shall think of as a partially fixed disc for which every edge is fixed. Let $\g$ be the inclusion map from the abstract slit octahedron to its copy in $D_0$. The slit octahedron is illustrated again in Figure~\ref{pa:fig:fcvK}, for ease of reference. We now select a 4-gon in $D_0$ by choosing some internal vertex and taking the four faces that surround it. For instance, we may select the bottom internal vertex, which is incident to the edges labelled $x_3,y_3,x_4$ and $y_4$. This gives us the triangulated 4-gon represented by the four faces in the bottom half of the diagram. We now create a new partially fixed disc $D_1$ as follows. First we remove this 4-gon from $D_0$ and replace it by a dispersed ring decomposition of the 4-gon (see Definition~\ref{2rprisms}). Then we declare all the internal edges of the dispersed ring decomposition to be unfixed, and the map $\g$ takes the same values as before, but is applied only to the fixed edges. The disc $D_1$ is illustrated in Figure~\ref{pa:fig:1ststepfcvK}, with the unfixed edges in red.

	Any 4-gon in $K_4$ has at least $\gamma_4(\epsilon)\theta_4^{5} n^{9}$ different dispersed ring decompositions in $K_3$. Since the trivial maximum number of these dispersed ring decompositions is $n^9$, by Lemma~\ref{pa:numofvKs} (because the number of internal vertices is 9), the number of copies of $D_1$ in the complex $K_{3}$ is within a constant of its trivial maximum, as we wanted. 
			
	\begin{figure}
		\centering
		\begin{tikzpicture}[scale=3.0, every node/.style={scale=1.0}]
		\draw (0,0) ellipse (2cm and 1cm);
		
		\draw (-2,0) node {$\bullet$};
		\draw (2,0) node {$\bullet$};
		\draw (-2/3,0) node {$\bullet$};
		\draw (2/3,0) node {$\bullet$};
		\draw (0,2/3) node {$\bullet$};
		
		\draw (-0.4,-0.4) node[color=red] {$\bullet$};
		\draw (0.4,-0.4) node[color=red] {$\bullet$};
		\draw (-0.4,-0.7) node[color=red] {$\bullet$};
		\draw (0.4,-0.7) node[color=red] {$\bullet$};
		\draw (0,-0.55) node[color=red] {$\bullet$};
		\draw (0,-0.2) node[color=red] {$\bullet$};
		\draw (-1,-0.2) node[color=red] {$\bullet$};
		\draw (1,-0.2) node[color=red] {$\bullet$};
		\draw (0,-0.88) node[color=red] {$\bullet$};
		
		\draw[-{Latex[length=4mm, width=1mm]}, color=red] (-0.4,-0.4) -- (0.4,-0.4);
		\draw[-{Latex[length=4mm, width=1mm]}, color=red] (-0.4,-0.4) -- (-0.4,-0.7);
		\draw[{Latex[length=4mm, width=1mm]}-, color=red] (-0.4,-0.7) -- (0.4,-0.7);
		\draw[-{Latex[length=4mm, width=1mm]}, color=red] (0.4,-0.7) -- (0.4,-0.4);
		
		\draw[-{Latex[length=4mm, width=1mm]}, color=red] (-0.4,-0.4) -- (0,-0.55);
		\draw[{Latex[length=4mm, width=1mm]}-, color=red] (0.4,-0.4) -- (0,-0.55);
		\draw[{Latex[length=4mm, width=1mm]}-, color=red] (-0.4,-0.7) -- (0,-0.55);
		\draw[-{Latex[length=4mm, width=1mm]}, color=red] (0.4,-0.7) -- (0,-0.55);
		
		\draw[-{Latex[length=4mm, width=1mm]}, color=red] (2/3,0) -- (0,-0.2);
		\draw[{Latex[length=4mm, width=1mm]}-, color=red] (-2/3,0) -- (0,-0.2);
		\draw[{Latex[length=4mm, width=1mm]}-, color=red] (0.4,-0.4) -- (0,-0.2);
		\draw[-{Latex[length=4mm, width=1mm]}, color=red] (-0.4,-0.4) -- (0,-0.2);
		
		\draw[-{Latex[length=4mm, width=1mm]}, color=red] (-2,0) -- (-1,-0.2);
		\draw[{Latex[length=4mm, width=1mm]}-, color=red] (-2/3,0) -- (-1,-0.2);
		\draw[{Latex[length=4mm, width=1mm]}-, color=red] (-0.4,-0.7) -- (-1,-0.2);
		\draw[-{Latex[length=4mm, width=1mm]}, color=red] (-0.4,-0.4) -- (-1,-0.2);
		
		\draw[-{Latex[length=4mm, width=1mm]}, color=red] (2/3,0) -- (1,-0.2);
		\draw[{Latex[length=4mm, width=1mm]}-, color=red] (2,0) -- (1,-0.2);
		\draw[{Latex[length=4mm, width=1mm]}-, color=red] (0.4,-0.4) -- (1,-0.2);
		\draw[-{Latex[length=4mm, width=1mm]}, color=red] (0.4,-0.7) -- (1,-0.2);
		
		\draw[-{Latex[length=4mm, width=1mm]}, color=red] (0,-0.88) -- (-0.4,-0.7);
		\draw[{Latex[length=4mm, width=1mm]}-, color=red] (0,-0.88) -- (0.4,-0.7);
		
		\draw [{Latex[length=4mm, width=1mm]}-, color=red] (0,-0.88) to [out=180,in=320] (-2,0);
		\draw [-{Latex[length=4mm, width=1mm]}, color=red] (0,-0.88) to [out=0,in=220] (2,0);

		\draw[-{Latex[length=4mm, width=1mm]}, color=red] (-0.4,-0.4) -- (-2/3,0);
		\draw[{Latex[length=4mm, width=1mm]}-, color=red] (0.4,-0.4) -- (2/3,0);
		\draw[{Latex[length=4mm, width=1mm]}-, color=red] (-0.4,-0.7) -- (-2,0);
		\draw[-{Latex[length=4mm, width=1mm]}, color=red] (0.4,-0.7) -- (2,0);
		
		\draw[{Latex[length=4mm, width=2mm]}-] (0,2/3) -- (2/3,0) ;

		\draw[{Latex[length=5mm, width=1mm]}-] (2,0) -- (0,2/3) ;
		\draw[{Latex[length=4mm, width=2mm]}-] (0,2/3) -- (-2,0) ;

		\draw[-{Latex[length=4mm, width=2mm]}] (0,2/3) -- (-2/3,0) ;
		\draw[-{Latex[length=4mm, width=2mm]}] (-2,0) -- (-2/3,0) ;
		\draw[-{Latex[length=4mm, width=2mm]}] (2/3,0) -- (-2/3,0) ;
		\draw[-{Latex[length=5mm, width=1mm]}] (2/3,0) -- (2,0) ;
		\draw[-{Latex[length=4mm, width=2mm]}] (1.9988,0.005) -- (2,0);
		\draw[-{Latex[length=4mm, width=2mm]}] (1.9988,-0.005) -- (2,0);
		%\draw[-Latex] (2,-0.01) -- (2,0);
		
		%[pattern=north west lines, pattern color=red]
		%\begin{scope}
		%\clip (-2,0) rectangle (2,-2); % clipped area
		%\draw[pattern=north west lines, pattern color=blue] (0,0) ellipse (2cm and 1cm);
		%\end{scope}
		\end{tikzpicture}
		\caption{The disc $D_1$ obtained after the first popular replacement in a slit octahedron. The dispersed ring decomposition is represented with the red part of the diagram. All labels have been omitted for simplicity.}\label{pa:fig:1ststepfcvK}
	\end{figure}
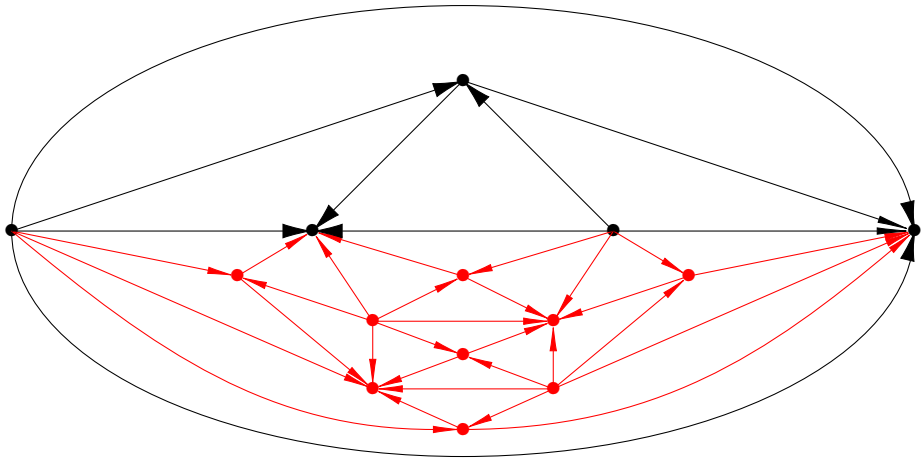
		
	The next step is to select another $2r$-gon by choosing another internal vertex, this time of $D_1$. We can do this by picking all the faces of $D_1$ that contain some given internal vertex that is incident to at least one unfixed edge. For instance, we might take the leftmost internal black vertex in Figure~\ref{pa:fig:1ststepfcvK}.
	
	This gives us a 6-gon $F$, since this vertex is contained in six faces of $D_1$. Let $D_2$ be the partially fixed disc obtained by replacing $F$ with a {dispersed} ring decomposition and declaring all its internal edges to be unfixed. It is already challenging to draw $D_2$ in detail, and we shall see shortly that it is not important to track the precise structure of the discs that we obtain at each step. Nevertheless, we include an illustration of $D_2$ in Figure~\ref{pa:fig:2ndstepfcvK} to help clarify the process.
	
	\begin{figure}
		\centering
		\begin{tikzpicture}[scale=3.0, every node/.style={scale=1.0}]
		\draw (0,0) ellipse (2cm and 1cm);
		
		\draw (-2,0) node {$\bullet$};
		\draw (2,0) node {$\bullet$};
		%\draw (-2/3,0) node {$\bullet$};
		\draw (2/3,0) node {$\bullet$};
		\draw (0,2/3) node {$\bullet$};
		
		\draw (-0.4,-0.4) node[color=red] {$\bullet$};
		\draw (0.4,-0.4) node[color=red] {$\bullet$};
		\draw (-0.4,-0.7) node[color=red] {$\bullet$};
		\draw (0.4,-0.7) node[color=red] {$\bullet$};
		\draw (0,-0.55) node[color=red] {$\bullet$};
		\draw (0,-0.2) node[color=red] {$\bullet$};
		\draw (-1,-0.2) node[color=red] {$\bullet$};
		\draw (1,-0.2) node[color=red] {$\bullet$};
		\draw (0,-0.88) node[color=red] {$\bullet$};
		
		\draw (-0.2,-0.1) node[color=red] {$\bullet$};
		\draw (-0.5,-0.14) node[color=red] {$\bullet$};
		\draw (-0.8,-0.1) node[color=red] {$\bullet$};
		\draw (-1.1,0) node[color=red] {$\bullet$};
		\draw (-0.2,0.3) node[color=red] {$\bullet$};
		\draw (0,0.1) node[color=red] {$\bullet$};
		
		\draw (-0.41,0.05) node[color=red] {$\bullet$};
		
		\draw (-0.75,0.25) node[color=red] {$\bullet$};
		\draw (-1.2,-0.08) node[color=red] {$\bullet$};
		\draw (-0.7,-0.22) node[color=red] {$\bullet$};
		\draw (-0.3,-0.22) node[color=red] {$\bullet$};
		\draw (0.05,-0.05) node[color=red] {$\bullet$};
		\draw (0.05,0.3) node[color=red] {$\bullet$};
		
		\draw[color=red] (-0.75,0.25) -- (0,2/3);
		\draw[color=red] (-0.75,0.25) -- (-0.2,0.3);
		\draw[color=red] (-0.75,0.25) -- (-1.1,0);
		\draw[color=red] (-0.75,0.25) -- (-2,0);
		
		\draw[color=red] (-1.2,-0.08) -- (-1,-0.2);
		\draw[color=red] (-1.2,-0.08) -- (-0.8,-0.1);
		\draw[color=red] (-1.2,-0.08) -- (-1.1,0);
		\draw[color=red] (-1.2,-0.08) -- (-2,0);
		
		\draw[color=red] (-0.7,-0.22) -- (-1,-0.2);
		\draw[color=red] (-0.7,-0.22) -- (-0.8,-0.1);
		\draw[color=red] (-0.7,-0.22) -- (-0.5,-0.14);
		\draw[color=red] (-0.7,-0.22) -- (-0.4,-0.4);
		
		\draw[color=red] (-0.3,-0.22) -- (-0.2,-0.1);
		\draw[color=red] (-0.3,-0.22) -- (0,-0.2) ;
		\draw[color=red] (-0.3,-0.22) -- (-0.5,-0.14);
		\draw[color=red] (-0.3,-0.22) -- (-0.4,-0.4);
		
		\draw[color=red] (0.05,-0.05) -- (-0.2,-0.1);
		\draw[color=red] (0.05,-0.05) -- (0,-0.2) ;
		\draw[color=red] (0.05,-0.05) -- (2/3,0);
		\draw[color=red] (0.05,-0.05) -- (0,0.1);
		
		\draw[color=red] (0.05,0.3) -- (0,2/3);
		\draw[color=red] (0.05,0.3) --  (-0.2,0.3);
		\draw[color=red] (0.05,0.3) -- (2/3,0);
		\draw[color=red] (0.05,0.3) -- (0,0.1);

		\draw[color=red] (-0.2,-0.1) -- (-0.5,-0.14);
		\draw[color=red] (-0.5,-0.14) -- (-0.8,-0.1) -- (-1.1,0) -- (-0.2,0.3) -- (0,0.1) -- (-0.2,-0.1);
		
		\draw[color=red] (-0.2,-0.1) -- (0,-0.2);
		\draw[color=red] (0,0.1) -- (2/3,0);
		\draw[color=red] (-0.2,0.3) -- (0,2/3);
		\draw[color=red] (-1.1,0) -- (-2,0);
		\draw[color=red] (-0.8,-0.1) -- (-1,-0.2);
		\draw[color=red] (-0.5,-0.14) -- (-0.4,-0.4);
		
		\draw[color=red] (-0.2,-0.1) -- (-0.41,0.05);
		\draw[color=red] (0,0.1) -- (-0.41,0.05);
		\draw[color=red] (-0.2,0.3) -- (-0.41,0.05);
		\draw[color=red] (-1.1,0) -- (-0.41,0.05);
		\draw[color=red] (-0.8,-0.1) -- (-0.41,0.05);
		\draw[color=red] (-0.5,-0.14) -- (-0.41,0.05);

		\draw[color=red] (-0.4,-0.4) -- (0.4,-0.4);
		\draw[color=red] (-0.4,-0.4) -- (-0.4,-0.7);
		\draw[color=red] (-0.4,-0.7) -- (0.4,-0.7);
		\draw[color=red] (0.4,-0.7) -- (0.4,-0.4);
		
		\draw[color=red] (-0.4,-0.4) -- (0,-0.55);
		\draw[color=red] (0.4,-0.4) -- (0,-0.55);
		\draw[color=red] (-0.4,-0.7) -- (0,-0.55);
		\draw[color=red] (0.4,-0.7) -- (0,-0.55);
		
		\draw[color=red] (2/3,0) -- (0,-0.2);
		%\draw[{Latex[length=4mm, width=1mm]}-, color=red] (-2/3,0) -- (0,-0.2);
		\draw[color=red] (0.4,-0.4) -- (0,-0.2);
		\draw[color=red] (-0.4,-0.4) -- (0,-0.2);
		
		\draw[color=red] (-2,0) -- (-1,-0.2);
		%\draw[{Latex[length=4mm, width=1mm]}-, color=red] (-2/3,0) -- (-1,-0.2);
		\draw[color=red] (-0.4,-0.7) -- (-1,-0.2);
		\draw[color=red] (-0.4,-0.4) -- (-1,-0.2);
		
		\draw[color=red] (2/3,0) -- (1,-0.2);
		\draw[color=red] (2,0) -- (1,-0.2);
		\draw[color=red] (0.4,-0.4) -- (1,-0.2);
		\draw[color=red] (0.4,-0.7) -- (1,-0.2);
		
		\draw[color=red] (0,-0.88) -- (-0.4,-0.7);
		\draw[color=red] (0,-0.88) -- (0.4,-0.7);
		
		\draw [color=red] (0,-0.88) to [out=180,in=320] (-2,0);
		\draw [color=red] (0,-0.88) to [out=0,in=220] (2,0);

		%\draw[-{Latex[length=4mm, width=1mm]}, color=red] (-0.4,-0.4) -- (-2/3,0);
		\draw[color=red] (0.4,-0.4) -- (2/3,0);
		\draw[color=red] (-0.4,-0.7) -- (-2,0);
		\draw[color=red] (0.4,-0.7) -- (2,0);
		
		\draw (0,2/3) -- (2/3,0);

		\draw (2,0) -- (0,2/3);
		\draw (0,2/3) -- (-2,0);

		%\draw[-{Latex[length=4mm, width=2mm]}] (0,2/3) -- (-2/3,0) ;
		%\draw[-{Latex[length=4mm, width=2mm]}] (-2,0) -- (-2/3,0) ;
		%\draw[-{Latex[length=4mm, width=2mm]}] (2/3,0) -- (-2/3,0) ;
		\draw[] (2/3,0) -- (2,0) ;
		\draw[] (1.9988,0.005) -- (2,0);
		\draw[] (1.9988,-0.005) -- (2,0);
		%\draw[-Latex] (2,-0.01) -- (2,0);
		
		%[pattern=north west lines, pattern color=red]
		%\begin{scope}
		%\clip (-2,0) rectangle (2,-2); % clipped area
		%\draw[pattern=north west lines, pattern color=blue] (0,0) ellipse (2cm and 1cm);
		%\end{scope}
		\end{tikzpicture}
		\caption{The partially fixed disc $D_2$ obtained after the second replacement. The fixed part is shown in black, and the unfixed part in red. All labels and directions have been omitted for simplicity.}\label{pa:fig:2ndstepfcvK}
	\end{figure}
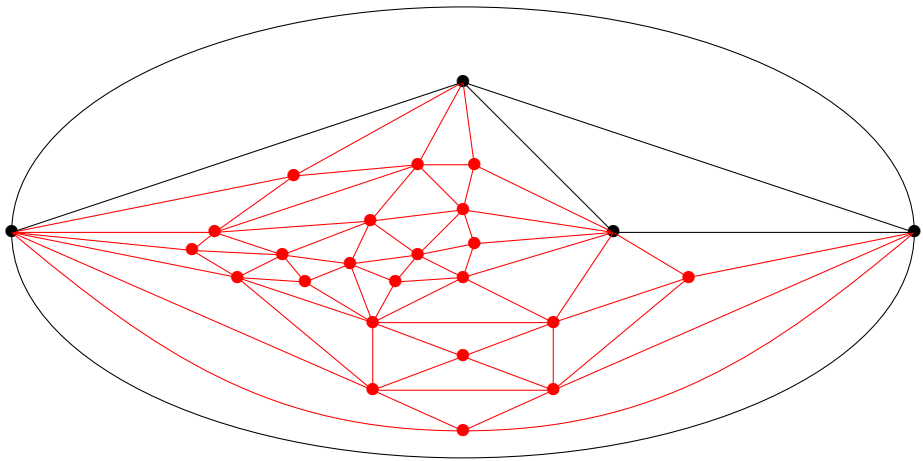
	
	Since any given 6-gon in $K_3$ has at least $\gamma_{3}\theta_3^{7}n^{13}$ different dispersed ring decompositions in $K_2$, we may obtain a copy of $D_2$ in $K_{2}$ by taking any one of the $\gamma_4\theta_4^5n^9$ copies of $K_1$ and then replacing the image of the 6-gon $F$ in that copy by any one of its $\gamma_{3}\theta_{3}^7n^{13}$ dispersed ring decompositions. We now claim that this gives us 
	\[(\gamma_4\theta_4^5n^9)(\gamma_{3}\theta_{3}^7n^{13})=\gamma_{3}\gamma_4\theta_{3}^7\theta_4^5n^{22}\]
 different copies of $D_2$ in $K_{2}$, but to verify this we must ensure that each copy we have just described is counted at most once.
	
	Suppose that $D$ and $D'$ are copies of $D_1$ in $K_3$ that, following replacements of their respective copies of $F$, both give the same copy of $D_2$ in $K_2$. Then $D$ and $D'$ must agree on all but the internal edges of $F$. However, we chose $F$ in such a way that one of the internal edges of $F$ is fixed and thus shared between $D$ and $D'$. But since a $2r$-gon is fully determined 
	by the boundary edges and a single internal edge (since two edges of a face in the {linear simplicial complex} uniquely determine the third), we see that $D=D'$.
	
	Therefore we do not overcount, and the number of copies of $D_2$ in $K_{2}$ is indeed $\gamma_{3}\gamma_4\theta_{3}^7\theta_4^5n^{22}$. Again, it is easy to see that this is within a constant of the trivial maximum, since a dispersed ring decomposition of a 6-gon has thirteen internal vertices, so the number of internal red vertices after the second unfixing is 22 (as the sceptical reader can verify from Figure \ref{pa:fig:2ndstepfcvK}).
	
	The remaining two steps are similar. At the next step, we can replace the 8-gon around the rightmost, internal black vertex in Figure~\ref{pa:fig:2ndstepfcvK} by its dispersed ring decomposition, with all the internal edges unfixed, to create a partially fixed disc $D_3$. 
	
	By Lemma~\ref{pa:shatteredcount}, the number of dispersed ring decompositions in $A_1$ is at least $\gamma_2\theta_2^9n^{17}$, so that the number of copies of $D_3$ in $K_{1}$ is at least $\gamma_2\theta_2^9n^{17}$ times the number of copies of $D_2$ in $K_{2}$. But we will also have added $8+8+1=17$ new internal red vertices, so the trivial maximum increases by a factor of $n^{17}$. Therefore, the number of copies of $D_3$ is within a factor $\gamma_2\gamma_{3}\gamma_4\theta_2^9\theta_{3}^7\theta_4^5$ of the maximum possible.
	
	In $D_3$ there is one remaining internal vertex that is incident to fixed edges. This vertex is the internal vertex of an 8-gon in $K_3$, so we may finish by replacing this 8-gon with a dispersed ring decomposition to obtain a partially fixed disc $D_4$, for which only the two boundary edges are fixed. As before, Lemma~\ref{pa:shatteredcount} gives us at least $\gamma_1\theta_1^9n^{17}$ dispersed ring decompositions in $K_{0}$, and therefore Lemma~\ref{pa:numofvKs} tells us that the number of copies of $D_4$ is within the constant factor $\gamma_1\gamma_2\gamma_{3}\gamma_4\theta_1^9\theta_2^9\theta_{3}^7\theta_4^5$ of the trivial maximum.
	
	Drawings of the full structure of $D_3$ and $D_4$ would be too complicated to be illuminating, but we include Figure~\ref{pa:fig:fcvKproof}, which gives a global view of the replacement sequence we have performed. In this figure we show $D_1$, $D_2$, $D_3$ and $D_4$ but instead of drawing all the unfixed edges, we simply indicate where they are with red hatching.
	
	\begin{figure}
		\centering
		\begin{subfigure}{\linewidth}
			\centering
			\begin{tikzpicture}[scale=1.8, every node/.style={scale=1.0}]
			\draw (0,0) ellipse (2cm and 1cm);
			
			\draw (-2,0) node {$\bullet$};
			\draw (2,0) node {$\bullet$};
			\draw (-2/3,0) node {$\bullet$};
			\draw (2/3,0) node {$\bullet$};
			\draw (0,2/3) node {$\bullet$};

			\draw[{Latex[length=4mm, width=2mm]}-] (0,2/3) -- (2/3,0) ;

			\draw[{Latex[length=5mm, width=1mm]}-] (2,0) -- (0,2/3) ;
			\draw[{Latex[length=4mm, width=2mm]}-] (0,2/3) -- (-2,0) ;

			\draw[-{Latex[length=4mm, width=2mm]}] (0,2/3) -- (-2/3,0) ;
			\draw[-{Latex[length=4mm, width=2mm]}] (-2,0) -- (-2/3,0) ;
			\draw[-{Latex[length=4mm, width=2mm]}] (2/3,0) -- (-2/3,0) ;
			\draw[-{Latex[length=5mm, width=1mm]}] (2/3,0) -- (2,0) ;
			\draw[-{Latex[length=4mm, width=2mm]}] (1.9988,0.005) -- (2,0);
			\draw[-{Latex[length=4mm, width=2mm]}] (1.9988,-0.005) -- (2,0);
			%\draw[-Latex] (2,-0.01) -- (2,0);
			
			%[pattern=north west lines, pattern color=red]
			\begin{scope}
			\clip (-2,0) rectangle (2,-1.3); % clipped area
			%\clip (-2,0) rectangle (1,0.55); % clipped area
			\draw[pattern=north west lines, pattern color=red] (0,0) ellipse (2cm and 1cm);
			\end{scope}
			\end{tikzpicture}
		\end{subfigure}

		\begin{subfigure}{\linewidth}
			\centering
			\begin{tikzpicture}[scale=1.8, every node/.style={scale=1.0}]
			\draw (0,0) ellipse (2cm and 1cm);
			
			\draw (-2,0) node {$\bullet$};
			\draw (2,0) node {$\bullet$};
			
			\draw (2/3,0) node {$\bullet$};
			\draw (0,2/3) node {$\bullet$};

			\draw[{Latex[length=4mm, width=2mm]}-] (0,2/3) -- (2/3,0) ;

			\draw[{Latex[length=5mm, width=1mm]}-] (2,0) -- (0,2/3) ;
			\draw[{Latex[length=4mm, width=2mm]}-] (0,2/3) -- (-2,0) ;

			\draw[-{Latex[length=5mm, width=1mm]}] (2/3,0) -- (2,0) ;
			\draw[-{Latex[length=4mm, width=2mm]}] (1.9988,0.005) -- (2,0);
			\draw[-{Latex[length=4mm, width=2mm]}] (1.9988,-0.005) -- (2,0);
			%\draw[-Latex] (2,-0.01) -- (2,0);
			
			%[pattern=north west lines, pattern color=red]
			\begin{scope}
			\clip (-2,1) rectangle (2,-1.3); % clipped area
			\clip[rotate around={(18.435:(-2,0)}] (-4,0) rectangle (4,-10); % clipped area
			\clip[rotate around={(315:(0,2/3)}] (-4,2/3) rectangle (4,-10); % clipped area
			\draw[pattern=north west lines, pattern color=red] (0,0) ellipse (2cm and 1cm);
			\end{scope}
			
			\begin{scope}
			\clip (-2,0) rectangle (2,-1.3); % clipped area
			\clip[rotate around={(135:(0,2/3)}] (-4,2/3) rectangle (4,-10); % clipped area
			\draw[pattern=north west lines, pattern color=red] (0,0) ellipse (2cm and 1cm);

			\end{scope}
			\end{tikzpicture}
		\end{subfigure}

		\begin{subfigure}{\linewidth}
			\centering
			\begin{tikzpicture}[scale=1.8, every node/.style={scale=1.0}]
			\draw (0,0) ellipse (2cm and 1cm);
			
			\draw (-2,0) node {$\bullet$};
			\draw (2,0) node {$\bullet$};
			\draw (0,2/3) node {$\bullet$};

			\draw[{Latex[length=5mm, width=1mm]}-] (2,0) -- (0,2/3) ;
			\draw[{Latex[length=4mm, width=2mm]}-] (0,2/3) -- (-2,0) ;

			\draw[-{Latex[length=4mm, width=2mm]}] (1.9988,0.005) -- (2,0);
			\draw[-{Latex[length=4mm, width=2mm]}] (1.9988,-0.005) -- (2,0);
			%\draw[-Latex] (2,-0.01) -- (2,0);
			
			%[pattern=north west lines, pattern color=red]
			\begin{scope}
			\clip (-2,1) rectangle (2,-1.3); % clipped area
			\clip[rotate around={(18.435:(-2,0)}] (-4,0) rectangle (4,-10); % clipped area
			\clip[rotate around={(341.565:(0,2/3)}] (-4,2/3) rectangle (4,-10); % clipped area
			\draw[pattern=north west lines, pattern color=red] (0,0) ellipse (2cm and 1cm);
			\end{scope}
			\end{tikzpicture}
		\end{subfigure}

		\begin{subfigure}{\linewidth}
			\centering
			\begin{tikzpicture}[scale=1.8, every node/.style={scale=1.0}]
			\draw (0,0) ellipse (2cm and 1cm);
			
			\draw (-2,0) node {$\bullet$};
			\draw (2,0) node {$\bullet$};

			\draw[-{Latex[length=4mm, width=2mm]}] (1.9988,0.005) -- (2,0);
			\draw[-{Latex[length=4mm, width=2mm]}] (1.9988,-0.005) -- (2,0);
			%\draw[-Latex] (2,-0.01) -- (2,0);
			
			%[pattern=north west lines, pattern color=red]
			\begin{scope}
			%\clip (-2,0) rectangle (2,-2); % clipped area
			\draw[pattern=north west lines, pattern color=red] (0,0) ellipse (2cm and 1cm);
			\end{scope}
			\end{tikzpicture}
		\end{subfigure}
		\caption{The sequence of four popular replacements from the proof of Lemma~\ref{pa:lem10}. Starting with a fixed disc corresponding to a slit octahedron, we progressively unfix all but the two boundary edges. Our unfixing process modifies the triangulation, and we represent the modified part with the red hatching (for example, the top figure represents $D_1$, shown in full detail in Figure~\ref{pa:fig:1ststepfcvK}). All edges in the triangulation represented by the red hatching are unfixed.}\label{pa:fig:fcvKproof}
	\end{figure}
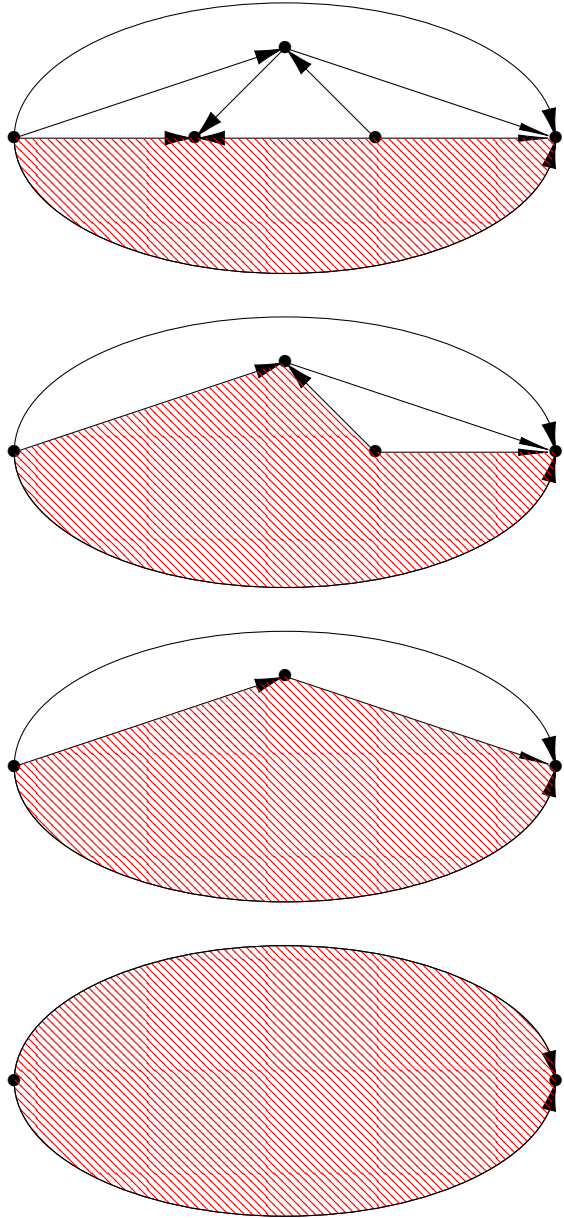
	
	Recall that this entire collection of copies of $D_4$ in $K_0$ was obtained by starting with a given slit octahedron, which yielded a disc with boundary edges $z$ and $z_j$. By performing this sequence of popular replacements for each choice of $j\in\{1,\dots,M\}$ we obtain $M$ different collections of copies of the same partially fixed disc. Each of these collections has a fixed boundary, but one of the two fixed boundary edges differs from collection to collection. By taking the union over all these collections, we obtain a final collection $\mathcal{D}$ of copies of $D_4$ in which only the label on one of the two boundary edges is fixed. 
	
	Now we need an upper bound for the maximum number of copies of the partially fixed disc $D_4'$, which is the same as $D_4$ except that only one of the two boundary edges is fixed. We cannot immediately apply Lemma~\ref{pa:numofvKs} since the entire boundary is not fixed. But we can modify $D_4'$ by attaching one new triangular face onto the unfixed boundary edge and fixing the other two edges of this face. We thus obtain a new partially fixed disc $D_4''$ with a boundary consisting of three fixed edges, and every internal edge {of $D_4''$} is unfixed. The maximum number of copies of $D_4''$ is at most the maximum number of copies of $D_4'$, since adding extra fixed edges cannot increase the number. We can now apply Lemma~\ref{pa:numofvKs} to $D_4''$, which has the same number of internal vertices as $D_4$. Therefore the maximum number of copies of $D_4''$ is the same as that of $D_4$, and hence the maximum number of copies of $D_4'$ is at most that of $D_4$. 
	
	But the size of the collection $\mathcal{D}$ is at least $M$ times the number of copies of $D_4$ found in the popular replacement process just described, which is within a constant factor $\gamma_1\gamma_2\gamma_{3}\gamma_4\theta_1^9\theta_2^9\theta_{3}^7\theta_4^5$ of the maximum possible. Therefore if $$M\gamma_1\gamma_2\gamma_{3}\gamma_4\theta_1^9\theta_2^9\theta_{3}^7\theta_4^5$$
	$$\ge M\epsilon^{2^{450}}>1$$
	then we have our contradiction. Therefore we may take $L=K_4$, which has at least $\alpha_4n^2\ge \epsilon^{2^{400}}n^2$ faces.	
\end{proof}

Our `removal lemma' for slit octahedra follows from this lemma.

\begin{theorem}\label{noslitocts}
	Let $K$ be an $n\times n\times n$ linear tripartite simplicial complex that contains at least $\e n^5$ octahedra. Then there is a subcomplex $L$ of $K$ with at least $\e^{O(1)} n^2$ faces that contains no slit octahedra.
\end{theorem}
\begin{proof}
	We apply Lemma~\ref{pa:lem10}. This gives us a subcomplex $L$ of $K$ with at least $\epsilon^{2^{400}}n^2$ faces such that the associated graph $G(K)$ has maximum degree at most $\epsilon^{-2^{450}}$.
	
	Let the vertices of $K$ and $L$ be $u,v$ and $w$, and let $X$ be the set of edges from $u$ to $v$, let $Y$ the set of edges from $v$ to $w$, and let $Z$ the set of edges from $u$ to $w$. To avoid confusion, it is important to keep in mind that these edges are the \emph{vertex} sets of the auxiliary graph $G(K)$.
	
	We now pick a maximal independent subset $I_X$ of $X$ in $G(K)$. We first pick a vertex $x\in X$ that (when considered as an edge of $K$) belongs to the largest number of faces of $K$ and add it to $I_X$. Then we discard all vertices in the neighbourhood of $x$ in the graph $G(K)$ and repeat, picking at each stage the remaining vertex that belongs to the largest number of faces of $K$. Since the maximum degree of $G(K)$ is at most $\epsilon^{-2^{450}}$, we end up picking at least $\epsilon^{2^{450}}n$ vertices from $X$, and these vertices when thought of as edges belong to at least a fraction $\epsilon^{2^{450}}$ of the faces of $K$ since at each stage we chose an edge that belonged to an above average number of faces. 
	Let $K_1$ be the subcomplex of $K$ induced by $I_X$, $Y$ and $Z$. Then $K_1$ has at least $\epsilon^{2^{451}}n^2$ faces, and inside $K_1$ there is no slit octahedron with its boundary edges belonging to $X$. 
	
	Since $G(K_1)$ is a subgraph of $G(K)$, it also has maximum degree at most $\epsilon^{-2^{450}}$, and we may similarly choose an independent set $I_Y$ in the graph $G(K_1)$ of at least $\epsilon^{2^{450}}n$ vertices from $Y$, accounting for at least a fraction $\epsilon^{2^{450}}$ of the faces of $K_1$. This gives us a subcomplex $K_2$ with at least $\epsilon^{2^{452}}n^2$ faces with no slit octahedron with its boundary edges belonging to either $X$ or $Y$. 
	
	Finally, we choose an independent set $I_Z$ in the graph $G(K_2)$ of at least $\epsilon^{2^{450}}n$ vertices from $Z$, accounting for the greatest fraction of faces of $K_2$. This gives us a subcomplex $K_3$ with at least $\epsilon^{2^{453}}n^2$ faces with no slit octahedra, which we take as our subcomplex $L$.
\end{proof}

Theorem~\ref{noflappycub}, stated in the introduction, follows immediately from Theorem~\ref{noslitocts}. 

\begin{proof}[Proof of Theorem~\ref{noflappycub}]
	The result follows by applying Theorem~\ref{noslitocts} to the partial Latin square $A$ (viewed as a linear hypergraph as explained in Lemma~\ref{PLSequalsH}) and noting that the resulting subcomplex $L$ of $K(A)$ corresponds to a subset $B$ of the partial Latin square $A$ of positive density (depending only on $\e$). Moreover, since $L$ contains no slit octahedra, it follows that $B$ satisfies the quadrangle condition (Definition~\ref{qc2}).
\end{proof}

\subsection{The general case}

Almost all of the complexity of the general case is contained in the detailed account given for the slit octahedron in the previous section. What remains is to describe how the replacement steps work in general, so that we can see that the argument for the slit octahedron generalizes straightforwardly to arbitrary discs with boundaries of length 2.

The outline of the approach is as above. Given a disc $D$ with boundary of length 2 and a linear tripartite simplicial complex $K$, we shall define the auxiliary graph $G(K,D)$ on the set of edges in the 1-skeleton of $K$ by joining $K$-edges $x$ and $y$ by a $G(K,D)$-edge if there is a copy of $D$ in $K$ with $x$ and $y$ as its two boundary edges.

The main lemma will show that we may pass to a dense subcomplex $L$ of $K$ such that the auxiliary graph $G(L,D)$ has bounded degree for each $D$ of size below some chosen bound. If this is the case, then the elimination of small discs of boundary length 2 is straightforward -- as in the proof of Theorem~\ref{noslitocts}, we will simply pass down to independent sets in the graphs $G(L,D)$ in such a way that we avoid discarding too much of $L$.

The proof of the main lemma is similar to that of Lemma~\ref{pa:lem10}. Given $M$ fixed copies of the disc $D$ with boundary edges $z$ and $z_j$ (for $j=1,\dots,M$), we shall unfix the edges by using popular decompositions of constituent $2r$-gons that surround internal vertices. At each stage we have, for each $j$, a collection of almost maximal size of copies of a partially fixed disc with boundary edges $z$ and $z_j$. We aim to show that once all edges incident to internal vertices are unfixed, we will have more than the trivial maximum number of copies of a certain partially fixed disc in $K'$ unless $M$ is bounded above by some constant that is independent of $n$ (which will have a power dependence on $\epsilon$, with the exponent depending on the number of faces of $D$). 

{The condition that $b\geq 100$ in the next lemma is purely for convenience, as when $b\geq 100$ it makes certain calculations easier. Of course, the result for $b\geq 100$ implies the result for $b<100$.}

\begin{lemma}\label{pa:lem11}
	Let $K$ be an $n\times n\times n$ linear tripartite simplicial complex containing at least $\epsilon n^5$ octahedra, and $b\ge 100$. Then we can pass to a subcomplex $L$ of $K$ with at least $\epsilon^{b^{20b}}n^2$ faces such that for each disc $D$ with at most $b$ faces and a boundary of length 2, the maximum degree in the graph $G(L,D)$ is at most $\epsilon^{-b^{20b}}$.
\end{lemma}

\begin{proof}
	We begin the proof, as we began the proof of Lemma~\ref{pa:lem10}, by applying Theorem~\ref{surfacethm3.3}, which we do with $k=2b$. We obtain a sequence $K=K_0\supset K_1\supset \dots$ with the property that $K_i$ has at least $\alpha_i(\epsilon,2b)n^2$ faces and for each $r=2,\dots,k$ we have that every $2r$-gon in $A_i$ has at least $\gamma_i(\e,2b)\theta_i(\epsilon,2b)^{2r+1}n^{4r+1}$ different {dispersed ring decompositions}, where each of $\alpha_i,\gamma_i$ and $\theta_i$ are at least $\epsilon^{(2b)^{15i}}\ge\e^{b^{20i}}$. Our tripartite complex $L$ will be $K_b$, which has at least $\epsilon^{b^{20b}}n^2$ faces. Note that $L$ has been chosen independently of any particular disc $D$: it will in fact serve for all discs with boundary of length 2 and at most $b$ faces. 
Note that the number of internal vertices of any such disc is at most $3b/4<b$, since each internal vertex is contained in at least four faces and each face contains at most three internal vertices.
	
	Now let $D$ be any disc with at most $b$ faces and with all its edges fixed. Our goal is to unfix all edges except the boundary edges. As before, our unfixing steps involve picking vertices from the diagram, removing all of their incident faces and re-triangulating the resulting $2r$-gon using the dispersed ring decomposition, taking all internal edges of this dispersed ring decomposition to be unfixed. Starting with $D=D_0$, this process will lead us to construct a sequence $D=D_0, D_1, D_2, \dots$ of partially fixed discs and associated collections $\mathcal{D}_i$ of copies of these discs, where the copies in the family $\mathcal D_i$ live in the complex $K_{s-i}$.
	
	In the previous section, we performed the replacements one by one and ensured at each stage that the size of $\mathcal{D}_i$ was within a constant of the maximum possible. For the general case, it will be simplest to perform the latter check at the end, once all replacements have been made and we have reached a partially fixed disc $D_s$ in which all edges incident to internal vertices are unfixed.
	
	At each stage, we pick any vertex $v$ inside $D_i$ (not on the boundary) such that $v$ is incident to fixed edges. We then consider the faces containing $v$ -- there are $2r_i$ of them giving a $2r_i$-gon. We replace this $2r_i$-gon with a dispersed ring decomposition with unfixed internal edges, giving us $D_{i+1}$. As before, the collection of copies $\mathcal{D}_{i+1}$ is obtained from $\mathcal{D}_i$ by choosing each possible replacement for each member of $\mathcal{D}_i$. As in the slit octahedron case, we will have that the size of $\mathcal{D}_{i+1}$ is equal to at least the size of $\mathcal{D}_i$ times the minimum number of different dispersed ring decompositions of the $2r_i$-gon in the complex $K_{s-i-1}$. We do not overcount, since if two copies of $D_i$ agree on all edges apart from those incident to $v$ then, since $v$ is also incident to a fixed edge, they must agree everywhere.
	
	At each stage we reduce the number of internal vertices incident to fixed edges by exactly one, so the number of unfixing steps that we need to perform is equal to the number of internal vertices of the disc $D$, which is at most $3b/4$. Moreover, the maximum degree of a vertex in $D$ is bounded above by $b$ and this increases by at most two with each popular replacement. Thus, the maximum value of $r$ for which we ever perform a popular replacement of a $2r$-gon is bounded above by $(b+2(3b/4))/2\leq 2b=k$.
	
	We now consider the disc $D_s$ that we get at the end of this process. Each time we do a popular replacement of a $4r_i$-gon, we increase the size of the family by a factor $\gamma_{k+1-i}\theta_{k+1-i}^{2r_i-1}n^{4r_i+1}$.
	So at the end of the process, the size of the collection $\mathcal D_s$ is at least
	\[\gamma_{b}^b\theta_{b}^{4b^2}\prod_{i=1}^s n^{4r_i+1}\ge \epsilon^{b^{20b}}\prod_{i=1}^s n^{4r_i+1}.\]
	
	The number of internal vertices of $D_s$ is $\sum_{i=1}^s (4r_i+1)$, since at each step of the unfixing process we replace one internal vertex by the $4r_i+1$ internal vertices of a dispersed ring decomposition. So, by Lemma~\ref{pa:numofvKs}, the maximum possible size of a collection of copies of $D_s$ that agree on the boundary edges is $\prod_{i=1}^s n^{4r_i+1}$. 
	
	Therefore $|\mathcal{D}_s|$ is within a constant factor of the maximum possible. Indeed the constant factor $\eta$ is bounded by
	$$\eta\ge \epsilon^{b^{20b}}.$$
	
	As before, we may perform the same unfixing process (in the same order) for each vertex $z_i$ ($i=1,\dots,M$). Each one gives us a collection of discs with fixed boundary edges. The union of these collections is $\mathcal{D}$, a collection of copies of the partially fixed disc $D'$ obtained by unfixing the appropriate boundary edge of $D_s$. By the same trick as in the previous section, we can apply Lemma~\ref{pa:numofvKs} to deduce that the maximum possible number of copies of $D'$ is in fact the same as the maximum possible number of copies of $D_s$, and therefore we obtain a contradiction if $M\eta>1$. Therefore $M\le \epsilon^{-b^{20b}}$, which proves the lemma.
\end{proof}

We are finally ready to prove Theorem~\ref{nosmallslitspheres}, which we restate here with explicit bounds.

\begin{theorem}\label{nosmallslitsphereseffective}
	Let $b$ be a positive integer, let $K$ be an $n\times n\times n$ linear tripartite simplicial complex, and suppose that $K$ contains at least $\e n^5$ octahedra. Then $K$ contains a subcomplex $L$ with at least $\e^{b^{25b}}n^2$ faces such that $L$ does not contain a copy of any disc with area less than $b$ and boundary of length 2.
\end{theorem}

\begin{proof}
	Let $u,v$ and $w$ be the vertices of $K$, let $X_1$ be the set of edges from $u$ to $v$, let $X_2$ be the set of edges from $v$ to $w$, and let $X_3$ be the set of edges from $u$ to $w$. We apply Lemma~\ref{pa:lem11} to obtain a subcomplex $L$ of $K$ such that the graph $G(L,D)$ has maximum degree at most $\epsilon^{-b^{20b}}$ for any tripartite disc $D$ homeomorphic to a disc with fewer than $b$ faces and with boundary of length 2. The goal is now to pass to subsets $X_1'\subset X$, $X_2'\subset X_2$, and $X_3'\subset X_3$ such that for $i=1,2,3$, $G_i(L_X,D)$ contains no edges for any choice of $D$ as above, where $L_X$ is the subcomplex of $L$ induced by the edges in $X_1, X_2$ and $X_3$ and $G_i(L_X,D)$ is the auxiliary graph with respect to $L_X$ and $D$ with vertex set $X_i$.
	
	In order to do this, we introduce the graph $G(L,b)$ which is the union of all graphs $G_i(L,D)$ where $D$ is as above and $i=1,2$ or $3$. Since a tripartite disc with boundary of length 2 has $3b/2+1$ vertices, the number of different $D$ is at most $(3b/2+1)^{b+1}$, so $G(L,b)$ has maximum degree at most $(3b/2+1)^{b+1}\epsilon^{-b^{20b}}$.
	
	Now, as in the proof of Theorem~\ref{noslitocts}, we select our subsets $X_i$ by passing to independent sets in the $G(L,b)$ in such a way that the number of faces in the induced subcomplex $L_V$ is maximized. Doing this gives us a subcomplex $L$ which is guaranteed to have at least 
	\[\Big((3b/2+1)^{-(b+1)}\epsilon^{b^{20b}}\Big)^3 n^2 \ge\epsilon^{b^{25b}}n^2\]
faces, and which contains no copy of any disc with area less than $b$ and boundary of length 2.
\end{proof}
\begin{remark}
	Of course, Theorem~\ref{nosmallslitsphereseffective} implies a version of Theorem~\ref{noslitocts}, although the bound is somewhat worse because Theorem~\ref{nosmallslitsphereseffective} uses crude estimates for the number of replacements required (whereas in the proof of Theorem~\ref{noslitocts} we determine an exact sequence of four replacements for the slit octahedron, and determine each $r_i$ required).
\end{remark}

\section{Obtaining the approximate isomorphism}

We sketched the rest of the proof of our main theorem in Section \ref{sketch}. In this short section we give a detailed proof. The previous sketch is probably sufficient for a reader with a background in geometric group theory, so this section is mainly for the benefit of combinatorialists who may not have such a background.

Our aim in this section is to prove the following statement, the hypotheses of which come from the conclusion of Theorem \ref{nosmallslitsphereseffective}.

\begin{proposition}\label{translation}
Let $b$ be a positive integer, let $K$ be the van Kampen complex of a partial Latin square $(X,Y,Z,A,\lambda)$ and suppose that $K$ does not contain a disc with area less than $b$ and boundary of length 2. Then there is a metric group $G$ and maps $\phi:X\to G$, $\psi:Y\to G$ and $\omega:Z\to G$ such that the images $\phi(X), \psi(Y)$ and $\omega(Z)$ are 1-separated sets, and $d(\phi(x)\psi(y),\omega(z))\leq b^{-1}$ whenever $(x,y)\in A$ and $\lambda(x,y)=z$.
\end{proposition}

More informally, the conclusion of the proposition is saying (for large $b$) that the partial Latin square is approximately isomorphic to part of the multiplication table of a metric group.

The definition of the metric group is given by a simple universal construction, as we said in Section \ref{sketch}. Assume, as we clearly may, that $X,Y$ and $Z$ are disjoint sets. Then $G$ is simply the free group generated by $X\cup Y\cup Z$. As for the metric on $G$, it is the largest metric that is compatible with the `approximate relations' $d(xy,z)\leq b^{-1}$, where we have such a relation for every triple $(x,y,z)$ with $(x,y)\in A$ and $\lambda(x,y)=z$ (or equivalently for every face of the associated linear tripartite 3-uniform hypergraph). Recall that we allow infinite distances.

However, if we want to prove Proposition \ref{translation}, it is simpler to use the more explicit description of this metric that we also mentioned in Section \ref{sketch}. For the main result, we do not need to know that the two metrics coincide, but we shall briefly indicate the (standard) proof once Proposition \ref{translation} is established.

\begin{definition}
Let $v_1$ and $v_2$ be two elements of the free group on $X\cup Y\cup Z$, let $w_1$ and $w_2$ be two words representing $v_1$ and $v_2$ and let $R$ be a set of relations of the form $xy=z$. The \emph{van Kampen distance} $d(v_1,v_2)$ is the smallest area of a van Kampen diagram with relations from $R$ and boundary word $w_1w_2^{-1}$. If no such diagram exists, then we set $d(v_1,v_2)$ to be infinite. \end{definition}

Of course, the above definition can be generalized to any set of relations on any free group, but we content ourselves with the case that concerns us in this paper.

It is not hard to check that the distance above is well-defined. To prove this, it is enough to show that if we insert an inverse pair into either $w_1$ or $w_2$, the smallest area of a van Kampen diagram with boundary word $w_1w_2^{-1}$ does not change. This is obvious, provided one allows van Kampen diagrams to contain degenerate parts that consist of an edge traversed forwards and then immediately backwards (appropriately directed, and with a label $u$ that counts as $u$ in the forwards direction and $u^{-1}$ in the backwards direction). 

A similar argument shows that it it is translation invariant. To see, for example, that it is left-translation invariant, note that any van Kampen diagram with boundary word $w_1w_2^{-1}$ can be converted into a van Kampen diagram with boundary word $uw_1w_2^{-1}u^{-1}$ by adding an edge labelled $u$ and directed into the vertex at the beginning of $w_1$ and the end of $w_2$, and conversely, given any van Kampen diagram with that boundary word, we can simply remove the edge labelled $u$ and obtain a van Kampen diagram for $w_1w_2^{-1}$. 

The fact that it satisfies the triangle inequality is again straightforward. Given van Kampen diagrams with boundary words $w_1w_2^{-1}$ and $w_2w_3^{-1}$, one can create a van Kampen diagram with boundary word $w_1w_3^{-1}$ by gluing them together along their common section of boundary $w_2$. The area of this new word is the sum of the areas of the first two words, so we are done.

Symmetry is obvious. To see that $d(v_1,v_2)=0$ only if $v_1=v_2$, observe that a van Kampen diagram of area zero is simply a labelled tree. The boundary word is obtained by following a path round the tree in the standard way, and a simple induction, removing one isolated vertex at a time, then shows that this word reduces to the empty word when one cancels inverse pairs.

\begin{proof}[Proof of Proposition \ref{translation}]
Let $G$ be the free group on $X\cup Y\cup Z$ with $b^{-1}$ times the van Kampen metric. If $(x,y)\in A$ and $\lambda(x,y)=z$, then the triangle with directed edges $ab$ labelled $x$, $bc$ labelled $y$, and $ac$ labelled $z$ is a van Kampen diagram of area 1 with boundary word $xyz^{-1}$, so $d(xy,z)\leq b^{-1}$. 

If $u$ and $v$ are distinct elements of $X\cup Y\cup Z$, then by hypothesis the smallest van Kampen diagram with boundary word $uv^{-1}$ has area at least $b$, so $d(u,v)\geq 1$. It follows that the images of $X, Y$ and $Z$ are 1-separated, as claimed.  
\end{proof}

\begin{remark} It is also straightforward to show that $d(u,v)=\infty$ unless $u$ and $v$ belong to the same one of the sets $X, Y$ and $Z$. Indeed, since every relation derived from the partial Latin square is of the form $x_iy_j=z_k$, there is a homomorphism from the group with generators $X\cup Y\cup Z$ and all those relations to the group with presentation $\langle x,y,z|xy=z\rangle$, which is isomorphic to the free group generated by any two of $x, y$ and $z$. Therefore, it is not possible to prove using the relations that two generators from different classes are equal. Since a van Kampen diagram with boundary word $uv^{-1}$ \emph{does} prove that $x=y$, the claim follows. 
\end{remark}

\begin{remark}
The fact that ($b^{-1}$ times) the van Kampen metric really is the largest that is compatible with the `approximate relations' $d(x_iy_j,z_k)\leq b^{-1}$ comes from an inductive argument. Suppose that $D$ is another metric compatible with these approximate relations. We need to prove that if $w_1,w_2$ are two words and there is a van Kampen diagram of area $s$ and boundary word $w_1w_2^{-1}$, then $D(w_1,w_2)\leq b^{-1}s$. If $s=1$, this is trivial. For larger $s$, we can split the van Kampen diagram into two parts, both of positive area, as follows.

Without loss of generality at least one edge $ab$ on the part of the boundary corresponding to $w_1$ bounds a face $F$. Let $w_3$ be the word obtained by following the boundary until it gets to $a$, then going round $F$ the other way to $b$, and proceeding to the end of $w_1$. Then $d(w_1,w_3)=1$ and $d(w_3,w_2)=s-1$. By induction it follows that $D(w_1,w_3)\leq b^{-1}$ and $D(w_3,w_2)\leq b^{-1}(s-1)$, and then by the triangle inequality we deduce that $D(w_1,w_2)\leq b^{-1}s$. 
\end{remark}

\begin{remark}
Another way of thinking about the van Kampen distance, and one of the main reasons van Kampen diagrams were defined, is that it corresponds to the length of the shortest proof that two words are equal given a certain set of relations. Roughly speaking, the length of the proof is defined to be the number of times the relations need to be used in order to transform one word into another. (Inserting or cancelling inverse pairs is not regarded as contributing to the proof length.) Thus, we can think of Theorem \ref{nosmallslitsphereseffective} as saying that a partial Latin square with many octahedra contains a dense partial Latin square for which there is no short proof that it is not isomorphic to part of the multiplication table of a group.
\end{remark}

\section{Concluding remarks}

It is important to stress that although algebraically the group $G$ introduced in the proof of Proposition~\ref{translation} is just a free group, the metric gives it a much more interesting structure. Indeed, one can think of this metric as an approximate group presentation: instead of declaring that certain words are equal to the identity, we declare that they are \emph{close} to the identity, and then we take the distance to be the largest one that is compatible with these `approximate relations'. (Note that this should be read as `approximate group-presentation' and not `approximate-group presentation'.)

Theorem \ref{main} gives us in particular a metric group $G$ and three 1-separated subsets $\phi(X),\psi(Y), \omega(Z)$ of $G$ of comparable size with the property that for a constant proportion of pairs $(x,y)\in \phi(X)\times \psi(Y)$ there exists $z\in \omega(Z)$ such that $d(xy,z)\leq\d$, where $\d=b^{-1}$. If we replace the condition $d(xy,z)\leq\d$ by the condition that $xy=z$, we obtain a condition that is very closely related to the definition of an approximate group. In particular, we can conclude that there is an approximate group $H$ of size not much larger than $|\phi(X)|$ and translates $xH$ and $Hy$ of $H$ such that a constant proportion of the points of $\phi(X)$ belong to $xH$ and a constant proportion of the points of $\psi(Y)$ belong to $Hy$. In the first appendix we show that a suitable `metric entropy version' of this result holds, which allows us to replace equality by approximate equality and obtain an appropriate conclusion, where the notion of an approximate group is replaced by that of an approximate group that is also approximate in a metric sense. We call these structures `rough approximate groups'. (To the best of our knowledge, this concept was first formulated by Tao \cite{taobp}, and a slight adaptation of it was introduced and studied by Hrushovski \cite{hrushovski}, who called it a metrically approximate subgroup.)

It would be very interesting to go further and describe in a more concrete way the structure of rough approximate groups, ideally obtaining an analogue of the results of Breuillard, Green and Tao on approximate groups \cite{BGT}. We have not attempted to formulate a conjecture along these lines, but examples such as taking a maximal $\d$-separated subset of a small ball about the identity in $\so3$, where the size of the ball tends to zero with $\d$ but much more slowly than $\d$, suggest that Lie groups of bounded rank are likely to play a role, and also that the part played by nilpotency may be significantly different. 

It is natural to ask whether there is an analogue of the results of this paper for Abelian groups. In a forthcoming paper we address this question, identifying a structure that plays the role that the octahedron plays for general groups, in the sense that if the number of copies of that structure in a partial Latin square is within a constant of maximal, then the partial Latin square has Abelian-group-like behaviour. The proof turns out to be quite a lot harder, because it is necessary to consider tripartite surfaces of higher genus, and that leads to significant complications.

\subsection*{Acknowledgement}
We would like to thank an anonymous referee, whose suggestions led us to rewrite the paper substantially and greatly improve its presentation.

\appendix
\appendixpage

\section{Rough approximate groups}

Let $G$ be a group. A subset $H$ of $G$ is a $k$-\emph{approximate subgroup} if it contains the identity, it is closed under taking inverses, and there exists a set $K$ of size at most $k$ such that $HH\subset KH$ -- that is, if the product set $HH$ can be covered by a bounded number of (left) translates of $H$. If $G$ is a metric group, we shall say that a subset $H$ is a $(k,\d)$-\emph{rough approximate subgroup} if there is a set $K$ of size at most $k$ such that $HH\subset(KH)_\d$, where for any subset $U$ we write $U_\d$ denotes the $\d$-expansion $\{x:d(x,U)\leq\d\}$ of $U$. Thus, $H$ is a rough approximate subgroup if every point in $HH$ can be approximated by a point in one of a bounded number of translates of $H$. By a rough approximate group, we mean simply a rough approximate subgroup of some metric group. (As with approximate groups themselves, it is possible to define rough approximate groups more intrinsically, but since ours arise naturally as subsets of an ambient group, we shall not do this.) 

{From now on, we will take the mappings $\phi,\psi$ and $\omega$ from Theorem~\ref{main} as given, and refer to $\phi(X),$ $\psi(Y)$ and $\omega(Z)$ as $X$, $Y$ and $Z$ instead.

With this convention,} Theorem \ref{main} yields for us three 1-separated subsets $X,Y,Z$ of a metric group $G$, all of roughly the same size, and a small positive number $\d$, such that $d(xy,Z)\leq\d$ for a positive proportion of pairs $(x,y)\in X\times Y$. In this appendix we shall deduce that there is a rough approximate subgroup $H$ of $G$ such that $X$ has substantial overlap with a left translate of $H$, $Y$ has substantial overlap with a right translate, and $Z$ has substantial overlap with a two-sided translate. The (slightly stronger) precise statement is Theorem~\ref{gettingrag} below. The arguments are mostly contained in either \cite{taobp} or \cite{taopaper}, and those that are not are fairly straightforward modifications or extensions of those arguments. It is for that reason, and because the result is something of an optional extra to our main result, that we present it in an appendix rather than in the main body of the paper. 

\subsection{Metric entropy definitions and some basic observations}

Given a subset $X$ of a metric space, and another subset $\Delta$, we say that $\Delta$ is an $\e$-\emph{net} of $X$ if for every $x\in X$ there exists $y\in\Delta$ such that $d(x,y)<\e$. An $\e$-\emph{separated subset} of $X$ is a subset $\Gamma$ such that $d(x,x')\geq\e$ for every pair of distinct elements $x,x'\in\Gamma$. Write $\nu_\e(X)$ for the smallest size of an $\e$-net of $X$, and $\sigma_\e(X)$ for the largest size of an $\e$-separated subset. We begin with three very basic lemmas.

\begin{lemma} \label{nusigma}
Let $X$ be a subset of a metric space and let $\e>0$. Then $\nu_\e(X)\leq\sigma_\e(X)\leq\nu_{\e/2}(X)$. \end{lemma}

\begin{proof}
Let $\Gamma$ be an $\e$-separated set of maximal size. Then in particular it is maximal. It follows that it is an $\e$-net. This proves the first inequality. 

Now let $\Delta$ be an $(\e/2)$-net. Then the balls of radius $\e/2$ about the points of $\Delta$ cover $X$, and no $\e$-separated set can contain more than one element in any of these balls. This proves the second inequality.
\end{proof}

\begin{lemma} \label{product}
Let $X$ and $Y$ be subsets of metric spaces and let $d$ be the metric on $X\times Y$ defined by $d((x,y),(x',y'))=d(x,x')\vee d(y,y')$. Then $\nu_\e(X\times Y)\leq\sigma_{\e/2}(X)\sigma_{\e/2}(Y)$.
\end{lemma}

\begin{proof}
By Lemma \ref{nusigma}, we have that
\[\nu_\e(X\times Y)\leq\nu_{\e/2}(X)\nu_{\e/2}(Y)\leq\sigma_{\e/2}(X)\sigma_{\e/2}(Y).\]
\end{proof}

\begin{lemma} \label{inverse}
Let $X$ be a subset of a metric group and let $\e>0$. Then $\nu_\e(X)=\nu_\e(X^{-1})$ and $\sigma_\e(X)=\sigma_\e(X^{-1})$.
\end{lemma}

\begin{proof} 
This is an immediate consequence of the fact that 
\[d(x,y)=d(y,x)=d(e,y^{-1}x)=d(x^{-1},y^{-1})\] 
for any two elements $x,y$ of a metric group.
\end{proof}

We shall write $\overline\nu_\e(X)$ for the size of the smallest non-strict $\e$-net of $X$ -- that is, of the smallest set $\Delta$ such that for every $x\in X$ there exists $y\in\Delta$ with $d(x,y)\leq\e$.

\begin{lemma} \label{p3lemma}
Let $X,Y,Z$ be 1-separated subsets of a metric group $G$, let $\d<\tfrac{1}{100}$, let $\e<1/6$, and suppose that $|Z|\leq\d^{-1}|X|^{1/2}|Y|^{1/2}$ and that $d(xy,Z)\leq\e$ for at least $\d|X||Y|$ pairs $(x,y)\in X\times Y$. Then there are subsets $X'\subset X$ and $Y'\subset Y$ with $|X'|\ge\d^{7}|X|$ and $|Y'|\ge\d^7|Y|$ such that $\overline\nu_{6\e}({X'}{Y'})\le\d^{-16}|X|^{1/2}|Y|^{1/2}$ and such that $d(xy,Z)\leq\e$ for at least $\d|X'||Y'|/4$ pairs $(x,y)\in X'\times Y'$.
\end{lemma}

\begin{proof}
Form a bipartite graph $G$ with vertex sets $X,Y$ by joining $x$ to $y$ if and only if $d(xy,z)\leq\e$. Then by hypothesis $G$ has density $\d$.

We shall apply Lemma~\ref{pa:lemtechnical}, but in order to do so we must first balance the sizes of the vertex sets. Suppose without loss of generality that $|X|\le |Y|$. From the above discussion, we recall that $|Y|\le \delta^{-4}|X|$. We now discard vertices of minimal degree from $Y$ one by one, until we arrive at a subset $Y_1\subset Y$ with $|Y_1|=|X|$. The edge density of the graph $G|_{X\times Y_1}$ is still at least $\delta$.

Applying Lemma \ref{pa:lemtechnical} with $k=1$, we can find $X'\subset X$ and $Y'\subset Y_1$ with $|X'|\ge \delta^2|X|/16\ge \d^7|X|$ and $|Y'|\geq\d^{2}|Y_1|/16\ge \delta^6|Y|/16\ge \d^7|Y|$ such that between any $x\in X'$ and $y\in Y'$ there are at least $\d^{9}|X||Y|$ paths of length 3 (with the two vertices in between not required to live in $X'$ and $Y'$) {and such that the graph $G|_{X'\times Y'}$ has density at least $\d/4$}. 

For each $x\in X'$ and $y\in Y'$, let $T(x,y)$ be the set of triples $(z_1,z_2,z_3)\in Z^3$ such that there exist $x_1\in X$ and $y_1\in Y$ with $d(xy_1,z_1), d(x_1y_1,z_2)$ and $d(x_1y,z_3)$ all at most $\e$. Since $X$, $Y$ and $Z$ are all 1-separated, there is a bijection between triples in $T(x,y)$ and paths of length 3 from $x$ to $y$ in the graph, so each set $T(x,y)$ has size at least $\d^{9}|X||Y_1|\ge \d^{13}|X||Y|$. 

Suppose now that $(z_1,z_2,z_3)$ belongs to $T(x,y)$ and $x_1,y_1$ are as above. Then from the three approximate relations and the fact that $$xy=xy_1(x_1y_1)^{-1}x_1y,$$ 
it follows that
\[d(xy,z_1z_2^{-1}z_3)\leq 3\e.\]

Now let $\Gamma=\{(x_1y_1),\dots,(x_my_m)\}$ be a $6\e$-separated subset of $X'Y'$. Then the balls of radius $3\e$ about the $x_iy_i$ are disjoint, from which it follows that the sets $T(x_i,y_i)$ are disjoint. But each one has size at least $\d^{13}|X||Y|$ and their union has size at most $|Z|^3$, so $m\le\d^{-13}|Z|^3|X|^{-1}|Y|^{-1}\le\d^{-16}|X'|^{1/2}|Y'|^{1/2}$. This bound holds for all $6\e$-separated subsets, so the result now follows from Lemma \ref{nusigma}.
\end{proof}

We remark that since $X$ and $Y$ are $1$-separated sets, we could if we wanted replace the cardinalities $|X'|$ and $|Y'|$ in the statement above by the quantities $\sigma_1(X')$ and $\sigma_1(Y')$. 

One of the main results of \cite{taopaper} is that if $X,Y$ are finite subsets of a group and $|XY|\leq C|X|^{1/2}|Y|^{1/2}$, then there exists an approximate group $H$ and sets $K,L$ of bounded size such that $X\subset KH$ and $Y\subset HL$. (One can of course take $K$ and $L$ to be the same by taking their union.) In the next subsection, we shall prove an analogous statement for our metric-entropy context. 

\subsection{Products with small metric entropy come from rough approximate groups}\label{appB}

The main theorem we prove in this subsection is the following metric-entropy variant of Theorem 4.6 of \cite{taopaper}.

\begin{theorem} \label{smalldoubling}
Let $G$ be a metric group, let $\b\geq 2048\e$, and let $X,Y\subset G$ be subsets such that $\nu_\e(XY)\leq C\sigma_\b(X)^{1/2}\sigma_\b(Y)^{1/2}$. Then there exists a $(16C^{16},256\e)$-rough approximate group $H\subset G$ and sets $K,L$ of sizes at most $256C^{32}$ and $2048C^{48}$, respectively, such that $KH$ is a {$584\e$-}net of $X$, $HL$ is a {$2304\e$-}net of $Y$, and $\nu_{128\e}(H)\leq 8C^{15}\sigma_\b(X)^{1/2}\sigma_\b(Y)^{1/2}$.
\end{theorem}

We begin with an analogue of the Ruzsa triangle inequality (which can also be found in \cite{taobp}).

\begin{lemma} \label{triangle}
Let $G$ be a metric group and let $U,V,W$ be subsets of $G$. Then $\nu_\e(U)\nu_\e(VW^{-1})\leq\sigma_{\e/4}(UV^{-1})\sigma_{\e/4}(UW^{-1})$.
\end{lemma}

\begin{proof}
Let $\Gamma_1$ be an $\e$-separated subset of $U$ and let $\Gamma_2$ be an $\e$-separated subset of $VW^{-1}$. Define $\phi:\Gamma_1\times\Gamma_2\to UV^{-1}\times UW^{-1}$ by choosing for each $x\in\Gamma_2$ a pair of elements $(v(x),w(x))\in V\times W$ such that $v(x)w(x)^{-1}=x$, and then for each $(u,x)\in\Gamma_1\times\Gamma_2$ defining $\phi(u,x)$ to be $(uv(x)^{-1},uw(x)^{-1})$. 

Suppose now that $(u_1,x_1)$ and $(u_2,x_2)$ are elements of $\Gamma_1\times\Gamma_2$ such that $d(\phi(u_1,x_1),\phi(u_2,x_2))<\d$, where for our product metric we take the maximum of the metrics on $UV^{-1}$ and $UW^{-1}$. Then $d(u_1v(x_1)^{-1},u_2v(x_2)^{-1})<\d$ and $d(u_1w(x_1)^{-1},u_2w(x_2)^{-1})<\d$. Since $G$ is a metric group, it follows that 
\begin{align*}d(x_1,x_2)&=d(v(x_1)w(x_1)^{-1},v(x_2)w(x_2)^{-1})\\
&=d(v(x_1)u_1^{-1}u_1w(x_1)^{-1},v(x_2)u_2^{-1}u_2w(x_2)^{-1})\\
&<\d+\d=2\d.\\
\end{align*}
Therefore, if $\d\leq\e/2$ we can deduce that $x_1=x_2$, since they are both elements of $\Gamma_2$. But then $d(u_1,u_2)=d(u_1v(x_1)^{-1},u_2v(x_1)^{-1})<\d$, which implies that $u_1=u_2$ as well.

Since $\Gamma_1$ and $\Gamma_2$ were arbitrary $\e$-separated subsets, it follows that 
$$\sigma_\e(U)\sigma_\e(VW^{-1})\leq\sigma_{\e/2}(UV^{-1}\times UW^{-1}),$$ and hence by Lemmas \ref{nusigma} and \ref{product}, that $$\nu_\e(U)\nu_\e(VW^{-1})\leq\sigma_{\e/4}(UV^{-1})\sigma_{\e/4}(UW^{-1}).$$
\end{proof}

\begin{corollary} \label{sumdiff}
Let $\e,\d>0$ and let $X,Y$ be a subsets of a metric group such that $\nu_\e(XY)\leq C\sigma_{\d}(X)^{1/2}\sigma_{16\e}(Y)^{1/2}$. Then $\nu_{8\e}(XX^{-1})\leq C^2\sigma_{\d}(X)$.
\end{corollary}

\begin{proof}
By Lemma \ref{triangle}, Lemma \ref{inverse} and our hypothesis, we have that
\[\nu_{8\e}(Y^{-1})\nu_{8\e}(XX^{-1})\leq\sigma_{2\e}(Y^{-1}X^{-1})^2=\sigma_{2\e}(XY)^2\leq\nu_{\e}(XY)^2\leq C^2\sigma_{\d}(X)\sigma_{16\e}(Y).\]
By Lemmas \ref{inverse} and \ref{nusigma}, $\nu_{8\e}(Y^{-1})=\nu_{8\e}(Y)\geq\sigma_{16\e}(Y)$, so the result follows.
\end{proof}

Our next lemma is a version of the Ruzsa covering lemma.

\begin{lemma} \label{covering}
Let $\e>0$ and let $A,B$ be subsets of a metric group such that $\nu_\e(AB)\leq C\sigma_{2\e}(B)$. Then there exists a set $K$ of size at most $C$ such that $KBB^{-1}$ is a $2\e$-net of $A$.
\end{lemma}

\begin{proof}
Let $K\subset A$ be maximal such that for any two distinct elements $x,x'\in K$ the distance between the sets $xB$ and $x'B$ is at least $2\e$. Then if $y\in A$ there must be some $x\in K$ such that $d(xB,yB)<2\e$, by maximality, from which it follows that $d(y,xBB^{-1})<2\e$. Therefore, $KBB^{-1}$ is a $2\e$-net of $A$. 

Now let $\Gamma$ be a $2\e$-separated subset of $B$. Then $K\Gamma$ is a $2\e$-separated subset of $KB$, which is contained in $AB$. It follows that $K\sigma_{2\e}(B)\leq\sigma_{2\e}(AB)$, which by Lemma \ref{nusigma} is at most $\nu_\e(AB)$. By hypothesis this is at most $C\sigma_{2\e}(B)$ and the result follows.
\end{proof}

Next we need a notion of `popular differences' that will be suitable for this metric-entropy context. 

\begin{definition}
Let $A$ be a subset of a metric group. We say that an element $d\in A^2$ is $(\e,\d,m)$-\emph{popular} if there are $m$ pairs $(x_i,y_i)\in A^2$ such that the sets $\{x_1,\dots,x_m\}$ and $\{y_1,\dots,y_m\}$ are $\d$-separated and $d(y_i^{-1}x_i,d)<\e$ for every $i$, 
\end{definition}

\begin{lemma} \label{popdiff}
Let $\d\geq 2\e$, let $A$ be a subset of a metric group such that $\nu_\e(AA^{-1})\leq{C}\sigma_\d(A)$ and let $S$ be the set of $(2\e,\d,\sigma_\d(A)/2C)$-popular elements of $A^{-1}A$. Then $\sigma_\d(S)\geq\sigma_\d(A)/2C$.
\end{lemma}

\begin{proof}
Let $\Gamma$ be a $\d$-separated subset of $A$ of size $\sigma_\d(A)$. Choose a partition of $AA^{-1}$ into $\nu_\e(AA^{-1})$ sets, each contained in an open ball of radius ${\e}$, and write $z\sim w$ if $z$ and $w$ belong to the same cell of the partition. 

If we choose a random cell from the partition, then the expected number of pairs $(x_1,x_2)\in\Gamma^2$ with $x_1x_2^{-1}$ in that cell is at least $\sigma_\d(A)^2/\nu_\e(AA^{-1})$. It follows that there are at least $\sigma_\d(A)^4/\nu_\e(AA^{-1})\geq\sigma_\d(A)^3/C$ quadruples $(x_1,x_2,x_3,x_4)\in\Gamma^4$ such that $x_1x_2^{-1}\sim x_3x_4^{-1}$, and hence, since the cells are contained in balls of radius $\e$, such that $d(x_3^{-1}x_1,x_4^{-1}x_2)<2\e$. It follows that for a randomly chosen $(x_1,x_3)\in\Gamma^2$ the expected number of pairs $(x_2,x_4)\in\Gamma^2$ such that $d(x_3^{-1}x_1,x_4^{-1}x_2)<{2}\e$ is at least $\sigma_\d(A)/C$. Since $\d\geq 2\e$, it is not possible to find $x,y,z\in\Gamma$ such that $x^{-1}y=x^{-1}z$ or such that $x^{-1}z=y^{-1}z$. It follows that the maximum number of pairs $(x_2,x_4)$ with $d(x_3^{-1}x_1,x_4^{-1}x_2)<{2}\e$ is at most $\sigma_\d(A)$. Therefore, there are at least $\sigma_\d(A)^2/2C$ pairs $(x_1,x_3)\in\Gamma$ such that $x_3^{-1}x_1$ is $(2\e,\d,\sigma_\d(A)/2C)$-popular.

By averaging we can find some $x_i$ for which there are at least $\sigma_\d(A)/2C$ popular pairs $(x_i,x_j)$. If $(x_i,x_j)$ and $(x_i,x_k)$ are two distinct such pairs, then $d(x_j^{-1}x_i,x_k^{-1}x_i)\geq\d$. It follows that there is a $\d$-separated subset of $S$ of size at least $\sigma_\d(A)/2C$, as claimed.
\end{proof}

\begin{lemma} \label{plunnecke}
Let $\d\geq 4\e$, let $A$ be a subset of a metric group such that $\nu_\e(AA^{-1})\leq C\sigma_\d(A)$ and let $S$ be the set of $(2\e,\d,\sigma_\d(A)/2C)$-popular elements of $A^{-1}A$. Then $\nu_{16\e}(AS^3A^{-1})\leq 8C^7\sigma_\d(A)$.
\end{lemma}

\begin{proof}
Let $x_0,x_7$ be elements of $A$ and let $d_1,d_2,d_3\in S$. Since each $d_i$ is popular, we can approximate $x_0d_1d_2d_3x_7^{-1}$ as $x_0x_1^{-1}x_2x_3^{-1}x_4x_5^{-1}x_6x_7^{-1}$ in several ways. More precisely, for each $i=1,3,5$ we have at least $\sigma_\d(A)/2C$ independent choices for the pair $(x_i,x_{i+1})$, and the individual coordinates of these choices form $\d$-separated sets.

Each such product gives us an element $(x_0x_1^{-1},x_2x_3^{-1},x_4x_5^{-1},x_6x_7^{-1})$ of the set $(AA^{-1})^4$. If $(x_0x_1^{-1},x_2x_3^{-1},x_4x_5^{-1},x_6x_7^{-1})$ and $(x_0x_1'^{-1},x_2'x_3'^{-1},x_4'x_5'^{-1},x_6'x_7^{-1})$ are two different such quadruples, then if their first $i$ coordinates agree and the $(i+1)$st coordinate is different, then $x_j=x_j'$ for $0\leq j<2i$, and hence for $j=2i$ as well, so we find that the two $(i+1)$st coordinates are $x_{2i}x_{2i+1}^{-1}$ and $x_{2i}x_{2i+1}'^{-1}$, which are separated by at least $\d\geq 4\e$.

We also have that if two elements of $AS^3A^{-1}$ are separated by at least $16\e$ and for each one we choose a quadruple as above, then at least one coordinate of the two quadruples will be separated by at least $4\e$, since the products of the two quadruples give the two elements.

It follows that 
\[\sigma_{16\e}(AS^3A^{-1})(\sigma_\d(A)/2C)^3\leq\sigma_{4\e}((AA^{-1})^4)\leq\nu_\e(AA^{-1})^4.\]
Since $\nu_\e(AA^{-1})\leq C\sigma_\d(A)$, this implies the result.
\end{proof}

\begin{lemma} \label{getrag}
Let $\d\geq 2\e$ and let $S$ be a subset of a metric group such that $S=S^{-1}$ and $\nu_\e(S^3)\leq C\sigma_\d(S)$. Then $S^2$ is a $(C,2\e)$-rough approximate group.
\end{lemma}

\begin{proof}
By Lemma \ref{covering} with $A=S^2$ and $B=S$ there is a set $K$ of size at most $C$ such that $KS^2$ is a $2\e$-net of $S^2$. 
\end{proof}

\begin{lemma} \label{coverA}
Let $\d\geq 2\e_1$, let $A$ be a subset of a metric group, let $H$ be a $(C_2,\e_2)$-rough approximate group, and suppose that $\nu_{\e_1}(AH)\leq C_1\sigma_\d(H)$. Then there is a set $K$ of size at most $C_1C_2$ such that $KH$ is a $(2\e_1+\e_2)$-net of $A$.
\end{lemma}

\begin{proof}
By Lemma \ref{covering} there is a set $K_1$ of size at most $C_1$ such that $K_1H^2$ is a $2\e_1$-net of $A$. By the definition of an approximate group there is also a set $K_2$ of size at most $C_2$ such that $K_2H$ is an $\e_2$-net of $H$. But then $K_1K_2H$ is a $(2\e_1+\e_2)$-net of $A$. 
\end{proof}

\noindent \textbf{Proof of Theorem \ref{smalldoubling}.}

If $X,Y$ are subsets of a metric group and $\nu_\e(XY)\leq C\sigma_\b(X)^{1/2}\sigma_\b(Y)^{1/2}$, then by Corollary \ref{sumdiff} we have the inequality $\nu_{8\e}(XX^{-1})\leq C^2\sigma_\b(X)$. By Lemmas \ref{popdiff} and \ref{plunnecke} we obtain a set $S$ with $S=S^{-1}$ and $\sigma_\b(S)\geq\sigma_\b(X)/2C^2$ such that $\nu_{128\e}(XS^3X^{-1})\leq 8C^{14}\sigma_\b(X)$. 

It follows that $\nu_{128\e}(S^3)\leq 16C^{16}\sigma_\b(S)$. Therefore, by Lemma \ref{getrag}, $S^2$ is a $(16C^{16},256\e)$-rough approximate group.

We also have that $\nu_{128\e}(XS^2)\leq 16C^{16}\sigma_\b(S^2)$. Therefore, by Lemma \ref{coverA} there is a set $K$ of size at most $256C^{32}$ such that $KS^2$ is a {$512\e$-}net of $X$.

By Lemma \ref{triangle}, 
\begin{align*}
\nu_{1024\e}(X)\nu_{1024\e}(S^2Y)&\leq\sigma_{256\e}(XS^2)\sigma_{256\e}(XY)\\
&\leq\nu_{128\e}(XS^2)\nu_{128\e}(XY)\\
&\leq 16C^{16}\sigma_\b(S^2).C\sigma_\b(X)^{1/2}\sigma_\b(Y)^{1/2}.\\
\end{align*}
But 
\[\sigma_\b(X)\leq\sigma_\b(XY)\leq\nu_{\b/2}(XY)\leq C\sigma_\b(X)^{1/2}\sigma_\b(Y)^{1/2},\] 
so $\sigma_\b(X)\leq C^2\sigma_\b(Y)$ and therefore $\sigma_\b(X)^{1/2}\sigma_\b(Y)^{1/2}\leq C\sigma_\b(Y)$. Also, since $\b\geq 2048\e$,
\[\sigma_\b(S^2)\leq\nu_{128\e}(XS^3X)\leq 8C^{14}\sigma_\b(X)\leq 8C^{14}\nu_{1024\e}(X).\]
It follows that $\nu_{1024\e}(Y^{-1}S^2)=\nu_{1024\e}(S^2Y)\leq 128C^{32}\sigma_\b(Y)$.

Therefore, by Lemma \ref{coverA} again it follows that there is a set $L$ of size at most $2048C^{48}$ such that $LS^2$ is a {$2304\e$-}net of $Y^{-1}$, which implies that $S^2L^{-1}$ is a {$2304\e$-}net of $Y$. \hfill $\square$
\vspace{0.3cm}

We conclude this appendix by combining Lemma \ref{p3lemma} and Theorem \ref{smalldoubling}. We shall present the result (mostly) without explicit constants, but it is not hard to obtain them.

\begin{theorem}\label{gettingrag}
Let $X,Y, Z$ be 1-separated subsets of a metric group $G$, let $0<\d<1/100$, let $\e>0$ be sufficiently small, and suppose that $|Z|\leq\d^{-1}|X|^{1/2}|Y|^{1/2}$ and that $d(xy,Z)\leq\e$ for at least $\d|X||Y|$ pairs $(x,y)\in X\times Y$. Then there exist subsets $X''\subset X$, $Y''\subset Y$ {and $Z''\subset Z$} with $|X''|=\d^{O(1)}|X|$, $|Y''|=\d^{O(1)}|Y|$ {and $|Z''|=\delta^{O(1)}|Z|$}, a $(\d^{-O(1)}, O(\e))$-rough approximate group $H\subset G$, and elements $u,v,{w}$ of $G$ such that $\nu_{O(\e)}(H)=\d^{-O(1)}|X|^{1/2}|Y|^{1/2}$, $X''\subset (uH)_{O(\e)}$, $Y''\subset (Hv)_{O(\e)}$, {$Z''\subset (X''Y'')_\e\cap(uwHv)_{O(\e)}$} {and $d(xy,Z'')\le\e$ for $\delta^{O(1)}|X''||Y''|$ pairs $(x,y)\in X''\times Y''$.}
\end{theorem}
\begin{proof}
Lemma \ref{p3lemma} gives us $X'\subset X$ and $Y'\subset Y$ with $|X'|\geq\d^7|X|$ and $|Y'|\geq\d^7|Y|$ such that $\nu_{O(\e)}(X'Y')=\d^{-O(1)}|X|^{1/2}|Y|^{1/2}$ {and such that $d(xy,Z)\le\e$ for $\delta^{O(1)}|X'||Y'|$ pairs $(x,y)\in X'\times Y'$}. Applying Theorem \ref{smalldoubling} (with $\b=1$), we obtain a $(\d^{-O(1)},O(\e))$-rough approximate group $H\subset G$ and sets $K,L$ of sizes $\d^{-O(1)}$ such that $X'\subset(KH)_{O(\e)}$ and $Y'\subset(HL)_{O(\e)}$. 

{We will pick $u\in K$ and $v\in L$ at random, and let $X''=X'\cap (uH)_{O(\e)}$ and $Y''=Y'\cap (Hv)_{O(\e)}$. By averaging there are choices $u\in K$ and $v\in L$ such that $|X''|= \d^{O(1)}|X|$, $|Y''|=\d^{O(1)}|Y|$ and $d(xy,Z)\le\e$ for $\delta^{O(1)}|X''||Y''|$ pairs $(x,y)\in X''\times Y''$.}

{Observe that since $X''\subset (uH)_{O(\e)}$ and $Y''\subset (Hv)_{O(\e)}$, we have that $X''Y''\subset (uHHv)_{O(\e)}$. Since $H$ is a $(\delta^{-O(1)},O(\e))$-rough approximate subgroup of $G$, this means that there exists a set $M\subset G$ of size $\delta^{-O(1)}$ such that $X''Y''\subset (uMHv)_{O(\e)}$.}

{Since $X''Y''\subset(uMHv)_{O(\e)}$, we have that $(X''Y'')_\e\subset(uMHv)_{O(\e)}$. Let 
	$$Z'=Z\cap(X''Y'')_\e \subset(uMHv)_{O(\e)}$$
	and observe that $d(xy,Z')\le\e$ for $\delta^{O(1)}|X''||Y''|$ pairs $(x,y)\in X''\times Y''$.}

{Now we choose $w\in M$ uniformly at random, and let $Z''=Z'\cap(uwHv)_{O(\e)}$. Since $|M|=\delta^{-O(1)}$, we have in expectation that $d(xy,Z'')\le\e$ for $\delta^{O(1)}|X''||Y''|$ pairs $(x,y)\in X''\times Y''$. Suppose without loss of generality that $|X|\ge |Y|$. If $d(xy,Z'')\le\e$ for at least $\delta^{O(1)}|X''||Y''|$ pairs $(x,y)\in X''\times Y''$, then there exists a choice of $y\in Y''$ such that $d(xy,Z'')\le\e$ for $\delta^{O(1)}|X''|$ choices of $x\in X''$. Since $X''$ is 1-separated, this implies that $|Z''|= \delta^{O(1)}|X''|= \delta^{O(1)}|Z|$. Therefore there is some choice of $w\in M$ satisfying our requirements.}
\end{proof}	

\section{A Bogolyubov-type lemma for $\so3$.} \label{bog}

In this section we look at properties of product sets of dense subsets of $\so3$. The main result we shall prove is the following lemma. Once we have it, it will enable us to prove that the partial binary operation on a maximal $\d$-separated subset of $\so3$ described in the introduction is defined for a constant proportion of pairs and gives rise to within a constant of the maximum possible number of associative triples.

\begin{lemma} \label{bogolyubov}
For every $\theta>0$ and $\e>0$ there exists $\eta>0$ such that if $A$ is any subset of $\so3$ of Haar measure at least $\theta$, then $AA^{-1}$ contains a proportion $1-\e$ of a ball of radius $\eta$ about the identity.
\end{lemma}

It follows straightforwardly that $AA^{-1}AA^{-1}$ contains the whole of a ball of some radius $\eta(\theta)$. Thus, this result can be thought of as a Bogolyubov lemma~\cite{Bog} for $\so3$, with balls about the identity playing the role of Bohr sets. Note that unlike in the Abelian case, there is an extra uniformity here: the ball we obtain depends only on the measure of $A$ and not on $A$ itself. This fact, which can be thought of as saying that the only departure from quasirandomness of the group $\so3$ is the obvious one that a product of two small balls is contained in a small ball, will be essential to our argument. A corollary of this result will be a statement that we claimed earlier in the paper: that if $\Gamma$ is a maximal $\d$-separated subset of $\so3$ and $\circ:\Gamma\times\Gamma\to\Gamma$ is a partially defined operation where $x\circ y=z$ if and only if $xy$ is close to $z$, then $\circ$ is defined for a dense set of pairs. (We give a precise formulation later.) 

To prove Lemma \ref{bogolyubov} we shall begin, as one might expect, by imitating the proof of Bogolyubov's lemma, using non-Abelian Fourier analysis. This will show that the structure of the convolution $\bbA*\bbAA$ essentially depends on the large Fourier coefficients of $\bbA$. As is well known, these all come from the low-dimensional representations of $\so3$: the further uniformity mentioned above comes from the fact that the number of low-dimensional representations is bounded. To prove the assertion about the ball, we use the fact that the low-dimensional representations can be described explicitly as follows. Every irreducible representation has odd dimension, and for each odd dimension there is exactly one irreducible representation, which is given by the action of $\so3$ on the space of spherical harmonics of degree $d$. (See for example \cite{KSS}.) 

Let us briefly recall the basic facts about non-Abelian Fourier analysis that we shall need. Given an integrable function $f:\so3\to\C$ and an irreducible representation $\rho$ of $\so3$, we define the Fourier coefficient $\hat f(\rho)$ by the formula
\[\hat f(\rho)=\E_xf(x)\overline{\rho(x)},\] 
where we are writing $\E_x$ for the average with respect to Haar measure on $\so3$. Note that if $\rho$ is a $k$-dimensional representation, then $\hat f(\rho)$ is a $k\times k$ matrix. 

The non-Abelian versions of Parseval's identity, the convolution identity, and the inversion formula are as follows. Parseval's identity states that for any two square-integrable functions $f,g:\so3\to\C$, 
\[\int_{\so3}f(x)\overline g(x)\,dx=\sum_\rho n_\rho\tr(\hat f(\rho)\hat g(\rho)^*),\]
where the sum is over all irreducible representations and $n_\rho$ is the dimension of $\rho$. The left-hand side is the obvious definition of the inner product of $f$ and $g$. As for the right-hand side, the matrix inner product $\langle A,B\rangle$ of two $k\times k$ matrices $A$ and $B$ is $\tr(AB^*)=\sum_{ij}A_{ij}B_{ij}^*$, so we can rewrite it as $\sum_\rho n_\rho\langle \hat f(\rho),\hat g(\rho)\rangle$, which is a natural way of defining the inner product $\langle\hat f,\hat g\rangle$. So, suitably interpreted, Parseval's identity is just the usual identity $\langle f,g\rangle=\langle\hat f,\hat g\rangle$. 

The convolution identity is also the same as it is in the Abelian case: $\widehat{f*g}(\rho)=\hat f(\rho)\hat g(\rho)$. Of course, here the product $\hat f(\rho)\hat g(\rho)$ is a matrix product. 

Finally, the inversion formula is
\[f(x)=\sum_\rho n_\rho\tr(\hat f(\rho)\overline{\rho(x)^*}),\]
where the equality is valid almost everywhere. 

The \emph{Hilbert-Schmidt norm} of a complex matrix $A$ is defined by the formula
\[\|A\|_{HS}^2=\tr(AA^*)=\sum_{x,y}|A(x,y)|^2.\]
The \emph{box norm} is defined by the formula 
\[\|A\|_\square^4=\sum_{x,x',y,y'}A(x,y)\overline{A(x,y')}\,\overline{A(x',y)}A(x',y').\]
It is also equal to $\tr(AA^*AA^*)=\langle AA^*,AA^*\rangle=\|AA^*\|$. 

A Cauchy-Schwarz argument shows that $\|AB\|_{HS}\leq\|A\|_\square\|B\|_\square$, and it is also not hard to prove that $\|A\|_\square\leq\|A\|_{HS}$. 

Now let us apply these facts to say something about the convolution $f*g$ of two bounded measurable functions on $\so3$. (By `bounded' we mean that $\|f\|_\infty,\|g\|_\infty\leq 1$.) By Parseval's identity we have that
\[\|f*g\|_2^2=\sum_\rho n_\rho\|\hf(\rho)\hg(\rho)^*\|_{HS}^2\]
The generalized Cauchy-Schwarz inequality for the box norm implies that $\|AB\|_{HS}^2\leq\|A\|_\square^2\|B\|_\square^2$, so the right-hand side is at most
\[\sum_\rho n_\rho\|\hf(\rho)\|_\square^2\|\hg(\rho)\|_\square^2\leq\sum_\rho n_\rho\|\hf(\rho)\|_{HS}^2\|\hg(\rho)\|_{HS}^2.\]
Also, the convolution identity and inversion formula together imply that
\[f*g(x)=\sum_\rho n_\rho\tr(\hf(\rho)\hg(\rho)\overline{\rho(x)^*}).\]
Let us fix a constant $C$ and split the right-hand side into the two functions
\[u(x)=\sum_{\rho:n_\rho\leq C} n_\rho\tr(\hf(\rho)\hg(\rho)\overline{\rho(x)^*}).\]
and
\[v(x)=\sum_{\rho:n_\rho>C} n_\rho\tr(\hf(\rho)\hg(\rho)\overline{\rho(x)^*}).\]
By Parseval's identity, 
\[\|v\|_2=\sum_{\rho:n_\rho>C}n_\rho\|\hf(\rho)\hg(\rho)\|_{HS}^2.\]
Also, Parseval's identity implies that $\|\hat f(\rho)\|_{HS}^2\leq n_\rho^{-1}\|f\|_2^2\leq n_\rho^{-1}$ for any bounded function $f$, so 
\begin{align*}
\sum_{\rho:n_\rho>C}n_\rho\|\hf(\rho)\hg(\rho)\|_{HS}^2&=\sum_{\rho:n_\rho>C}\|\hf(\rho)\|_\square^2\|\hg(\rho)\|_\square^2\\
&\leq \left(\sum_{\rho:n_\rho>C}n_\rho\|\hf(\rho)\|_{\square}^4\right)^{1/2}\left(\sum_{\rho:n_\rho>C}n_\rho\|\hg(\rho)\|_{\square}^4\right)^{1/2}\\
&\leq \left(\sum_{\rho:n_\rho>C}n_\rho\|\hf(\rho)\|_{HS}^4\right)^{1/2}\left(\sum_{\rho:n_\rho>C}n_\rho\|\hg(\rho)\|_{HS}^4\right)^{1/2}\\
&\leq C^{-2}\left(\sum_{\rho:n_\rho>C}n_\rho\|\hf(\rho)\|_{HS}^2\right)^{1/2}\left(\sum_{\rho:n_\rho>C}n_\rho\|\hg(\rho)\|_{HS}^2\right)^{1/2}\\
&\leq C^{-2}\|f\|_2\|g\|_2.\\
\end{align*}

It follows that if $C$ is large, then the function $f*g$ is well approximated in $L_2(\so3)$ by the function $u$ defined above, which was the part that comes from the representations of dimension at most $C$. 

Now let $B$ be a ball of radius $\eta$, where $\eta>0$ is a constant to be chosen later, and let $\mu_B$ be the characteristic measure of $B$. That is, if $B$ has Haar measure $\b$, then $\mu_B(x)=\b^{-1}$ for $x\in B$ and $\mu_B(x)=0$ otherwise. We shall show that if $\eta$ is sufficiently small, then $\|u-u*\mu_B\|_\infty$, and hence $\|u-u*\mu_B\|_2$, is small. 

By the convolution law and the inversion formula, 
\begin{equation}\label{eq1}
u(x)-u*\mu_B(x)=\sum_{\rho: n_\rho\leq C} n_\rho\tr(\hf(\rho)\hg(\rho)(I_{n_\rho}-\widehat{\mu_B}(\rho))\overline{\rho(x)^*}).
\end{equation}
In order to bound the size of the right-hand side, we shall show that if $\eta$ is small enough, then $\widehat{\mu_B}(\rho)$ is close to the identity (on $\C^{n_\rho}$) for all irreducible representations $\rho$ of dimension at most $C$.

By definition,
\[\widehat{\mu_B}(\rho)=\mathop{\E}_{x\in\so3}\mu_B(x)\overline{\rho(x)}=\mathop{\E}_{x\in B}\overline{\rho(x)}.\]

One can show easily that the $d$-dimensional spherical harmonics are equicontinuous: for instance, it follows from the fact that the space of $d$-dimensional spherical harmonics is compact when considered as a subset of $C(\so3)$. (This is the easy direction of the Arz\`ela-Ascoli theorem. If one wants, one can obtain estimates for the equicontinuity by using explicit formulae for the spherical harmonics, but we shall content ourselves with a qualitative statement here.) Therefore, for every $\e>0$ and every irreducible representation $\rho$ of dimension $n_\rho=2d+1$ there exists $\eta>0$ such that if $x$ is sufficiently close to the identity in $\so3$, then $\langle\rho(x)p,p\rangle\geq 1-\eta$ for every spherical harmonic $p$ of dimension $d$. It follows by averaging over all $p$ that $n_\rho^{-1}\tr\rho(x)\geq 1-\eta$, which implies that $\|\rho(x)-I_{n_\rho}\|_{HS}^2\leq 2\eta n_\rho$, and therefore that $\|\rho(x)-I_{n_\rho}\|_{\op}^2\leq 2\eta n_\rho$ as well.

It follows that for every $\d>0$ and every $C$ we may choose $\eta>0$ such that $\|\rho(x)-I_{n_\rho}\|_\op\leq\d$ for every $x\in B$ and every irreducible representation $\rho$ of $\so3$ of dimension at most $C$. This in turn implies by averaging that $\|\widehat{\mu_B}(\rho)-I_{n_\rho}\|_{\op}\leq\d$ for every such $\rho$, where $B$ is the ball of radius $\eta$ about the identity. But then in the right-hand side of (\ref{eq1}) we are taking the trace of a product of four matrices of which three have operator norm at most 1 and one has operator norm at most $\d$. It follows that the trace is at most $\d n_\rho$, and therefore that the right-hand side is in total at most $\d\sum_{\rho:n_\rho\leq C}n_\rho^2$. Choosing $\d$ in such a way that this sum is at most $\e/2$ (which we can do with $\d$ depending on $\e$ and $C$ only) we obtain that $|u(x)-u*\mu_B(x)|\leq\e$ for every $x\in\so3$, which implies that $\|u-u*\mu_B\|_2\leq\e/2$.

If we choose $C$ such that $C^{-2}\leq\e/4$, then $\|v-v*\mu_B\|_2\leq 2\|v\|_2\leq\e/4$. Since $f*g=u+v$, it follows that
\[\|f*g-f*g*\mu_B\|_2\leq\e.\]

Let us state this result formally for later reference. When we say `the ball of radius $\eta$' this is to be understood to be the ball with respect to any reasonable metric, such as the one coming from the operator norm or the Hilbert-Schmidt norm -- the result is true for all of them.

\begin{lemma} \label{qrlemma}
For every $\e>0$ there exists $\eta>0$ such that the following statement holds. Let $f$ and $g$ be two bounded measurable complex-valued functions defined on $\so3$, let $B$ be the ball of radius $\eta$ around the identity in $\so3$ and let $\mu_B$ be the characteristic measure of $B$. Then
\[\|f*g-f*g*\mu_B\|_2\leq\e.\]
\end{lemma}

We can interpret this result as a kind of partial quasirandomness property of $\so3$. For a fully quasirandom group (as defined in~\cite{quasirandomgps}), we can replace $\mu_B$ by the constant function that takes value 1 everywhere and the lemma above holds. Thus, when two bounded functions are convolved, the resulting function depends, up to a small $L_2$ error, only on the averages of those functions. In $\so3$ we cannot say that, but we can say that the resulting function does not depend on the fine structure of $f$ and $g$ and only on the averages over balls of radius $\eta$, since one can replace $f$ by $f*\mu_B$ or $g$ by $g*\mu_B$ without having much effect on the answer. (Strictly speaking, we have not proved that $f*g$ is close to $f*\mu_B*g$, since $\so3$ is non-Abelian, but a very minor modification of the above argument will do this as well.) Thus, $\so3$ is `quasirandom at fine scales'. This observation will play an important role in our argument.

Now we give the promised proof that the approximate multiplication we defined earlier is densely defined.

\begin{lemma}\label{denselydefined}
For every $\theta\in(0,1)$ there exists $\d>0$ such that for any maximal $\d$-separated subset $\Gamma$ of $\so3$, the proportion of $(x,y)\in\Gamma^2$ with $d(xy,\Gamma)\leq\theta\d$ is at least $(\theta/3)^9/16$.
\end{lemma}

\begin{proof}
Let us write $\Gamma_\theta$ for the set of all points $x$ within distance $\theta\d$ of a point in $\Gamma$. Since $\Gamma$ is a maximal $\d$-separated set, it is also a $\d$-net, so the union of the balls of radius $\d$ about each point is all of $\so3$. Since the balls of radius $\theta\d$ are disjoint and each one occupies at least a proportion $\theta^3/2$ of the ball of radius $\d$ with the same centre (proving this is one detail that we omit, but it uses the fact that $\so3$ is a three-dimensional manifold), it follows that $\Gamma_\theta$ has Haar measure at least $\theta^3/2$. Furthermore, given any $\eta>0$, if $\d$ is sufficiently small, we have for every ball $B$ of radius $\eta$ in $\so3$ that $|\Gamma_\theta\cap B|\geq(\theta^3/2)|B|$, where we write $|A|$ for the Haar measure of a subset $A$, except that if $A$ is a finite set then $|A|$ will denote its cardinality. 

Let $f$ be the characteristic function of $\Gamma_\theta$. Then the last assertion is equivalent to the statement that $f*\mu(B)(x)$ is at least $\theta^{3}/2$ for every $x\in\so3$. But that implies that $f*f*\mu_B(x)$ is at least $(\theta^3/2)\E_xf(x)\geq\theta^6/4$ for every $x\in\so3$. 

Now we apply Lemma \ref{qrlemma}. For any $\g>0$ we can choose $\eta>0$ (depending on $\theta$ and $\e$ but not on $\d$) such that $\|f*f-f*f*\mu_B\|_2\leq\g$. From this it follows that
\[|\langle f*f,f\rangle-\langle f*f*\mu_B,f\rangle|\leq\g\|f\|_2.\]
But $\langle f*f*\mu_B,f\rangle\geq\theta^9/8$ by the estimates above, so for suitable choice of $\g$ we can ensure that $\langle f*f,f\rangle\geq\theta^9/16$.

The left-hand side of this inequality is the quantity $\E_{x,y}f(x)f(y)f(xy)$. It is non-zero if and only if all of $x,y$ and $xy$ are within $\theta\d$ of points $x',y',z'$ of $\Gamma$. Since all balls of radius $\theta\d$ have the same measure, it follows that the proportion of $(x',y')\in\Gamma^2$ such that $x'y'$ is within $3\theta\d$ of some point $z'\in\Gamma$ is at least $\theta^9/16$. Replacing $\theta$ by $\theta/3$ gives the lemma as stated.
\end{proof}

We make another observation that uses part of the proof above.

\begin{lemma}
For every $\e,\theta\in(0,1)$ there exists $\d>0$ such that for any maximal $\d$-separated subset $\Gamma$ of $\so3$, the proportion of $z\in\Gamma$ such that $d(xy,z)\leq\theta\d$ for at least $(\theta/3)^6|\Gamma|/8$ pairs $(x,y)\in\Gamma^2$ is at least $1-\e$.
\end{lemma}

\begin{proof}	
Choose $\eta$, and therefore $B$, as in the proof of the previous lemma. Then the measure of the set of $x$ such that $|f*f(x)-f*f*\mu_B(x)|<\theta^6/8$ is at least $1-64\g^2\theta^{-12}$. For each $x$ in this set, we have $f*f(x)\geq\theta^6/8$ (since $f*f*\mu_B$ is always at least $\theta^6/4$). Let us denote this set of `popular products' by $W$. 

Choosing $\g$ appropriately, we can ensure that the measure of $W$ is at least $1-\e\theta^3/2$. Since $|\Gamma_\theta|\geq\theta^3/2$, it follows that $|W\cap\Gamma_\theta|\geq(1-\e)|\Gamma_\theta|$. From this it follows that the proportion of $z\in\Gamma$ such that there exists $z'\in W$ with $d(z,z')\leq\theta\d$ is at least $1-\e$. But if $z'\in W$, then $f*f(z')\geq\theta^6/8$, which implies that the proportion of $x\in\Gamma$ such that there exist $x',y,y'$ with $d(x,x')\leq\theta\d$, $y\in\Gamma$, $d(y,y')\leq\theta\d$ and $d(x'y',z')\leq\theta\d$ is at least $\theta^6/8$. But that implies that the number of pairs $(x,y)\in\Gamma$ such that $d(xy,z)\leq 3\theta\d$ is at least $(\theta^6/8)|\Gamma|$. Replacing $\theta$ by $\theta/3$ gives the result as stated.
\end{proof}

Note that the power $\theta^6$ makes sense above. Since $\so3$ is three-dimensional, the probability that two random points will be within $\theta$ of points in $\Gamma$ should be around $\theta^3.\theta^3$, and we have shown that most of the time we are within a constant of what this random model would predict.

For the partially defined operation $\circ$ to satisfy the hypothesis of Theorem~\ref{assoc_main}, we need in particular that there should be many associative triples -- that is, triples $x,y,z$ such that both $(x\circ y)\circ z$ and $x\circ(y\circ z)$ are defined (in which case, as we have noted, they must be equal). This can also be deduced from Lemma \ref{qrlemma}, as we now show.

\begin{lemma} \label{smoothconvolution}
Let $\d,\theta>0$, let $\Gamma$ be a maximal $\d$-separated subset of $\so3$, let $U$ and $V$ be $\d$-separated subsets of $\so3$ and suppose that $U$ and $V$ have cardinalities at least $\a\d^{-3}$ and $\b\d^{-3}$, respectively. Then 
\[|\{(u,v)\in U\times V:d(uv,\Gamma)\leq 3\theta\d\}|\geq\a\b\theta^3\d^{-6}/64.\]
\end{lemma}

\begin{proof}
Write $U_\theta, V_\theta$ and $\Gamma_\theta$ for the $\theta\d$-expansions of $U, V$ and $\Gamma$. Then $\Gamma_\theta$ has density at least $\theta^3/2$, while $U_\theta$ and $V_\theta$ have densities at least $\a\theta^3/2$ and $\b\theta^3/2$, respectively. 

Let $\e=\a\b\theta^{15/2}/32$. By Lemma \ref{qrlemma}, there exists $\eta>0$ such that, writing $\mu_B$ for the characteristic measure of the ball of radius $\eta$ about the identity, we have that
\[\|U_\theta*V_\theta-U_\theta*V_\theta*\mu_B\|\leq\e,\]
where we have written $U_\theta$ and $V_\theta$ for the characteristic functions of the sets $U_\theta$ and $V_\theta$.

Since $|\Gamma_\theta|\geq\theta^3/2$, it follows that its characteristic function, which again we write $\Gamma_\theta$, has the property that $\|\Gamma_\theta\|_2\leq 2\theta^{3/2}$. Writing $x\approx_\eta y$ as an abbreviation for $|x-y|\leq\eta$, we therefore have
\[\langle U_\theta*V_\theta,\Gamma_\theta\rangle\approx_{2\e\theta^{3/2}}\langle U_\theta*V_\theta*\mu_B,\Gamma_\theta\rangle=\langle U_\theta*V_\theta,\Gamma_\theta*\mu_B\rangle.\]
Recall from the proof of Lemma \ref{denselydefined} that $\Gamma_\theta*\mu_B$ is bounded below by $\theta^3/2$ everywhere. It follows that the inner product on the right-hand side is at least $\a\b\theta^9/8$, and therefore that the inner product on the left-hand side is at least $\a\b\theta^9/8-2\e\theta^{3/2}$, which is at least $\a\b\theta^9/{16}$.

Now let $(u,v)\in U\times V$. If $d(uv,\Gamma)>3\theta$, then by the triangle inequality the product of the balls of radius $\theta\d$ about $u$ and $v$ does not intersect $\Gamma_\theta$, so the pair $(u,v)$ contributes nothing to the inner product. And otherwise, since the balls have volume at most $2\theta^3\d^3$ and the product of the balls intersects at most one ball of radius $\theta\d$ about a point of $\Gamma$, the contribution is at most $4\theta^6\d^6$. It follows that the number of pairs $(u,v)\in U\times V$ such that $d(uv,\Gamma)\leq 3\theta$ is at least $\a\b\theta^3\d^{-6}/64$, as claimed.
\end{proof}

\begin{corollary}\label{mosty}
Let $\Gamma$ be a maximal $\d$-separated subset of $\so3$. Then for at least half of the elements $y\in\Gamma$ there are at least $\theta^3\d^{-3}/128$ elements $x\in\Gamma$ such that $d(xy,\Gamma)\leq 3\theta\d$ and at least $\theta^3\d^{-3}/128$ elements $z\in\Gamma$ such that $d(yz,\Gamma)\leq 3\theta\d$.
\end{corollary}

\begin{proof}
Let $A$ be the set of all $y$ for which there are fewer than $\xi\d^{-3}$ elements $x\in\Gamma$ with $d(xy,\Gamma)\leq 3\theta$. Let $|A|=\a\d^{-3}$. Since $|\Gamma|\geq\d^{-3}/2$, we have by Lemma \ref{smoothconvolution} that $\xi\a\d^{-6}\geq\a\theta^3\d^{-6}/128$.

This appears to place no restriction on $\a$, but that is because the restriction is hidden. We need $\d$ to be small compared with a parameter $\eta$ in the previous lemma which depends on $\e$, which in turn depends on $\a,\b$ and $\theta$. However, for any fixed $\a,\theta$ we obtain the above result for sufficiently small $\d$. In particular, we can obtain it for $\a=1/4$ and deduce that $\xi\geq\theta^3/128$.

A similar argument proves that at most quarter of all $y$ fail the other property, and we are done.
\end{proof}

We are now ready to obtain many associative triples.

\begin{corollary}
Let $\Gamma$ be a maximal $\d$-separated subset of $\so3$. Then there are at least $\theta^9\d^{-9}/2^{22}$ triples $(x,y,z)\in\Gamma^3$ such that both $x\circ(y\circ z)$ and $(x\circ y)\circ z$ are defined.
\end{corollary}

\begin{proof}
Let $y$ satisfy the conclusion of Corollary \ref{mosty}. Now let $U=\{x:d(xy,\Gamma)\leq 3\theta\}$ and let $V=\{yz:d(yz,\Gamma)\leq 3\theta\}$. (The lack of symmetry between those two definitions is deliberate.) Then $U$ and $V$ satisfy the assumptions of Lemma \ref{smoothconvolution} with $\a=\b=\theta^3/128$. It follows that there are at least $\theta^9\d^{-6}/2^{20}$ elements $(x,yz)$ of $U\times V$ such that $d(xyz,\Gamma)\leq 3\theta\d$. But in that case, if we define the operation $\circ$ using the parameter $6\theta$ in place of $\theta$, then for each such pair $(x,yz)$ we have that $z\in V$, so $d(yz,\Gamma)\leq 3\theta$, so $y\circ z$ is defined, and then
\[d(x(y\circ z),\Gamma)\approx_{3\theta}d(xyz,\Gamma)\leq 3\theta,\]
so $x\circ(y\circ z)$ is also defined. 

Since $x\in V$, we have that $d(xy,\Gamma)\leq 3\theta$, so $x\circ y$ is defined, and finally
\[d((x\circ y)z,\Gamma)\approx_{3\theta}d(xyz,\Gamma)\leq 3\theta,\]
so $(x\circ y)\circ z$ is also defined and equal to $x\circ(y\circ z)$.

We can do this for at least $\d^{-3}/4$ elements $y$, so it follows that there are at least $\theta^9\d^{-9}/2^{22}$ associative triples, as claimed.
\end{proof}


\begin{thebibliography}{}
	
	\bibitem{ACK} 
	V. Artamonov, S. Chakrabarti, S. Kumar Pal,
	\emph{Characterizations of highly non-associative
		quasigroups and associative triples},
	Quasigroups and Related Systems \textbf{25}, pp. 1--19,
	2017.
	
	
	\bibitem{Bog}
	N. Bogolio\`uboff. 
	\emph{Sur quelques propri\'et\'es arithm\'etiques des presque-p\'eriodes.} 
	Ann. Chaire Phys.
	Math. Kiev, {\bf 4}:pp. 185--205, 
	1939.
	
	
	\bibitem{Brandt} 
	H. Brandt,
	\emph{Verallgemeinierung des Gruppenbegriffs},
	Math. Ann. \textbf{96}, pp. 360--366,
	1927.
	
	\bibitem{BGT} 
	E. Breuillard, B. J. Green and T. C. Tao,
	\emph{The structure of approximate groups},
	Publ. Math. IHES \textbf{116}, pp. 115-221,
	2012.
	
	
	\bibitem{CHPS}
	D. Conlon, H. H\`an, Y. Person, M. Schacht
	\textit{Weak quasi‐randomness for uniform hypergraphs}, 
	Random Struct. Algor., \textbf{40}, pp. 1--38,
	2011.
	
	
	\bibitem{deneskeedwell}
	J. Den\'es and A. D. Keedwell, 
	\emph{Latin Squares and their Applications},
	Akad\'emiai Kiad\'o, Budapest, Hungary, 1974.
	
	\bibitem{Drapal} 
	A. Drapal,
	\emph{On quasigroups rich in associative triples},
	Discrete Math \textbf{44}, no. 3, pp. 251--265,
	1983.
	
	\bibitem{DRC} 
	J. Fox, B. Sudakov,
	\emph{Dependent random choice},
	Random Structures and Algorithms \textbf{38}, pp. 68--99,
	2011.
	
	\bibitem{gowersszem}
	W. T. Gowers,
	\emph{A new proof of Szemer\'edi's theorem},
	GAFA, Geom. Funct. Anal. \textbf{11}, pp. 465-588,
	2001.
	
	\bibitem{quasirandomgps} 
	W. T. Gowers,
	\emph{Quasirandom groups},
	Combin. Probab. Comp. \textbf{17}, no. 3,  pp. 363--387,
	2011.
	
	\bibitem{greenso3}
	B. Green,
	\emph{On a conjecture of Gowers and Long},
	{B. Lond. Math. Soc.}, \emph{to appear}, 2020.
	
	\bibitem{hrushovski}
	E. Hrushovski, 
	\emph{Metrically approximate subgroups},
	preprint.\\
	\url{http://www.ma.huji.ac.il/~ehud/metric-stabilizer.pdf}
	
	\bibitem{KSS}
	P. Y. Kosmann-Schwarzbach, S. F. Singer, 
	\emph{Spherical Harmonics}, In: Groups and Symmetries. Universitext. Springer, New York, NY, 2010. 
	
	\bibitem{EL}
	E. Levi, 
	\emph{Symmetric abstract independence relations and the group
		configuration theorem},
	Masters thesis, private communication.
	
	\bibitem{taobp}
	T. Tao, 
	\emph{Metric entropy analogues of sum set theory},
	blog post.\\
	\url{https://terrytao.wordpress.com/2014/03/19/metric-entropy-analogues-of-sum-set-theory/}
	
	\bibitem{taopaper}
	T. Tao, 
	\emph{Product set estimates for non-commutative groups},
	Combinatorica,		
	{\bf 28}, 
	(5), 
	pp, 547–594,
	2008.
	
	\bibitem{taocohomology}
	T. Tao,
	\emph{1\% quasimorphisms and group cohomology},
	blog post.\\
	\url{https://terrytao.wordpress.com/2018/07/07/1-quasimorphisms-and-group-cohomology/}
	
	
	
\end{thebibliography}
\end{document}